\documentclass[11pt]{amsart}  

\usepackage[a4paper,hmargin=3.5cm,vmargin=4cm]{geometry}
\usepackage{amsfonts,amssymb,amscd,amstext}
\usepackage{graphicx}
\usepackage[dvips]{epsfig}
\usepackage[utf8]{inputenc}
\usepackage{hyperref}
\usepackage{verbatim}
\usepackage{subfigure} % To place figures in parallel

\usepackage{fancyhdr}
\input xy
\xyoption{all}

\usepackage{enumerate}
\usepackage{mathrsfs}

\usepackage[up,bf]{caption}
\newtheorem{theorem}{Theorem}[section]
\newtheorem{proposition}[theorem]{Proposition}

\newtheorem{lemma}[theorem]{Lemma}
\newtheorem{corollary}[theorem]{Corollary}

\theoremstyle{definition}
\newtheorem{definition}[theorem]{Definition}
\newtheorem{remark}[theorem]{Remark}

\newtheorem{example}[theorem]{Example}
\newtheorem{conjecture}[theorem]{Conjecture}

\numberwithin{equation}{section}
\numberwithin{figure}{section}

\newcommand\Acal{\mathcal{A}}

\newcommand\Lcal{\mathcal{L}}

\newcommand\Ocal{\mathcal{O}}
\newcommand\Pcal{\mathcal{P}}

\newcommand\Ucal{\mathcal{U}}

\newcommand\bu{\mathbf{u}}
\newcommand\bv{\mathbf{v}}

\newcommand\Ascr{\mathscr{A}}

\newcommand\Cscr{\mathscr{C}}

\newcommand\Gscr{\mathscr{G}}

\newcommand\B{\mathbb{B}}
\newcommand\C{\mathbb{C}}
\newcommand\D{\overline{\mathbb D}}
\newcommand\CP{\mathbb{CP}}
\renewcommand\D{\mathbb D}

\newcommand\N{\mathbb{N}}

\newcommand\R{\mathbb{R}}

\renewcommand\S{\mathbb{S}}

\newcommand\Z{\mathbb{Z}}

\renewcommand\c{\mathbb{C}}
\newcommand\cd{\overline{\mathbb D}}
\newcommand\cp{\mathbb{CP}}
\renewcommand\d{\mathbb D}

\newcommand\n{\mathbb{N}}
\renewcommand\r{\mathbb{R}}
\newcommand\s{\mathbb{S}}
\renewcommand\t{\mathbb{T}}
\newcommand\z{\mathbb{Z}}

\newcommand\igot{\mathfrak{i}}

\renewcommand\igot{\mathfrak{i}}

\newcommand\pgot{\mathfrak{p}}

\newcommand\Agot{\mathfrak{A}}
\newcommand\Igot{\mathfrak{I}}

\newcommand\Ygot{\mathfrak{Y}}

\renewcommand\imath{\igot}

\newcommand\hra{\hookrightarrow}
\newcommand\lra{\longrightarrow}

\newcommand\longhookrightarrow{\ensuremath{\lhook\joinrel\relbar\joinrel\rightarrow}}

\newcommand\wt{\widetilde}
\newcommand\wh{\widehat}
\newcommand\di{\partial}
\newcommand\dibar{\overline\partial}

\newcommand\dist{\mathrm{dist}}
\renewcommand\span{\mathrm{span}}
\newcommand\length{\mathrm{length}}
\newcommand\reg{\mathrm{reg}}

\newcommand\Flux{\mathrm{Flux}}
\newcommand\GCMI{\mathrm{GCMI}}
\newcommand\CMI{\mathrm{CMI}}
\newcommand\NC{\mathrm{NC}}

\newcommand\supp{\mathrm{supp}}

\newcommand\tr{\mathrm{tr}}
\newcommand\Hess{\mathrm{Hess}}
\newcommand\Id{\mathrm{Id}}

\def\dist{\mathrm{dist}}
\def\span{\mathrm{span}}
\def\length{\mathrm{length}}
\def\Flux{\mathrm{Flux}}
\def\reg{\mathrm{reg}}

\usepackage{color}

\begin{document}

%%
%% Headings
%%

\fancyhead[LO]{New complex analytic methods in the theory of minimal surfaces: a survey}
\fancyhead[RE]{A.\ Alarc\'on and F.\ Forstneri\v c}
\fancyhead[RO,LE]{\thepage}
\thispagestyle{empty}

%%
%% Title
%%

\vspace*{1cm}
\begin{center}
{\bf\LARGE  New complex analytic methods in the theory of minimal surfaces: a survey}

\vspace*{0.5cm}

%
% Authors
%
{\large\bf Antonio Alarc\'on\ and\  Franc Forstneri\v c}

\end{center}

%% Addresses and finantial support
%\footnote[0]{\vspace*{-0.4cm}
%}

%
% Abstract, keywords, and MSC
%
\vspace*{5mm}

\begin{quote}
{\small
\noindent {\bf Abstract}\hspace*{0.1cm}
In this paper we survey recent developments in the classical theory of minimal surfaces in 
Euclidean spaces which have been obtained as applications of both classical and modern 
complex analytic methods; in particular, Oka theory, period dominating holomorphic sprays, gluing methods for holomorphic maps, and the Riemann-Hilbert boundary value problem. 
Emphasis is on results pertaining to the global theory of minimal surfaces,  
in particular, the Calabi-Yau problem, constructions of properly immersed and embedded minimal surfaces in $\R^n$ and in minimally convex domains of $\R^n$, results on the complex Gauss map,
isotopies of conformal minimal immersions, and the analysis of the homotopy type of the 
space of all conformal minimal immersions from a given open Riemann surface.

\vspace*{1mm}

\noindent{\bf Keywords}\hspace*{0.1cm} Minimal surface, Riemann surface, Oka manifold.

\vspace*{1mm}

\noindent{\bf Mathematics Subject Classification (2010)}\hspace*{0.1cm} 
53A10, 32B15, 32E30, 32H02
}
\end{quote}

%\vspace*{5mm}

%\tableofcontents

%%%%%%%%%%
%%%%%%%%%%
%%%%%%%%%% Section: Introduction
%%%%%%%%%%
%%%%%%%%%%
%%%%%%%%%%

\setcounter{tocdepth}{1}
\tableofcontents

\section{\sc Introduction}
\label{sec:intro}

An immersed surface $M\to\r^n$ in the $n$-dimensional Euclidean space $(n\ge 3)$ is said to be 
a {\em minimal surface} if it is locally area minimizing, meaning that sufficiently small pieces of the surface 
have the smallest area among all surfaces with the same boundary. Such surfaces were first studied by
Euler in 1744 who showed that the only area minimizing surfaces of rotation are planes and catenoids.
The subject was taken up by Lagrange in 1760 who studied the area functional
and came up with the differential equation of minimal graphs in $\R^3$. 
It was discovered by Meusnier in 1776 that a surface is minimal if and only if its mean curvature 
vector field ${\mathbf H} \colon M\to\r^n$ vanishes identically. 
Plateau pointed out in 1873 that minimal surfaces appear naturally as soap films, and 
Douglas \cite{Douglas1932TAMS} and Rad{\'o} \cite{Rado1930AM} 
independently proved in 1932  that every Jordan curve in $\R^3$ spans a minimal surface.

The influence of complex analysis in the study of minimal surfaces was apparent already in the third quarter of 
the 19th century when Enneper and Weierstrass provided an analytic formula for representing any minimal surface in $\r^n$.
The so-called {\em Enneper-Weierstrass representation formula} \eqref{eq:EW} relies on the fact that for an isometric immersion 
$X=(X_1,\ldots,X_n)\colon M\to \r^n$ from a Riemannian surface $M$ (a smooth surface 
endowed with a Riemannian metric), the metric Laplacian of the immersion 
equals two times its mean curvature vector:
\[
	\Delta X=(\Delta X_1,\ldots, \Delta X_n)=2 \mathbf H. 
\]
Since the vanishing of the Laplacian depends only on the conformal class of the metric,
it follows that a conformal (angle preserving) immersion $X\colon M\to \r^n$ from 
a Riemann surface $M$ is minimal if and only if it is harmonic, $\Delta X=0$;
equivalently, the $(1,0)$-derivative $\di X=(\di X_1,\ldots,\di X_n)$ is a holomorphic $\C^n$-valued
$1$-form. Furthermore, $X$ is conformal if and only if $\di X$ satisfies the null equation
\begin{equation}\label{eq:null1}
	(\di X_1)^2 + (\di X_2)^2 + \cdots + (\di X_n)^2 = 0.
\end{equation}
(See Sect.\ \ref{ss:null}.)
This reduces the construction of oriented conformal minimal surfaces $M\to \R^n$ to the construction of
holomorphic maps from the Riemann surface $M$ to the subvariety $\Agot^{n-1}$ of $\C^n$ defined by the equation
$z_1^2+z_2^2+\cdots +z_n^2=0$, the {\em null quadric}  (see \eqref{eq:null}). 
This is the basis for the Enneper-Weierstrass formula; see Theorem \ref{th:EW} and \eqref{eq:EWR}. 
The analogous formula applies to {\em null holomorphic immersions} $Z=(Z_1,\ldots, Z_n)\colon M\to\C^n$, i.e., 
holomorphic immersions satisfying  
\[
	(dZ_1)^2 + (dZ_2)^2 + \cdots + (d Z_n)^2 = 0.
\]
Every conformal minimal immersion $M\to\R^n$ is locally (on any simply connected subset
of $M$) the real part of a null holomorphic immersion into $\C^n$; conversely, the real and the imaginary part of any
null holomorphic immersion $M\to\C^n$ are conformal minimal immersions $M\to\R^n$. See Sect.\ \ref{ss:flux} for 
more details.

In the mid-20th century Osserman \cite{Osserman1986book} renewed the interest in the theory of minimal surfaces by 
showing in particular that the Enneper-Weierstrass representation formula is very useful for the construction of complete minimal 
surfaces in $\r^3$ with finite total curvature. This was the true starting point for the study of the global theory of minimal surfaces 
by complex analytic methods. However, as late as in the 1980's the prevailing thought was that hyperbolic Riemann surfaces 
(i.e., those carrying nonconstant negative subharmonic functions) play only a marginal role in the global theory of minimal surfaces. 
This belief was partially refuted by the pioneering works of Jorge and Xavier \cite{JorgeXavier1980AM} from 1980, 
Nadirashvili \cite{Nadirashvili1996IM} from 1996, and Morales \cite{Morales2003GAFA} from 2003 which 
combined the Enneper-Weierstrass formula with the classical Runge approximation theorem for holomorphic functions.

Nevertheless, the true power and versatility of this approach was revealed only in the last few years
by bringing into the picture some of the more powerful complex analytic methods 
originating in Oka theory (which amounts to holomorphic approximation techniques
combined with a nonlinear version of the $\dibar$-problem),
and by adapting the classical Riemann-Hilbert boundary value problem
to the constructions of conformal minimal surfaces in $\R^n$ and holomorphic null curves in $\C^n$.
One of the key advantages of these stronger complex analytic methods over the classical ones 
is that they allow a more precise control of the placement of the whole surface in the space. 
This enabled the authors, often in collaboration with 
B.\ Drinovec Drnov\v sek and  F.\ J.\ L\'opez, to construct minimal surfaces with  interesting global properties 
and with the complete control of the conformal stucture on the surface. In other words, not only it is now possible 
to find minimal surfaces with interesting global properties and with prescribed topological type, 
one can also control their conformal (holomorphic) type, a major advance in the theory.

The goal of this article is to present these  recent developments in a way that is 
accessible not only to researchers, but also to graduate students in both the field of minimal
surfaces and in complex analysis. What transpires from our narrative is that these two 
fields are much more closely intertwined than believed up to now, with major influences going in both directions. 
On the one hand, complex analytic methods are a powerful
tool in the classical minimal surface  theory; on the other hand, 
many questions about minimal surfaces lead to analogous questions about complex curves.
Several lines of thought have been pursued separately by researchers in these two fields 
without having been aware of the analogies and synergies. Other problems have been
considered only in one field and overlooked in the other one, even though they are perfectly natural
and interesting in both fields.

Let us consider an example. It has been known since the early 1960's that every open Riemann surface 
embeds properly holomorphically into $\C^3$ (see \cite[Theorem 2.4.1]{Forstneric2017E}; this is a
special case of the Bishop-Narasimhan-Remmert embedding theorem for Stein manifolds).
However, the question whether every such surface embeds as a smooth closed complex curve in the 
complex Euclidean plane $\C^2$ remains one of the most difficult open problems of complex analysis, known as the 
{\em Forster Conjecture} \cite{Forster1970} or the {\em Bell-Narasimhan Conjecture} 
\cite{BellNarasimhan1990EMS}. On the minimal surfaces side, there is the equally natural problem 
of determining the smallest dimension $d\ge 3$ for which every open Riemann surface embeds as a 
proper conformal minimal surface in $\R^d$. Since every complex curve in $\C^n$ is also a minimal
surface in $\R^{2n}$ by Wirtinger \cite{Wirtinger1936MMP}, we have $d\le 6$. 
The following recent result says in particular that $d\le 5$. 

%
% THEOREM: PROPER CONFORMAL MINIMAL IMMERSIONS AND EMBEDDINGS TO Rn
%
\begin{theorem}\label{th:I-properRn}
Let $M$ be an open Riemann surface.
\begin{itemize}
\item[\rm (a)] 
There is a proper conformal minimal immersion $M\to \r^3$. Moreover, proper immersions are dense 
in the space of all conformal minimal immersions $M\to \r^3$ (in the compact-open topology).
\vspace{1mm}
\item[\rm (b)] 
There is a proper conformal minimal immersion $M\to\R^4$ with simple double points,
and such immersions are dense in the space of all conformal minimal immersions $M\to \r^4$. 
\vspace{1mm}
\item[\rm (c)]  
There is a proper conformal minimal embedding $M\hra\R^5$. Moreover, proper conformal minimal embeddings 
are dense in the space of all conformal minimal immersions $M\to \r^n$ for any $n\ge 5$.
\end{itemize}
\end{theorem}

Part (a) is due to Alarc\'on and L\'opez  \cite{AlarconLopez2012JDG,AlarconLopez2014TAMS},
while parts (b) and (c) were proved in 2016 by the authors and L\'opez \cite{AlarconForstnericLopez2016MZ}.
A more precise result in this direction is Theorem \ref{th:SSY} which provides conformal minimal immersions
and embeddings with a proper projection to a coordinate $2$-plane, thereby 
giving an optimal negative answer to both Schoen-Yau's and Sullivan's conjectures
(see Sect.\ \ref{ss:SullivanSchoenYau}).

It is known that only a few open Riemann surfaces embed as proper minimal surfaces in $\R^3$ (see e.g.\ 
\cite{LopezRos1991JDG,Collin1997AM,CollinKusnerMeeksRosenberg2004JDG,MeeksRosenberg2005AM,MeeksPerezRos2015AM}), 
so the smallest embedding dimension for open Riemann surfaces as minimal surfaces satisfies $d\ge 4$. 
This leaves us with the question whether $d=4$ or $d=5$. 
An affirmative answer to the Forster-Bell-Narasimhan Conjecture would imply $d=4$.
In the last decade, powerful new methods for constructing proper holomorphic embeddings of open 
Riemann surfaces into $\C^2$ have been developed by Wold and Forstneri\v c, using the technique of 
exposing boundary points and pushing the boundary of the surface in $\C^2$
to infinity by holomorphic automorphisms; see the recent survey in \cite[Chap.\ 9]{Forstneric2017E}.  
For example, every circular domain in $\C$ embeds properly holomorphically into $\C^2$ \cite{ForstnericWold2013APDE}.
These results, along with the absence of any conceptual obstructions,  speak in favor of the Forster-Bell-Narasimhan Conjecture.
Since minimal surfaces are more abundant than complex curves, the following conjecture has an even 
better chance.

%
%   THE ALARCON-FORSTNERIC CONJECTURE 
%
\begin{conjecture} \label{conj:AF}
Every open Riemann surface admits a proper conformal minimal embedding into $\R^4$.
\end{conjecture}

On the other hand, every open Riemann surface which is known to properly embed as a
conformal minimal surface in $\r^4$ is also known to properly embed as a complex curve in $\c^2$,
the main reason being that no automorphisms of $\R^n$ other than the rigid motions preserve
map minimal surfaces to minimal surfaces.

%
%  THE CALABI-YAU PROBLEM
%
Another example where the analogies between the fields of complex analysis and minimal surfaces
become even more apparent is the problem of constructing complete bounded minimal surfaces in $\R^n$ 
(the Calabi-Yau problem) and complete bounded complex submanifolds in $\C^n$ (Yang's problem). 
We describe this topic briefly, referring to Sect.\ \ref{ss:CY} for a more complete presentation.

Recall that an immersion $X\colon M\to \R^n$ from an open manifold $M$ is said to be {\em complete} if the 
pullback $g=X^*(ds^2)$ of the Euclidean metric on $\R^n$ by the immersion is a complete Riemannian 
metric on $M$; equivalently, given any divergent path $\gamma\colon [0,1)\to M$
(meaning that the point $\gamma(t)$ leaves every compact subset of $M$ as $t$ approaches $1$),
the path $t\mapsto X(\gamma(t))\in \R^n$ has infinite Euclidean length in $\R^n$.

The  {\em Calabi-Yau problem for minimal hypersurfaces} asks whether 
there exist complete bounded minimal hypersurfaces in $\R^n$ for $n\ge 3$. Calabi conjectured
that such hypersurfaces do not exist (see \cite[p.\ 170]{Calabi1965Conjecture}). 
The first counterexample  was given
in 1996 by Nadirashvili \cite{Nadirashvili1996IM} who constructed a complete bounded immersed
minimal disc in $\R^3$. A plethora of results  followed extending Nadirashvili's construction 
to more general surfaces; see Sect.\ \ref{ss:CY}. However, with the techniques available at that time 
it was impossible to control the conformal structure or the boundary behavior of the examples. 
The following considerably more precise result  was proved 
in 2015 by the authors together with Drinovec Drnov\v sek and  
L\'opez \cite[Theorem 1.1]{AlarconDrinovecForstnericLopez2015PLMS}.

%
%  OUR CALABI-YAU THEOREM, JORDAN BOUNDARIES
%
\begin{theorem} \label{th:I-complete}
For every compact bordered Riemann surface $M$ and integer $n\ge 3$ there is a continuous map 
$X\colon M\to\R^n$ whose restriction to the interior $\mathring M=M\setminus bM$ is a complete 
conformal minimal immersion $X\colon \mathring M\to \R^n$.
If $n\ge 5$ then $X\colon M\to\R^n$ can be chosen a topological embedding. 
\end{theorem}

If $X$ is as in the theorem, then $X(M)$ is a complete minimal surface in $\R^n$ bounded by finitely many 
Jordan curves, and we have a control of its conformal structure. 

One of the tools that made this construction possible is the adaptation 
of the Riemann-Hilbert boundary value problem to null holomorphic curves and conformal minimal immersions
(see Sects. \ref{ss:RHcomplex} and \ref{ss:RH}). 
This topic was started by the authors  \cite{AlarconForstneric2015MA} in dimension $n=3$.
A more general result in any dimension $n\ge 3$ was obtained by the authors together with Drinovec Drnov{\v s}ek and L\'opez 
\cite{AlarconDrinovecForstnericLopez2015PLMS}. We showed  
that one can increase the intrinsic boundary distance in the Riemann
surface $M$ by an arbitrarily big amount by changing the conformal minimal 
immersion $M\to \R^n$ as little as desired in the $\Cscr^0$-norm (see Lemma \ref{lem:Jordan}).
This is achieved by applying the Riemann-Hilbert method in a certain spiralling construction, 
somewhat resembling Nash's method \cite{Nash1954}  of constructing $\Cscr^1$ isometric immersions of 
Riemannian manifolds into Euclidean spaces. 
The same technique applies to complex curves in $\C^n$ for any $n\ge 2$, to
holomorphic null curves in $\C^n$ for any $n\ge 3$, and to holomorphic Legendrian curves 
in complex contact manifolds,  thereby yielding an analogue of Theorem \ref{th:I-complete} in that setting 
(see \cite{AlarconForstneric2017IMRN,AlarconForstnericLopez2017CM}).

On the other hand, Calabi's original conjecture holds for embedded minimal surfaces of finite topology in $\R^3$ since these are 
necessarily proper in $\R^3$ according to Colding and Minicozzi \cite[Corollary 0.13]{ColdingMinicozzi2008AM};
their result was extended to surfaces of finite genus and countably many ends by Meeks, P\'erez, and Ros \cite{MeeksPerezRos-CY}. 
Nothing seems known about Calabi's conjecture in dimensions $n>3$.

The analogous problem in complex analysis was raised in 1977 by 
Yang \cite{Yang1977} who asked whether there exist complete bounded complex submanifolds 
in $\C^n$. The first such examples were found in 1979 by Jones \cite{Jones1979PAMS}
who showed that the unit disc $\D=\{z\in\C:|z|<1\}$ admits a complete bounded holomorphic immersion
into $\C^2$, embedding into $\C^3$, and proper embedding into the ball of $\c^4$.
Interest in this subject was revived only recently, due mainly to the influence
of the developments in minimal surface theory.  In 2013 the authors showed in \cite{AlarconForstneric2013MA}
that every bordered Riemann surface admits a complete proper holomorphic immersion
into the ball of $\C^2$ and embedding into the ball of $\C^3$. 
A flurry of recent activity followed and we refer to \cite[Sect.\ 4.18]{Forstneric2017E} for a survey. 
Contrary to the case of minimal hypersurfaces in $\r^n$ for $n>3$, 
complete bounded complex hypersurfaces in $\c^n$ exist in arbitrary dimension $n\ge 2$; 
see Alarc\'on and L\'opez \cite{AlarconLopez2016JEMS} for $n=2$ and Globevnik \cite{Globevnik2015AM} for any $n$.
Indeed, Globevnik showed in \cite{Globevnik2015AM}  that the ball $\B^n$ of $\C^n$ can be foliated by complete closed 
complex hypersurfaces, 
and subsequently Alarc\'on \cite{Alarcon2018} proved that every smooth complete 
complex hypersurface in $\B^n$ can be embedded into a {\em nonsingular} holomorphic foliation
of $\B^n$ all of whose leaves are complete.

Another classical topic of minimal surface theory is to understand which 
Riemann surfaces properly immerse or embed as conformal minimal surfaces in a given
domain $\Omega\subset \R^n$. The case $\Omega=\R^n$ is covered by Theorem \ref{th:I-properRn}.
When considering minimal surfaces in proper domains $\Omega \subsetneq \R^n$, especially bounded ones,
one must restrict attention to surfaces of hyperbolic conformal type.
Classically this problem was studied for {\em convex} domains (see \cite{MartinMorales2005DMJ}).
The authors, jointly with Drinovec Drnov\v sek and L\'opez, proved in 
\cite{AlarconDrinovecForstnericLopez2015PLMS} that every bordered Riemann
surface admits a complete proper conformal minimal immersion into any convex domain
$\Omega\subset\R^n$, $n\ge 3$, which can be chosen an embedding if $n\ge 5$
and an immersion with simple double points if $n=4$. The Riemann-Hilbert boundary
value problem for conformal minimal immersions (see Theorem \ref{th:RHCMI}) plays
a major role in our proof. It provides an inductive construction by which all of the surface is kept inside $\Omega$ 
at every step, pushing its boundary closer and closer to $b\Omega$.

In the subsequent work \cite{AlarconDrinovecForstnericLopez2018TAMS} of the same authors
this result was extended to the substantially bigger class of all {\em minimally convex}
(also called {\em $2$-convex}) domains. A domain $\Omega\subset \R^n$ is minimally convex if it admits a smooth exhaustion
function $\rho\colon \Omega\to\R_+$ such that the smallest two eigenvalues 
$\lambda_1(x),\lambda_2(x)$ of its Hessian $H_\rho(x)=\left(\frac{\di^2 \rho(x)}{\di x_j\di x_k}\right)$ at any 
point $x\in\Omega$ satisfy $\lambda_1(x)+\lambda_2(x)>0$. If $\Omega$ is smoothly bounded and
$\nu_1(x)\le \nu_2(x)\le\cdots\le \nu_{n-1}(x)$ are the principal curvatures
of the boundary $b\Omega$ at the point $x\in b\Omega$ from the inner side,
then $\Omega$ is minimally convex if and only if $\nu_1(x) + \nu_2(x)\ge 0$ holds
for every $x\in b\Omega$. In particular, if $n=3$ and $S$ is a properly embedded 
minimal surface in $\R^3$ (in which case $\nu_1(x) + \nu_2(x)=0$ holds identically on $S$), 
then each connected component of $\R^3\setminus S$ is a 2-convex domain. 
The following main result in this direction is \cite[Theorem 1.1]{AlarconDrinovecForstnericLopez2018TAMS}
(see Sect.\ \ref{ss:proper}).

%
%  THEOREM: PROPER CMIs IN MINIMALLY CONVEX DOMAINS
%
\begin{theorem}\label{th:I-minimally_convex}
Assume that $\Omega$ is a minimally convex domain in $\R^n$ $(n\ge 3)$, $M$ is a compact 
bordered Riemann surface, and $X\colon M\to\Omega$ is a conformal minimal immersion.
Then $X$ can be approximated uniformly on compacts in $\mathring M=M\setminus bM$
by proper (and complete if so desired) conformal minimal immersions $\mathring M\to \Omega$.
\end{theorem}

Examples in  \cite{AlarconDrinovecForstnericLopez2018TAMS}
show that, in dimension $n=3$, minimally convex domains form the biggest class of domains for which 
one can expect general approximation results for conformal minimal immersions by proper ones.
This line of results is intimately related to the construction,
due to Drinovec Drnov\v sek and Forstneri\v c \cite{DrinovecForstneric2007DMJ},
of proper holomorphic maps from bordered Riemann surfaces into any complex manifold $\Omega$
admitting an exhaustion function whose Levi form has at least 
two positive eigenvalues at every point (hence into any Stein manifold of dimension $>1$). 
Analysis of the proof in \cite{AlarconDrinovecForstnericLopez2018TAMS} reveals
a deeper reason behind this connection. 

Another interesting and important object in the theory is the {\em complex Gauss map}
of a conformal minimal surface $X\colon M\to\R^n$. This is the Kodaira-type holomorphic map
$G_X\colon M\to \CP^{n-1}$ defined by 
\[ 
	G_X(p) = [\di X_1(p) \colon \cdots \colon \di X_n(p)]\in \CP^{n-1}, \quad p\in M.
\]
In view of \eqref{eq:null1} the map $G_X$ assumes values in the complex hyperquadric
\[ 
     Q_{n-2} = \bigl\{[z_1:\ldots : z_n]\in\CP^{n-1} : z_1^2+ \cdots + z_n^2 = 0\bigr\}. 
\]   
This map is especially interesting in dimension $n=3$. In this case, the quadric $Q_{1}$ is the image of a 
quadratically embedded rational curve $\CP^1\hra \CP^2$, and hence we may consider 
$G_X$ as a holomorphic map $g_X\colon M\to \CP^1$, i.e.,  a meromorphic function on $M$.
The complex Gauss map $g_X$ of a minimal surface in $\r^3$ provides crucial information about its geometry. 
Several important properties of the surface depend only on its Gauss map, in particular, 
the Gauss curvature and the Jacobi operator 
(see e.g.\ \cite{MeeksPerez2004SDG,MeeksPerez2012AMS,Osserman1980DG,Osserman1986book}). 
Furthermore, it was shown by Barbosa and do Carmo \cite[Theorem 1.2]{BarbosaDoCarmo1976AJM}
that the minimal surface $X(M)\subset\R^3$ is stable if the spherical image 
$g_X(M)\subset \CP^1$ of $X(M)$ has area less than $2\pi$;  
this holds for example if $g_X(M)$ is a proper subset of the unit disc $\d\subset\c$.
Therefore, it is interesting to know the following recent result 
\cite[Corollary 1.2]{AlarconForstnericLopez2017JGEA} of the authors with L{\'o}pez.

%
%   THE MAIN THEOREM ON THE GAUSS MAP
%
\begin{theorem}\label{th:I-Gauss}
Every meromorphic function on an open Riemann
surface $M$ is the complex Gauss map of a conformal minimal immersion $X\colon M\to \R^3$.
Furthermore, $X$ can be chosen as the real part of a holomorphic null curve $Z=X+\imath Y \colon M\to\C^3$.
\end{theorem}

The analogous result holds in higher dimensions, see \cite[Theorem 1.1]{AlarconForstnericLopez2017JGEA}.
We refer to Sections \ref{ss:Gauss} and  \ref{ss:OnGauss} for more on this topic.

It is a natural problem to understand the homotopy type of the space $\CMI(M,\R^n)$ 
of all conformal minimal immersions of a given open Riemann surface $M$ to $\R^n$,
endowed with the compact-open topology. A conformal minimal immersion $M\to\R^n$ is 
called {\em nonflat} if its image is not contained in any affine plane;  the space of all such
immersions is denoted by $\CMI_{\mathrm{nf}}(M,\R^n)$. Similarly, $\NC_{\mathrm{nf}}(M,\C^n)$
denotes the space of nonflat holomorphic null immersions $M\to\C^n$, and 
\[
	\Re \NC_{\mathrm{nf}}(M,\C^n) =\{\Re Z\colon M\to\R^n : Z\in \NC_{\mathrm{nf}}(M,\C^n)\}
	\subset \CMI_{\mathrm{nf}}(M,\R^n). 
\]
A continuous map $\phi\colon X\to Y$ of topological spaces is a {\em weak homotopy equivalence}
if it induces a bijection of the path components of the two spaces and an isomorphism
$\pi_k(X,x_0) \stackrel{\cong}{\longrightarrow} \pi_k(Y,\phi(x_0))$ of their homotopy groups for each $k\in\N$ and $x_0\in X$. 
The map is a {\em homotopy equivalence} if there is a continuous map $\psi\colon Y\to X$ such that 
$\psi\circ\phi\colon X\to X$ and $\phi\circ\psi\colon Y\to Y$ are homotopic to the identity on the respective spaces.
Forstneri\v c and L\'arusson \cite{ForstnericLarussonCAG} proved the following result
(see Sect.\ \ref{ss:rough}).

%
%  THE THEOREM ON ROUGH SHAPE
%
\begin{theorem}\label{th:I-shape}
Let $M$ be an open Riemann surface, and let  $\theta$ be a holomorphic $1$-form without zeros on $M$. The map 
\[ 
	\CMI_{\mathrm{nf}}(M,\R^n) \longrightarrow \Cscr(M,\Agot^{n-1}_*),\qquad X\longmapsto \di X/\theta,
\]
is a weak homotopy equivalence, and is a homotopy equivalence if $M$ has finite topological type 
(i.e., finite genus and number of ends). Likewise, the inclusion 
\[
	 \Re(\NC_{\mathrm{nf}}(M,\C^n)) \longhookrightarrow  \CMI_{\mathrm{nf}}(M,\R^n) 
\] 
is a weak homotopy equivalence, and is a homotopy equivalence (indeed, the inclusion of
a strong deformation retract) if $M$ has finite topological type.
\end{theorem}

Besides complex analysis, the proof of  Theorem \ref{th:I-shape} strongly relies on
Gromov's {\em convex integration method} which originates in his paper \cite{Gromov1973IZV}
and has been fully developed in his monograph \cite{Gromov1986} (see also Spring \cite{Spring2010}). 
In the case at hand, this technique provides families of loops with specified integrals in the null quadric.

The main interest of Theorem \ref{th:I-shape} lies in the fact that the space $\Cscr(M,\Agot^{n-1}_*)$ is 
quite easy to understand. When $M$ has finite topology, the second part of 
Theorem \ref{th:I-shape} may be interpreted as follows.

{\em We can simultaneously continuously deform all nonflat conformal minimal immersions $M\to \R^n$
to those with vanishing flux, keeping the latter ones fixed.}

That a single conformal minimal immersion can be deformed to one with zero flux 
was first shown by Alarc\'on and Forstneri\v c in \cite{AlarconForstneric2017CRELLE}.
It was later proved in \cite[Corollary 1.6]{AlarconForstnericLopez2017JGEA} that 
such a deformation exists through a family of conformal minimal immersions 
with the same complex Gauss map.

In Theorem \ref{th:I-shape} we have excluded flat conformal minimal immersions and 
null curves; these present technical difficulties in the analysis of the 
structure of the respective mapping spaces. When $M$ is a compact 
bordered Riemann surface, the space $\CMI^r_{\mathrm{nf}}(M,\R^n)$ of all nonflat 
conformal minimal immersions $M\to\R^n$ of class $\Cscr^r$ $(r\in \N)$ is 
a real analytic Banach manifold (see Theorem \ref{th:structure}), while flat immersions 
seem to be singular points of $\CMI^r(M,\R^n)$.
Nevertheless, in \cite[Theorem 7.1]{AlarconForstnericLopez2017JGEA} the connected components
of $\CMI(M,\R^n)$ were identified  as follows.

%
%   THEOREM ON PATH COMPONENTS
%
\begin{theorem}\label{th:I-pathcomponents}
Let $M$ be a connected open Riemann surface. The inclusion 
$\CMI_{\mathrm{nf}}(M,\R^n) \hra \CMI(M,\R^n)$ of the space of all nonflat conformal minimal immersions 
$M\to\R^n$ into the space of all conformal minimal immersions induces a bijection of path components 
of the two spaces. In particular, the set of path components of $\CMI(M,\R^3)$ is in bijective correspondence 
with elements of the abelian group $(\Z_2)^l$ where $H_1(M;\Z)=\Z^l$, and $\CMI(M,\R^n)$ is path connected if $n>3$.
\end{theorem}

It was shown by the authors and L\'opez \cite{AlarconForstnericLopezMAMS,AlarconLopez2015GT}
that complex analytic methods may also be used in the construction of 
{\em non-orientable minimal surfaces} in $\R^n$ by working on their oriented double-sheeted coverings. 
In \cite[Example 6.1]{AlarconForstnericLopezMAMS} the reader can find the first known example 
of a properly embedded minimal M\"obius strip in $\R^4$ (see also Sect.\ \ref{ss:minimal}). 
Space does not permit us to include these results.
Another recently developed topic that is not treated in this survey but relies on complex analytic tools is the theory of 
uniform approximation by complete minimal surfaces in $\r^3$ with finite total curvature, due to 
L\'opez \cite{Lopez2014JGA,Lopez2014TAMS}.

There are several other important aspects of the classical theory of minimal surfaces 
in Euclidean spaces where substantial progress has been made in recent years but are not covered in this paper; 
see in particular the survey by P{\'e}rez \cite{Perez2017} and the monographs by
Meeks and P\'erez \cite{MeeksPerez2004SDG,MeeksPerez2012AMS}.

The organization of the paper is evident from the table of contents. We include proofs of the main complex analytic 
results used in the constructions; an exception is the Riemann-Hilbert
boundary value problem for conformal minimal surfaces and null curves 
(see Sect.\ \ref{ss:RH}) whose proofs are too complex to be included. 
The inductive procedures leading to the proofs of the main results are for the most part only sketched, referring 
to the original sources for further details.

%%%%%%%%%%
%%%%%%%%%%
%%%%%%%%%% Section: From minimal surfaces to complex analysis and back
%%%%%%%%%%
%%%%%%%%%%
%%%%%%%%%%

\section{\sc From minimal surfaces to complex analysis and back}
\label{sec:null}

In this section we briefly review those classical facts relating minimal surfaces to complex analysis 
which are indispensable for the subsequent discussion. More complete presentations are available
in the books of Osserman \cite{Osserman1986book}, 
Colding and Minicozzi \cite{ColdingMinicozzi1999CLNM,ColdingMinicozzi2011AMS}, 
and several others. For geometry of surfaces we refer to do Carmo \cite{doCarmo1976book}, 
and for the theory of Riemann surfaces we refer to the monographs by Donaldson \cite{Donaldson2011book},
Farkas and Kra \cite{FarkasKra1992}, and Forster \cite{Forster1981book}.

Let $\R^n$ and $\C^n$ denote the real and the complex 
Euclidean space of dimension $n\in\N=\{1,2,3,\ldots\}$, respectively. 
We also write $\R_\pm=\{x\in \R: \pm x\ge 0\}$, $\C^n_*=\C^n\setminus \{0\}$, and $\C_*=\C^1_*$.
We denote the coordinates on $\R^n$ by $(x_1,\ldots,x_n)$ and those on $\C^n$ by $z=(z_1,\ldots,z_n)$,
where $z=x+\imath y$ with $x,y\in\R^n$ and $\imath=\sqrt{-1}$.
Maps to these spaces will be denoted by the corresponding capital letters, e.g.\ $X\colon M\to\R^n$ and
$Z\colon M\to \C^n$.  We denote the Euclidean inner product and norm on $\R^n$ by 
\[
	x\,\cdotp y=\sum_{j=1}^n x_j y_j, \qquad |x|^2=x\,\cdotp x = \sum_{j=1}^n x_j^2.
\]
The space $\R^n$ is endowed with the Riemannian metric $ds^2=dx_1^2+\cdots+dx_n^2$.

%
%    RIEMANN SURFACES
%
\subsection{Riemann surfaces}\label{ss:Riemann}

A {\em Riemann surface} is a one dimensional complex manifold. 
We denote the space of all holomorphic functions on a Riemann surface $M$ by $\Ocal(M)$.
Since every minimal surface in $\R^n$ is parametrized by a nonconstant harmonic map 
from a Riemann surface, such a surface cannot be compact and without boundary.
Hence we shall mainly consider Riemann surfaces that are either {\em open}
(i.e., non-compact and without boundary) or compact with nonempty boundary.
On an open Riemann surface $M$ we have the classical 
Runge-Mergelyan approximation and Weierstrass interpolation theorems. The former says that, 
given a {\em Runge} (also called {\em $\Ocal(M)$-convex}) compact subset $K\subset M$ 
(i.e., such that $M\setminus K$ has no relatively compact components), every continuous function $K\to\c$
that is holomorphic on the interior $\mathring K$ of $K$ may be approximated uniformly on $K$ by functions in $\Ocal(M)$. 
The latter says that every map $\Lambda\to\c$ on a closed discrete subset $\Lambda\subset M$ extends to a function 
in $\Ocal(M)$. (Weierstrass's original theorem for planar domains \cite{Weierstrass1986} was extended 
to open Riemann surfaces by Florack \cite{Florack1948SMIUM} in 1948.)
An open Riemann surface is the same thing as a $1$-dimensional Stein manifold
(see Sect.\ \ref{ss:Stein}), and the aforementioned results extend to any Stein manifold as the  
Oka-Weil approximation theorem and Oka-Cartan extension theorem, respectively.

The following classification of open Riemann surfaces has important implications in the theory of minimal surfaces;
see Farkas and Kra \cite[p.\ 179]{FarkasKra1992}.

\begin{definition}\label{def:hyperbolic}
An open Riemann surface is said to be {\em hyperbolic} if it carries nonconstant negative subharmonic functions; 
otherwise it is said to be {\em parabolic}.
\end{definition}

By Koebe's uniformization theorem, the only simply connected open Riemann surfaces up to a biholomorphism 
are $\c$ which is parabolic by Liouville's theorem, and the unit disc $\d=\{\zeta\in\c\colon |\zeta|<1\}$ which is hyperbolic. 
Every compact Riemann surface from which finitely many points have been removed is parabolic. 
We refer to the survey by Grigor'yan \cite{Grigoryan1999BAMS} for further information on parabolicity and 
hyperbolicity of Riemannian manifolds.

A {\em compact bordered Riemann surface} is a compact Riemann surface $M$ with boundary 
$bM\neq\varnothing$ consisting of finitely many pairwise disjoint Jordan curves. The interior $\mathring M=M\setminus bM$ 
of such $M$ is a {\em bordered Riemann surface} and is a hyperbolic open Riemann surface. 
Such $\mathring M$ is biholomorphic to a smoothly bounded domain in a compact Riemann surface without boundary.

The only topological invariants of a connected oriented surface $M$ are its genus and number of ends. 
We say that $M$ has {\em finite topological type} if both its genus $g$ and the number $m$ of its ends are finite;
such $M$ is biholomorphic to a domain in a compact Riemann surface $R$ 
from which finitely many points and closed discs have been removed (see Stout \cite{Stout1965TAMS}).
Its first homology group equals $H_1(M;\Z)\cong \Z^l$ where $l=2g+m-1$.
There exist smooth Jordan curves $C_1,\ldots, C_l$ in $M$ representing a basis
of $H_1(M;\Z)$. If $M$ is either open or compact with nonempty boundary,
then these curves can be chosen such that their union $C=\bigcup_{j=1}^l C_j$ is contained
in $\mathring M$ and is Runge in $\mathring M$; furthermore, $C$ is a strong deformation retract of $M$.

%
%   IMMERSED SURFACES, RIEMANNIAN METRICS, ISOTHERMAL COORDINATES
%
\subsection{Immersed surfaces, Riemannian metrics, and isothermal coordinates}
\label{ss:isothermal}

Let $S$ be a smooth real surface and $X=(X_1,\ldots,X_n)\colon S\to \R^n$ be a smooth immersion.
The pullback of the Euclidean metric $ds^2$ on $\R^n$ by the immersion $X$ is a Riemannian metric on $S$:
\[
	g=X^*(ds^2) = (dX_1)^2+\cdots + (dX_n)^2.
\]
In any smooth local coordinate $(u,v)$ on $S$ we have  that
\begin{equation}\label{eq:gABC}
	g= Adu^2 + 2B du dv + Cdv^2,
\end{equation}
where $A>0,B,C>0$ are smooth functions and $AC-B^2>0$; this is called the {\em first fundamental form} of the immersed surface. 
If \eqref{eq:gABC} holds in coordinates $(u,v)$ on a domain $\Omega\subset S$, 
then the area of the immersed surface $X \colon \Omega \to\R^n$ equals
\begin{equation}\label{eq:area}
	\mathrm{Area}\,(X(\Omega)) = \int_{\Omega} \sqrt{AC-B^2} \,  du dv. 
\end{equation}
The $2$-form
\begin{equation}\label{eq:areael}
	dA_{X(\Omega)}=\sqrt{AC-B^2} \,  du dv
\end{equation}
is called the {\em area element} of the immersed surface $X|_{\Omega}\colon \Omega \to\r^n$.

By the celebrated {\em isometric immersion theorem} of Nash \cite{Nash1954,Nash1956},
every smooth Riemannian metric on $S$ is induced by a smooth embedding
$X\colon S\to\R^n$ to a Euclidean space. See Gromov \cite{Gromov1986,Gromov2017} for more information.

A Riemannian metric $g$ on an orientable surface $S$ 
determines an almost complex structure operator $J\colon TS\to TS$, with $J^2=-\Id$, 
by the condition that for every unit vector $\xi\in T_p S$ the pair $(\xi,J\xi)$ 
is a positively oriented $g$-orthonormal basis of the tangent space $T_p S$.
Riemannian metrics $g,\tilde g$ on $S$ are said to be {\em conformally equivalent}
if $\tilde g=\lambda g$ for some function $\lambda>0$; such metrics
determine the same almost complex structure $J$. Conversely, a choice of $J$ uniquely
determines a conformal class of Riemannian metrics. Around any point of $S$ there exist smooth 
{\em isothermal coordinates} $(u,v)$ in which the Riemannian metric $g$ has the form 
\begin{equation}\label{eq:metricisothermal}
	g = \lambda (du^2 + dv^2)
\end{equation}
for some positive function $\lambda>0$. A change of coordinates which puts the metric in this form
is found by solving the {\em Beltrami equation} (see Ahlfors \cite{Ahlfors2006}). 
In such coordinates, the associated almost complex structure $J$ 
is the {\em standard almost complex structure} on $\R^2_{(u,v)}\cong \C$
given by $J_{st} \, \frac{\di}{\di u}  = \frac{\di}{\di v}$, $J_{st}\, \frac{\di}{\di v}   = -\frac{\di}{\di u}$.
The transition map between any two isothermal coordinates 
is a conformal isomorphism, hence holomorphic or antiholomorphic with respect
to the complex coordinate $\zeta=u+\imath v$.
If $S$ is orientable, we obtain an  atlas $\Ucal=\{(U_j,\phi_j)\}$ consisting 
of an open covering $\{U_j\}$ of $S$ and positively oriented isothermal coordinates 
$\phi_j\colon U_j \to \phi_j(U_j) \subset \R^2\cong \C$ whose transition maps $\phi_{i,j}=\phi_i\circ\phi_j^{-1}$
are biholomorphisms; that is, $\Ucal$ is a complex atlas determining on $S$ the structure of a Riemann surface. 
If $S$ is connected and non-orientable, it admits a double-sheeted covering map $\pi\colon \wt S\to S$
with $\wt S$ orientable, and the same argument applied to the metric $\pi^*g$ on $\wt S$ 
shows that $\wt S$ carries the structure of a Riemann surface such that the projection map $\pi$ is conformal.

%
%    MINIMAL SURFACES
%
\subsection{Minimal surfaces}\label{ss:minimal}
Assume that $S$ is a smooth orientable surface and $X\colon S\to \R^n$ is a smooth immersion. 
Let $N\colon S\to\r^n$ be a smooth vector field along $X$ such that  for every $p\in S$ the vector 
$N(p)$ has unit length and is orthogonal to the tangent plane $dX_p(T_p S)\subset\R^n$ of $X$. 
Given a smooth function $\psi\colon S\to\r$ with compact support, there is a number $\epsilon>0$ such that the maps
\[
	X^t=X+t\psi N : S\to\r^n,\qquad t\in (-\epsilon,\epsilon),
\]
are again smooth immersions. Such a family of immersions is called a {\em normal variation with compact support} of $X=X^0$. 
The associated area functional is
\begin{equation}\label{eq:areafunctional}
	\Acal(t)={\rm Area}\, (X^t(S)),\qquad t\in(-\epsilon,\epsilon),
\end{equation}
and the {\em first variation of area formula} says that
\[
	\Acal'(0)=\frac{d\Acal(t)}{dt}\Big|_{t=0}=-2\int_S {\bf H}\cdot N\, \psi\, dA_X,
\]
where ${\bf H}$ and $dA_X$ are the mean curvature vector field and the area element  \eqref{eq:areael} of $X$,
respectively. This leads to the following observation due to Meusnier.

\begin{proposition}\label{pro:1vf}
Let $S$ be a smooth orientable surface and $X\colon S\to \R^n$ be a smooth immersion. 
The following two conditions are equivalent:
\begin{itemize}
\item $X$ is a critical point of the area functional for all normal variations with compact support.
\item The mean curvature vector field ${\bf H}\colon S\to \r^n$ of $X$ vanishes identically.
\end{itemize}
\end{proposition}

An immersed surface $X\colon S\to\r^n$ is said to be {\em minimal} if it satisfies the equivalent conditions in 
Proposition \ref{pro:1vf}. It follows from the {\em second variation of area formula} 
(see \cite[p.\ 95]{Nitsche1989} or \cite[p.\ 83--84]{DierkesHildebrandtKusterWohlrab1992-I}) 
that every minimal surface in $\r^n$ minimizes area locally.
Surfaces which minimize area globally are said to be {\em area minimizing}. 
Furthermore, for a minimal surface $X\colon S\to\r^n$ we have 
that $\Acal''(0)> 0$ for all normal variations with compact support if and only if every such variation of $X$ strictly increases the area; 
if this holds then $X$ is said to be {\em strongly stable}. If, on the contrary, for some variation the second derivative is negative, 
$\Acal''(0)<0$, then there are nearby surfaces with smaller area and  $X$ is then called {\em unstable}. Finally, those 
minimal surfaces for which $\Acal''(0)\ge 0$ holds for all normal variations with compact support are said to be {\em stable}. 
The stability property has important implications in the theory of minimal surfaces.

The simplest example of a minimal surface in $\r^n$ is an affine plane, which is in fact area minimizing. 
A classical result by Wirtinger \cite{Wirtinger1936MMP} says that every holomorphic curve in $\c^n=\r^{2n}$ $(n\ge 2)$ is area minimizing as well, hence a minimal surface.  The following are some of the most classical examples of minimal surfaces in $\r^3$.

%
%   CATENOIDS
%
\noindent$\bullet$ The {\em catenoids} (Euler, 1744) were the first minimal surfaces in $\r^3$ to be discovered, apart from (pieces of) affine planes. Planes and catenoids are the only minimal surfaces of revolution in $\r^3$. Here is a parametrization of a catenoid:
\[
	X(\rho,\theta)=\Big(c\cosh\big(\frac{\rho}{c}\big)\cos \theta, c\cosh\big(\frac{\rho}{c}\big)\sin \theta, \rho\Big),
	\quad \rho\in\r,\, \theta\in [-\pi,\pi),
\]
where $c\in\r\setminus\{0\}$ is a constant. See Figure \ref{fig:catenoid}.

%
%   HELICOIDS
%
\noindent$\bullet$ The {\em helicoids}, discovered by Meusnier in 1776, are the only ruled minimal surfaces in $\r^3$ besides planes. Here is a parametrization of a helicoid:
\[
	X(u,v)=\bigl(u\cos(cv),u\sin (cv), v\bigr),\quad (u,v)\in\r^2,
\]
where $c\in\r\setminus\{0\}$ is a constant. See Figure \ref{fig:catenoid}.
\begin{figure}[ht]
    \begin{center}
    \subfigure{
    \includegraphics[width=0.32\textwidth]{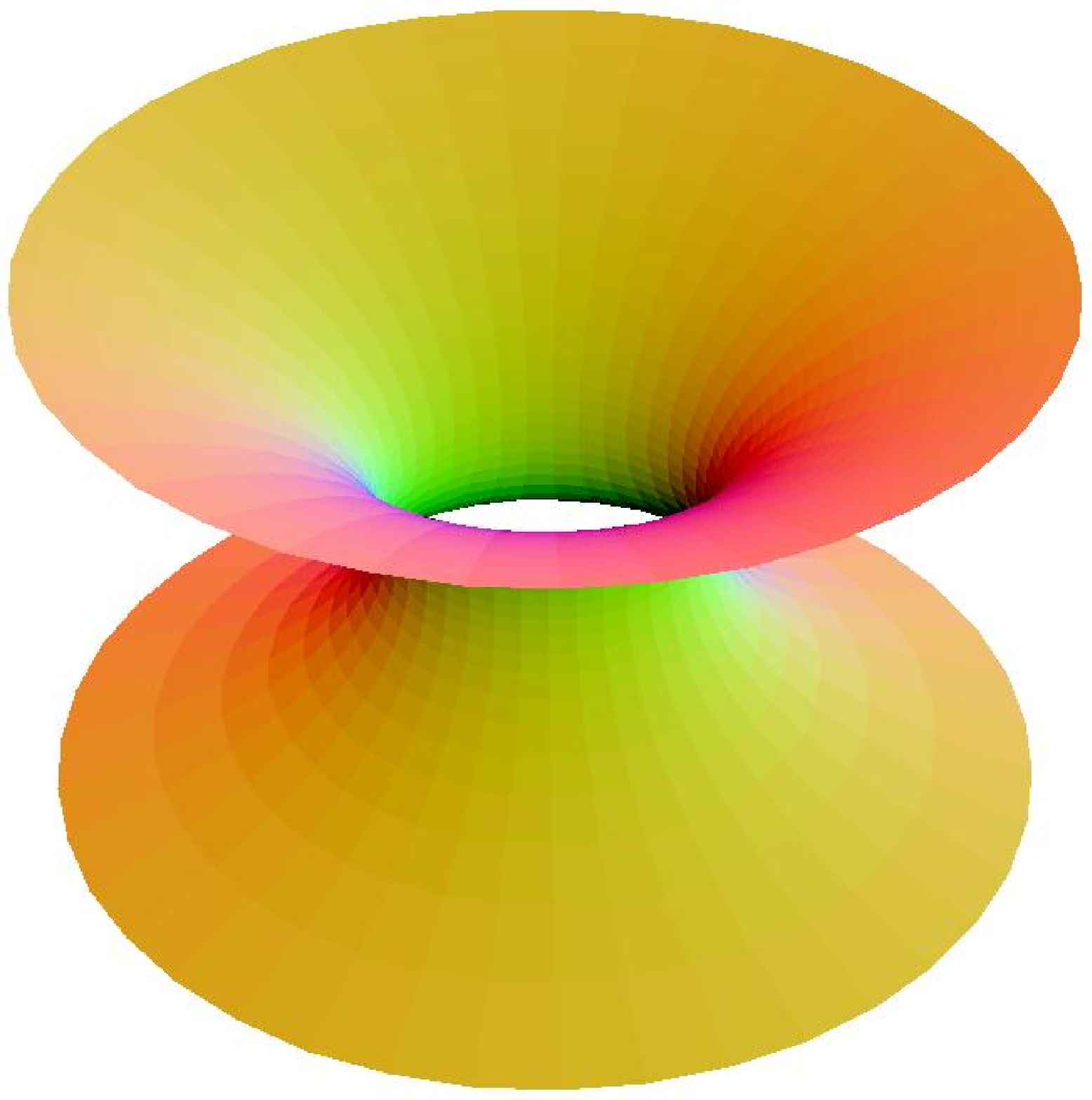}}
    \hspace{10mm}
    \subfigure{
   \includegraphics[width=0.27\textwidth]{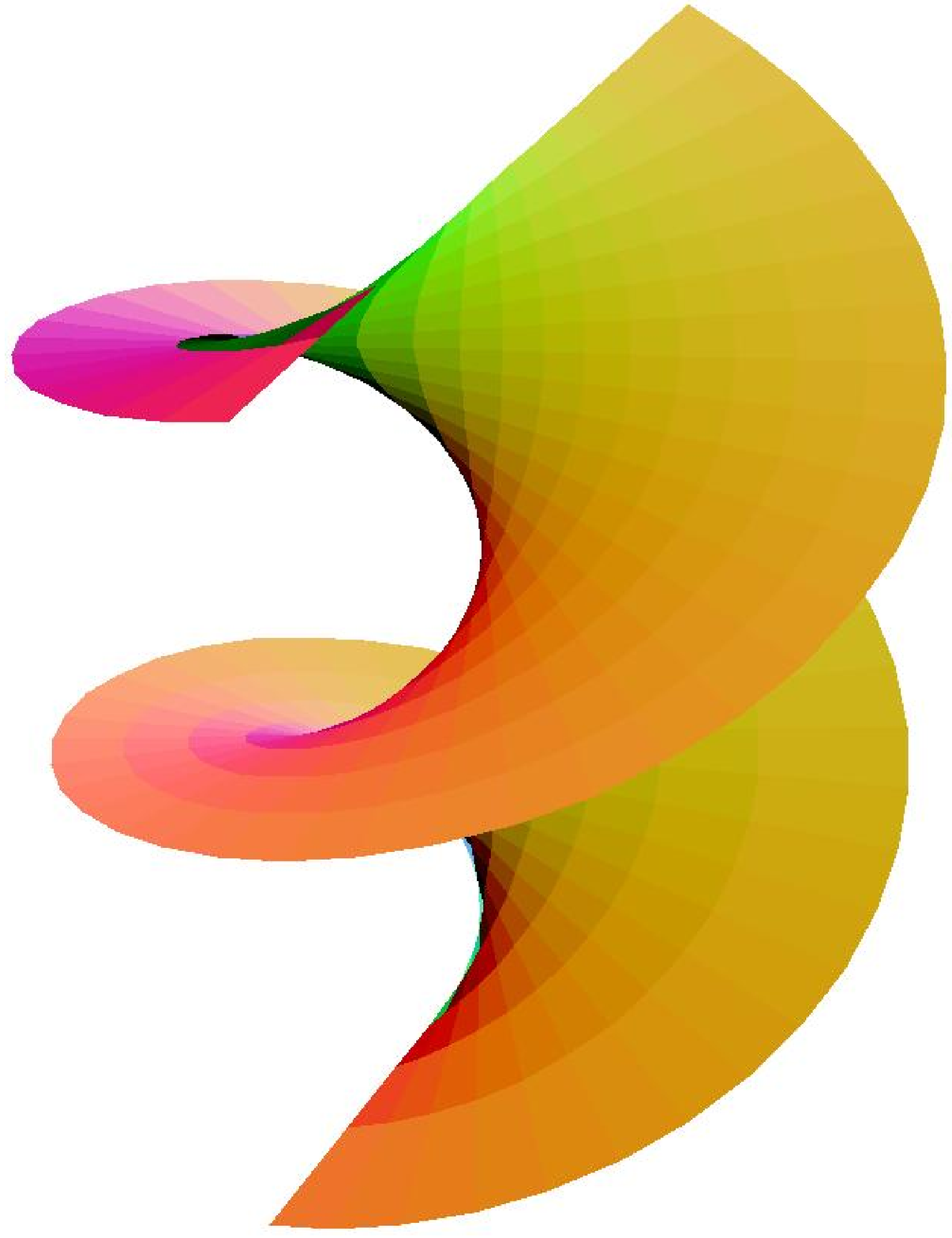}}
     \end{center}
     \vspace{-0.1cm}
     \caption{A catenoid (left) and a helicoid (right)}\label{fig:catenoid}
\end{figure}

%
%   ENNEPER'S SURFACE
%
\noindent$\bullet$ The {\em Enneper surface}, which has self-intersections, was discovered by Enneper in 1864. 
Here is a parametrization; see Figure \ref{fig:enneper}:
\[
	X(u,v)=\Big( \big(1-\frac{u^2}3+v^2\big)u, -\big(1-\frac{v^2}3+u^2\big)v, u^2-v^2 \Big),
	\quad (u,v)\in\r^2.
\]

%
%   RIEMMAN MINIMAL EXAMPLES
%
\noindent$\bullet$ The {\em Riemann minimal examples} form a $1$-parameter family of singly periodic minimal surfaces 
with infinitely many ends asymptotic to parallel planes (see Figure \ref{fig:enneper}).
These surfaces, described by Riemann in a posthumous paper from 1867, are the only minimal surfaces in $\r^3$, 
besides planes, catenoids, and helicoids, that are foliated by circles and affine lines in parallel planes. \begin{figure}[ht]
    \centering
    \subfigure{
    \includegraphics[width=0.38\textwidth]{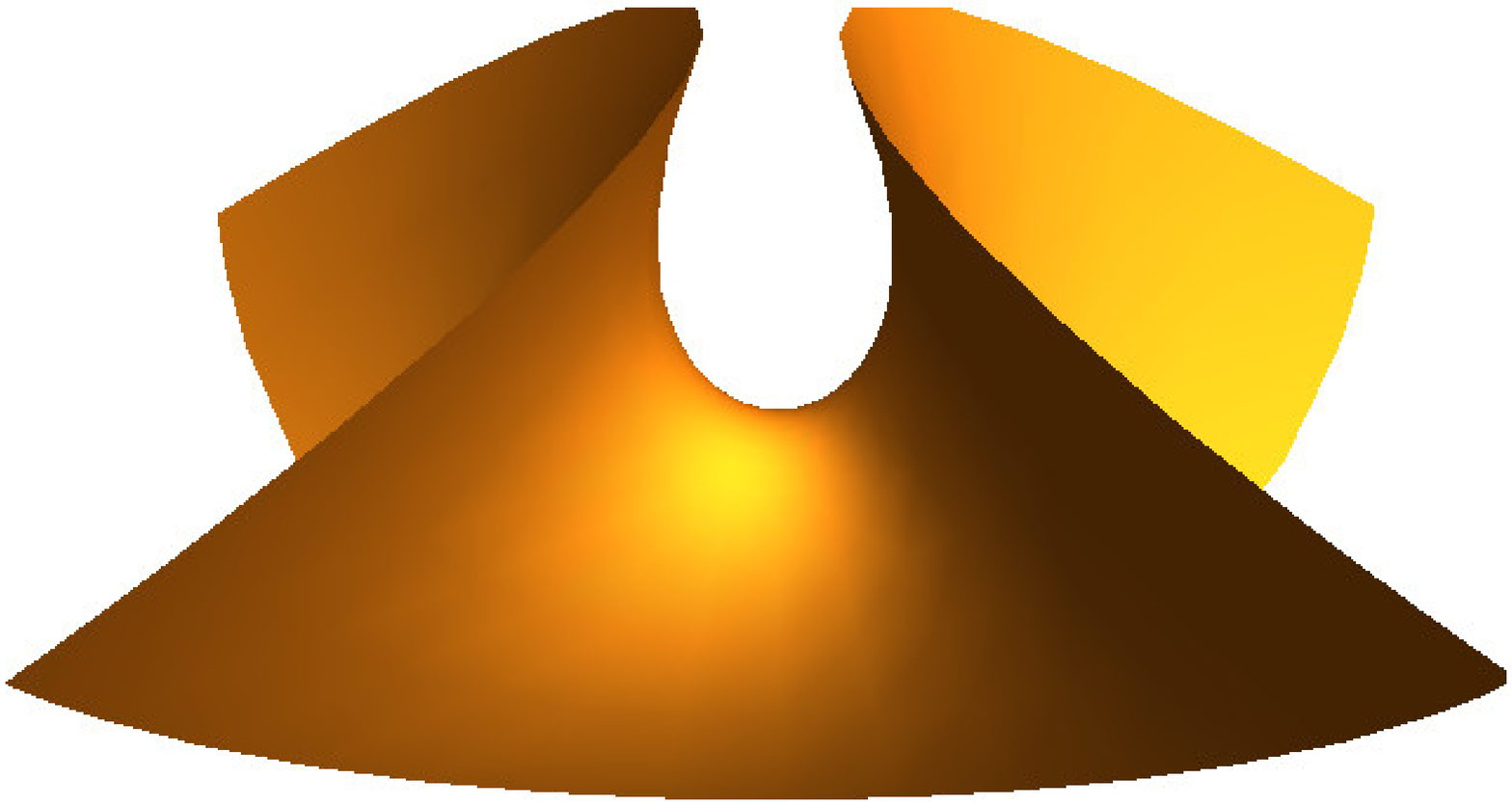}}
    \hspace{5mm}
    \subfigure{
   \includegraphics[width=0.38\textwidth]{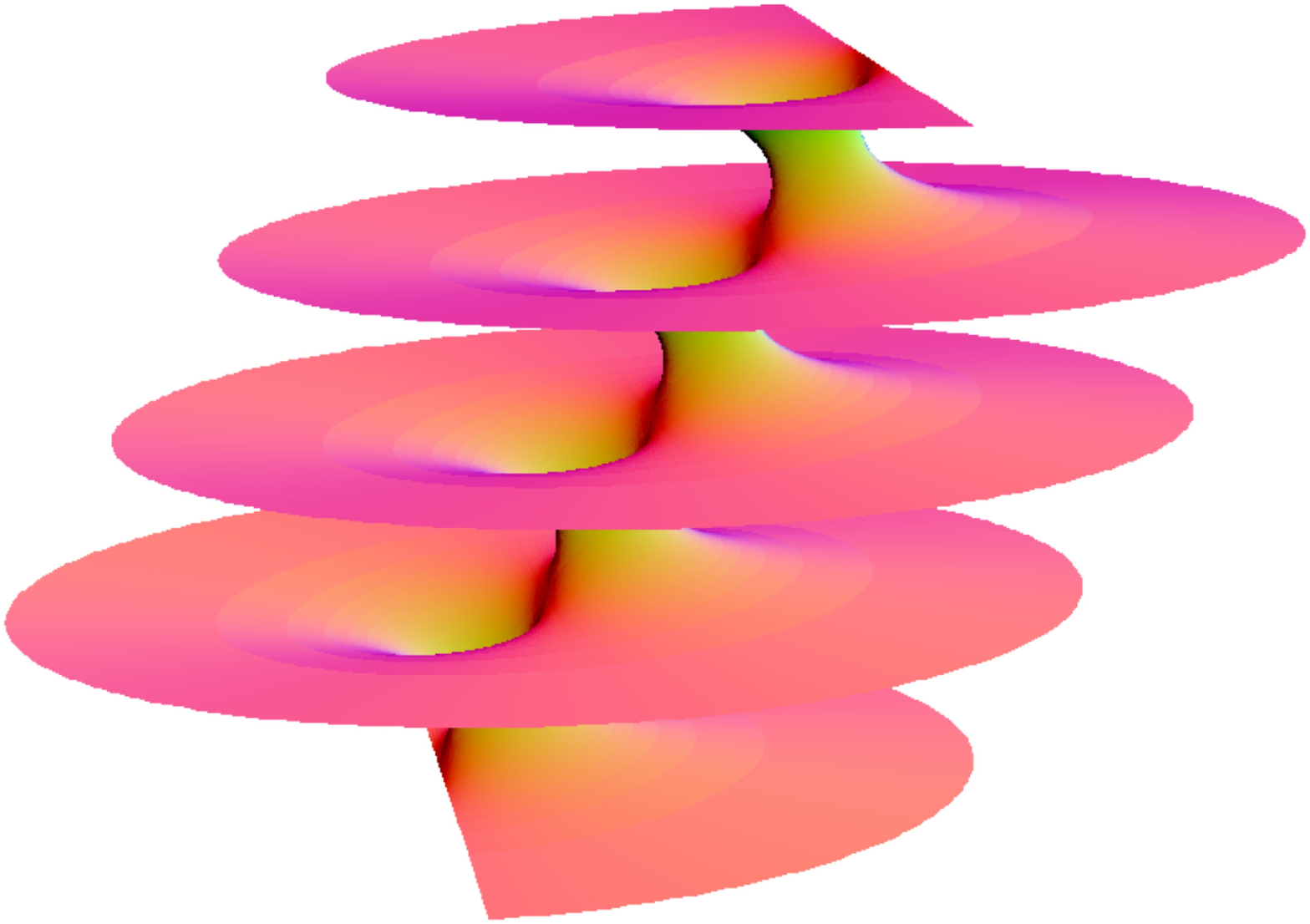}}
     \centering
     \caption{An Enneper surface (left) and a Riemann example (right)}\label{fig:enneper}
\end{figure}

%
%   A PROPERLY EMBEDDED MINIMAL MOBIUS STRIP IN R4
%
\noindent$\bullet$ 
{\em A properly embedded minimal M\"obius strip in $\R^4$} was found by
the authors and L\'opez \cite[Example 6.1]{AlarconForstnericLopezMAMS}.
The harmonic map $X\colon \c_*\to\r^4$ given by
\[
	X(\zeta)=\Re\left( \imath\big(\zeta+\frac1{\zeta}\big), \zeta-\frac1{\zeta},
	\frac{\imath}2 \big(\zeta^2-\frac1{\zeta^2}\big), \frac12 \big(\zeta^2+\frac1{\zeta^2}\big) \right) 
\]
is a proper conformal minimal immersion
such that $X(\zeta_1)=X(\zeta_2)$ if and only if $\zeta_1=\zeta_2$ or $\zeta_1=\Igot(\zeta_1)$,
where $\Igot$ is the fixed-point-free antiholomorphic involution on $\CP^1$ (and on $\C_*$) given by 
$\Igot(\zeta)=-1/\bar\zeta$. Since $\C_*/\Igot$ is a M\"obius strip,
the image surface $X(\c_*)\subset \R^4$ is a properly embedded minimal 
M\"obius strip in $\r^4$.

Another famous example is Meeks's immersed M\"obius strip in $\r^3$ with finite total curvature $-6\pi$
(see \cite[Theorem 2]{Meeks1981DMJ} and \cite[Example 2.6 and Figure 2.3]{AlarconForstnericLopezMAMS}).

%
%   SS: CONFORMAL IMMERSIONS AND THE NULL QUADRIC
%
\subsection{Conformal immersions and the null quadric}\label{ss:null}
Note that vectors $x,y\in\R^n$ are of the same size and orthogonal to each other if and only if the complex vector
$z=x+\imath y\in \C^n$ belongs to the {\em null quadric}
\begin{equation}
\label{eq:null}
	\Agot=\bigl\{z=(z_1,z_2,\ldots,z_n)\in\c^n : z_1^2+z_2^2+\cdots + z_n^2=0\bigr\}.
\end{equation}
Indeed, we have
$\sum_{j=1}^n (x_j+\imath y_j)^2 = |x|^2-|y|^2 + \imath\, x\,\cdotp y$ from which the claim follows. 
Elements $z\in \Agot$ are called {\em null vectors}. Note that $\Agot$ is a complex cone which is 
nonsingular except at the vertex $0\in\Agot$. The {\em punctured null quadric}
\begin{equation}
\label{eq:null*}
	\Agot_*= \Agot\setminus\{0\} = \Agot_{\reg}
\end{equation}
is a homogeneous space of the complex Lie group $\C_*\oplus {\mathrm O}_n(\C)$, where 
${\mathrm O}_n(\C)=\{A\in GL_n(\C): AA^t=I\}$ is the orthogonal group over $\C$.
It follows that maps $M\to \Agot_*$ from any Stein manifold (in particular, from any open Riemann surface) 
satisfy  the {\em Oka principle} (see Theorem \ref{th:OP}). 
This is the most important fact in applications of complex analysis
to the theory of minimal surfaces in $\R^n$. 

Let $M$ be a Riemann surface. Note that an immersion
\begin{equation}\label{eq:immersion}
	X=\left(X_{1},X_{2},\ldots,X_{n}\right) : M\to \R^n
\end{equation}
is conformal if and only if, in any local holomorphic coordinate $\zeta=u+\imath v$ on $M$, 
the partial derivatives $X_u=\di X/\di u = (X_{1,u},\ldots,X_{n,u})\in \R^n$ and 
$X_v=\di X/\di v = (X_{1,v},\ldots,X_{n,v})\in \R^n$ at any given point
have the same length and are orthogonal: 
\begin{equation}\label{eq:conformal} 
	|X_u|=|X_v|>0, \qquad X_u\, \cdotp X_v=0.
\end{equation}
Following the customary notation in complex analysis, we set
\[
	X_\zeta = \frac{\di X}{\di \zeta}= \frac{1}{2}\left(X_u-\imath X_v \right),\qquad 
	X_{\bar \zeta} = \frac{\di X}{\di \bar\zeta} = \frac{1}{2}\left(X_u+\imath X_v \right).
\]
Thus, the equation $X_{\bar \zeta} =0$ characterizes holomorphic functions.
In view of what has been said above, condition \eqref{eq:conformal} is equivalent to 
\begin{equation}\label{eq:conformalC} 
	2X_\zeta = X_u - \imath X_v\in \Agot_*
\end{equation}
where $\Agot_*$ is given by \eqref{eq:null*}. The exterior derivative on $M$ splits into the sum 
\[
	d=\di+\dibar
\]
of the $(1,0)$-part $\di$ and the $(0,1)$-part $\dibar$, where 
\[
	\di X= X_\zeta d\zeta, \qquad \dibar X= X_{\bar \zeta} d\bar\zeta.
\]
Hence, an immersion \eqref{eq:immersion} is conformal  if and only  if  
the $(1,0)$-differential $\di X=(\di X_1,\ldots,\di X_n)$ satisfies the nullity condition
\begin{equation}\label{eq:sum-zero}
	 (\di X_1)^2 + (\di X_2)^2 + \cdots + (\di X_n)^2 =0.
\end{equation}

%
%  THE ENNEPER-WEIERSTRASS REPRESENTATION FORMULA
%
\subsection{The Enneper-Weierstrass representation formula}\label{ss:EW}
Assume now that $M$ is an open Riemann surface and $X\colon M\to \r^n$ is a conformal immersion. 
In any local holomorphic coordinate $\zeta=u+\imath v$ on $M$ the Laplacian equals
\begin{equation}\label{eq:Delta}
	\Delta=\left(\frac{\di}{\di u}\right)^2+\left(\frac{\di}{\di v}\right)^2 = 
	 4\frac{\di^2}{\di \bar \zeta \,\di  \zeta}. 
\end{equation}
In particular, $X$ is harmonic if and only if the $1$-form $\di X$ is holomorphic.
It is classical (see Osserman \cite{Osserman1986book}) that 
\begin{equation}\label{eq:DeltaH}
	\Delta X =2 \mu \mathbf H
\end{equation}
where $\mathbf H$ is the mean curvature vector field of $X$
and $\mu = |X_u|^2=|X_v|^2$. Taking into account also \eqref{eq:DeltaH} gives the following classical result.

\begin{theorem}\label{th:equivalent}
Let $M$ be an open Riemann surface and $X:M\to\R^n$ $(n\ge 3)$ be a smooth conformal immersion.
Then the following conditions are pairwise equivalent.
\begin{itemize} 
\item $X$ is minimal (a stationary point of the area functional \eqref{eq:areafunctional}).
\item $X$ has vanishing mean curvature vector field: $\mathbf H=0$.
\item $X$ is harmonic: $\Delta X=0$.
\item The $\C^n$-valued $1$-form $\di X=(\di X_1,\ldots,\di X_n)$ is nowhere vanishing on $M$, holomorphic,
and satisfies the nullity condition \eqref{eq:sum-zero}.
\item Let $\theta$ be a nowhere vanishing holomorphic $1$-form on $M$. The map 
$f=2\di X/\theta\colon M\to \C^n$ is holomorphic and assumes values in $\Agot_*$ \eqref{eq:null*}. 
\end{itemize}
If these conditions hold then the induced Riemannian metric on $M$ equals
\begin{equation}\label{eq:metric}
	X^*(ds^2)=2 \big(|\di X_1|^2+\cdots+|\di X_n|^2\big).
\end{equation}
\end{theorem}

Every open Riemann surface $M$ admits a nowhere vanishing holomorphic
$1$-form $\theta$ by the Oka-Grauert principle (see \cite[Theorem 5.3.1(iii)]{Forstneric2017E}).
If $X\colon M\to\R^n$ is a conformal minimal immersion then the $1$-form $2\di X=f\theta$
with values in $\Agot_*$ has exact real part since $dX= \di X +\dibar X = 2\Re(\di X)$;
equivalently, $\oint_C \Re (f\theta)=0$ for every closed curve $C\subset M$.
Conversely, every holomorphic $1$-form $f\theta$ with values in $\Agot_*$ and exact real part 
$\Re(f\theta)$ determines a conformal minimal immersion by integration.
We record this observation in the following corollary to Theorem \ref{th:equivalent}.

\begin{theorem}[The Enneper-Weierstrass representation formula] \label{th:EW}
Let $M$ be a connected open Riemann surface, $\theta$ be a nowhere vanishing holomorphic $1$-form on $M$,
and $p_0\in M$ be an arbitrary point. 
Every conformal minimal immersion $X\colon M\to \R^n$ $(n\ge 3)$ is of the form
\begin{equation}\label{eq:EW}
	X(p)=X(p_0) +  \int_{p_0}^p \Re(f\theta), \quad p\in M,
\end{equation}
where $f\colon M\to \Agot_*$ is a holomorphic map into the punctured null quadric
such that the $\R^n$-valued $1$-form $\Re(f\theta)$ is exact. We have that $2\di X=f\theta$.
\end{theorem}

%
%    THE COMPLEX GAUSS MAP
%
\subsection{The complex Gauss map}\label{ss:Gauss}
Let $X=(X_1,\ldots,X_n) \colon M\to\R^n$ be a conformal minimal immersion. Its differential 
$\di X= (\di X_1, \ldots, \di X_n)$ determines the Kodaira type holomorphic map
\begin{equation}\label{eq:GX}
	G_X\colon M\to \CP^{n-1},\quad G_X(p) = [\di X_1(p) \colon \cdots \colon \di X_n(p)],
	\quad p\in M,
\end{equation}
called the {\em generalized Gauss map} of $X$.
In view of the equation \eqref{eq:sum-zero}, $G_X$ assumes values in the complex hyperquadric
\begin{equation}\label{eq:nullquadric-projected}
     Q_{n-2} = \bigl\{[z_1:\ldots : z_n]\in\CP^{n-1} : z_1^2+ \cdots + z_n^2 = 0\bigr\}. 
\end{equation}     
In the case $n=3$ the quadric $Q_1\subset \CP^2$ is the image of a quadratically
embedded Riemann sphere $\CP^1\hra \CP^2$, and the {\em complex Gauss map} of a conformal minimal
immersion $X=(X_1,X_2,X_3)\colon M\to\R^3$ is defined to be the holomorphic map 
\begin{equation}\label{eq:C-Gauss}
	g_X = \frac{\di X_3}{\di X_1-\imath \, \di X_2}  =  \frac{\di X_2-\imath\, \di X_1}{\imath\, \di X_3} : M \lra \CP^1.
\end{equation}
The function $g_X$ equals the stereographic projection of the real Gauss map
$N=(N_1,N_2,N_3) \colon M\to \S^2 \subset \r^3$ to the Riemann sphere $\CP^1$; explicitly,
\[
	g_X = \frac{N_1+\imath N_2}{1-N_3} : M \lra \C\cup\{\infty\}=\CP^1.
\]
We can recover the differential $\di X=(\di X_1,\di X_2,\di X_3)$ from the pair $(g_X,\phi_3)$
with $\phi_3=\di X_3$ by the classical {\em Weierstrass formula}:
\begin{equation}\label{eq:EWR}
	\di X= \Phi=(\phi_1,\phi_2,\phi_3) =
	\left( \frac{1}{2} \left(\frac{1}{g_X}-g_X\right),  \frac{\imath}{2} \left(\frac{1}{g_X}+g_X \right),1\right) \phi_3.
\end{equation}
(See \cite[Lemma 8.1, p.\ 63]{Osserman1986book}.) 
Conversely, given a pair $(g,\phi_3)$ consisting of a holomorphic map $g\colon M\to \CP^1$
and a holomorphic $1$-form $\phi_3$ on $M$, the meromorphic $1$-form 
$\Phi=(\phi_1,\phi_2,\phi_3)$ defined by \eqref{eq:EWR} satisfies $\sum_{j=1}^3 \phi_j^2=0$;
it is the differential $\di X$ of a conformal minimal immersion $X\colon M\to\R^3$
if and only if it is holomorphic, nowhere vanishing, and its real periods vanish.
If this holds then the map $X$ is obtained from $\Phi$ \eqref{eq:EWR} by integration:
\[
	X(p)=X(p_0) +  2 \int_{p_0}^p \Re(\Phi), \qquad p\in M.
\]

The generalized Gauss map $G_X$ is of great importance
in the theory of minimal surfaces; see Osserman \cite{Osserman1986book} and the papers \cite{AlarconForstnericLopez2017JGEA,Fujimoto1983JMSJ,Fujimoto1990JDG,HoffmanOsserman1980MAMS,
LopezPerez2003IUMJ,OssermanRu1997JDG,Ros2002DG,Ru1991JDG}, 
among many others. In particular, the complex Gauss map $g_X \colon M\to\CP^1$ 
\eqref{eq:C-Gauss} of a minimal surface in $\r^3$ provides crucial information about its geometry
since the key quantities such as the Gauss curvature and the Jacobi operator
depend only on $g_X$  
(see e.g.\ \cite{MeeksPerez2004SDG,MeeksPerez2012AMS,Osserman1980DG,Osserman1986book}). 
The authors together with F.\ J.\ L{\'o}pez have recently shown in 
\cite[Corollary 1.2]{AlarconForstnericLopez2017JGEA} that every meromorphic function on an open 
Riemann surface $M$ is the complex Gauss map of a conformal minimal immersion $X\colon M\to \R^3$; 
furthermore, $X$ can be chosen as the real part of a holomorphic null curve $Z=X+\imath Y \colon M\to\C^3$.

%
%   FLUX AND PERIOD MAP
%
\subsection{Flux, period map, conjugate surfaces, and null curves} \label{ss:flux}
The {\em conjugate differential} of a smooth map $X:M\to\r^n$ is defined by  
\[
	d^c X= \imath(\dibar X - \di X) = 2 \Im (\di X).  
\]
Recall that $d=\di+\dibar$. Hence we get 
\[
	2\di X = dX + \imath d^c X, \qquad 
	dd^c X= 2\imath\,  \di\dibar X = \Delta_\zeta X \cdotp  du\wedge dv,
\] 
where the last equation holds in any local holomorphic coordinate $\zeta=u+\imath v$.
Thus, the map $X$ is harmonic if and only if $d^c X$ is a closed vector valued $1$-form on $M$, 
and in this case $d^c X=dY$ holds for any local harmonic conjugate $Y$ of $X$. 

The {\em flux map}  of a harmonic map $X\colon M \to\r^n$  
is the group homomorphism $\Flux_X\colon H_1(M;\z)\to\r^n$ given by
\begin{equation} \label{eq:flux}
	\Flux_X([C])=\oint_C d^c X = \oint_C 2 \Im (\di X), \qquad 	[C]\in H_1(M;\z).
\end{equation}
The integral is independent of the choice of the path in a given homology class, 
and we shall write $\Flux_X(C)$ for $\Flux_X([C])$ in the sequel.

Fix a nowhere vanishing holomorphic $1$-form $\theta$ on $M$.  
% (Such exists by the Oka-Grauert principle, cf.\ \cite[Theorem 5.3.1]{Forstneric2017E}.)
Associated to any holomorphic map $f\colon M\to\c^n$ is the {\em period homomorphism} 
$\Pcal(f)\colon H_1(M;\z) \to \c^n$ defined on any closed oriented curve $C\subset M$ by 
\begin{equation}\label{eq:period-map}
	\Pcal(f)(C) = \oint_{C} f\theta. 
\end{equation}
The map $f$ corresponds to a conformal minimal immersion $X \colon M\to\r^n$ as in (\ref{eq:EW})
if and only if $f(M)\subset \Agot_*$ and $\Re(\Pcal(f))=0$; in this case, $X$ is given by \eqref{eq:EW} and
\begin{equation}\label{eq:FP}
	\Flux_X = \Im (\Pcal(f)) : H_1(M;\z) \to \r^n. 
\end{equation}
We have $\Flux_X=0$ if and only if $X$ admits a globally defined harmonic conjugate 
$Y\colon M\to\R^n$ (the {\em conjugate conformal minimal surface}), and in this case
the holomorphic immersion $Z=(Z_1,\ldots,Z_n) = X+\imath Y : M\to \C^n$ satisfies 
\[
	(dZ_1)^2 + (dZ_2)^2 + \cdots + (dZ_n)^2 = 0.
\]
Such $Z$ is called a {\em null holomorphic immersion} of $M$ into $\C^n$.
We see as in Theorem \ref{th:EW} that every null holomorphic immersion is of the form 
\begin{equation}\label{eq:NULL}
	Z(p)=Z(p_0) +  \int_{p_0}^p f\theta, \qquad p\in M,
\end{equation}
where $f\colon M\to \Agot_*$ is a holomorphic map into the punctured null quadric
such that the $\C^n$-valued holomorphic $1$-form $f\theta$ is exact. The minimal surfaces
\[
	X_t = \Re(e^{\imath t} Z) \colon M\to\R^n,\qquad t\in \R
\]
are called the {\em associated minimal surfaces} of the null curve $Z$.

\begin{example}
The {\em catenoid} and the {\em helicoid}  (see Figure \ref{fig:catenoid})
are conjugate minimal surfaces --- the real and the imaginary part of the null curve $Z\colon \C\to \C^3$ given by
\[ 
	Z(\zeta) = (\cos\zeta,\sin\zeta,-\imath \zeta),\qquad \zeta=x+\imath y\in\C. 
\]
Consider the family of minimal surfaces ($t\in \R$):
\begin{eqnarray*}
  X_t(\zeta) &=& \Re\left( e^{\imath t} Z(\zeta) \right) \\
     &=& \cos t \left( 
     \begin{matrix} \cos x \cdotp \cosh y \cr \sin x \cdotp \cosh y \cr y\end{matrix}\right)
      + \sin t \left( 
      \begin{matrix} \sin x \cdotp\sinh y \cr -\cos x \cdotp\sinh y \cr x\end{matrix}\right)
\end{eqnarray*}
At $t=0$ we have a parametrization of a catenoid, and at $t=\pm \pi/2$ we have a 
(left or right handed) helicoid.
\end{example}

%
%    MAPPING SPACES
%
\subsection{Spaces of mappings} \label{ss:spaces}
If $M$ is an open Riemann surface, then $\Ocal(M)$ is the algebra of holomorphic functions  $M\to\c$, 
$\Ocal(M,X)$ is the space of  holomorphic mappings $M\to X$ to a complex manifold $X$, 
\[
	\CMI(M,\r^n)
\] 
is the set of all conformal minimal immersions $M\to \r^n$, and 
\[
	\NC(M,\C^n)
\] 
is the space of all null holomorphic immersions $M\to\C^n$. These spaces are endowed with the compact-open topology.

Assume now that $M$ is a compact bordered Riemann surface (see Sect.\ \ref{ss:Riemann}).
Given $r\in \z_+=\{0,1,2,\ldots\}$,  we denote by $\Ascr^r(M)$ the space of all functions
$M\to \c$ of class $\Cscr^r(M)$ that are holomorphic in $\mathring M=M\setminus bM$. 
More generally, for any complex manifold $X$ we let $\Ascr^r(M,X)$ denote the space 
of maps $M\to X$ of class $\Cscr^r$ which are holomorphic in $\mathring M$.
We write $\Ascr^0(M)=\Ascr(M)$ and $\Ascr^0(M,X)=\Ascr(M,X)$. 
Note that $\Ascr^r(M,\c^n)$ is a complex Banach space, and for any complex manifold $X$ the space $\Ascr^r(M,X)$ 
is a complex Banach manifold modeled on the Banach space $\Ascr^r(M,\c^n)$ with $n=\dim X$ 
(see \cite[Theorem 8.13.1]{Forstneric2017E} or \cite[Theorem 1.1]{Forstneric2007AJM}). 
A compact bordered Riemann surface $M$ can be considered as a smoothly bounded compact domain 
in an open Riemann surface $R$. It is classical that each function in $\Ascr^r(M)$ can be approximated 
in the $\Cscr^r(M)$ norm by functions in $\Ocal(M)$, i.e., functions holomorphic in a
neighborhood of $M$ in $R$. The same holds for maps to an arbitrary complex manifold 
(see \cite[Theorem 5.1]{DrinovecForstneric2007DMJ}). 
For any $r\in \n$ we denote by 
\[
	\CMI^r(M,\r^n)
\] 
the set of all  conformal minimal immersions $M\to\r^n$ 
of class $\Cscr^r(M)$. More precisely, an immersion $F\colon M\to \r^n$ of class $\Cscr^r$ belongs to 
$\CMI^r(M,\r^n)$ if and only if $\di F$ is a $(1,0)$-form of class $\Cscr^{r-1}(M)$ that
satisfies the nullity condition (\ref{eq:sum-zero}) and is holomorphic on the interior 
$\mathring M$.  Similarly, 
\[
	\NC^r(M,\C^n) 
\]
denotes the space of all null holomorphic immersions $M\to \C^n$ of class $\Ascr^r(M)$.

The following notions will play an important role in our analysis.

%
%
%    NONFLAT CMIS
%
%
\begin{definition}
\label{def:nondegenerate}
Let $M$ be a connected open or bordered Riemann surface, 
let $\theta$ be a nowhere vanishing holomorphic $1$-form on $M$, 
and let $\Agot$ be the null quadric (\ref{eq:null}).
\begin{enumerate}
\item A holomorphic map $f:M\to \Agot_*$ is {\em flat} if the image 
$f(M)$ is contained in a complex ray $\c\nu\subset \Agot$ $(\nu\in \Agot_*)$ of the null quadric,
and is {\em nonflat} otherwise.

\vspace{1mm}

\item A conformal minimal immersion $X\colon M\to \r^n$ is {\em nonflat} if 
the map $f=\di X/\theta: M\to \Agot_*$ is nonflat; equivalently,  if the image $X(M)\subset\R^n$ 
is not contained in an affine plane. A null holomorphic immersion 
$Z\colon M\to \C^n$ is {\em nonflat} if the map $f=dZ/\theta: M\to \Agot_*$ is nonflat.

\vspace{1mm}

\item  A holomorphic map $f:M\to \Agot_*$ is {\em full} if the image 
$f(M)$ is not contained in any complex hyperplane of $\C^n$.
A conformal minimal immersion $X\colon M\to \r^n$ is {\em full} if the image $X(M)$ is not contained 
in any affine hyperplane of $\R^n$. 
\end{enumerate}
\end{definition}

For a conformal minimal immersion $M\to \r^3$, nonflat and full are equivalent conditions. 
However, in dimensions $n>3$  we clearly have that $\text{full $\Longrightarrow$nonflat}$,
but the converse is obviously not true. 
If $M$ is an open Riemann surface, we denote by $\CMI_{\rm nf}(M,\r^n)$ the open subset of $\CMI(M,\r^n)$ 
consisting of all immersions which are nonflat on every connected component of $M$. 
The analogous notation 
\[
	\CMI_{\rm nf}^r(M,\r^n) \subset \CMI^r(M,\r^n)
\]
is used for a compact bordered Riemann surface $M$. 
Likewise,  $\NC_{\rm nf}(M,\C^n)$ is the space of all nonflat holomorphic null curves.

Since the tangent space $T_z\Agot$ is the kernel at $z$ of the $(1,0)$-form $\sum_{j=1}^n z_j\, dz_j$, 
we have $T_z\Agot = T_w\Agot$ for $z,w\in\c^n\setminus\{0\}$ if and only if $z$ and $w$ are colinear.
This implies

\begin{lemma}\label{lem:nonflat}
A holomorphic map $f\colon M\to \Agot_*$ is nonflat if and only if the linear span of the tangent spaces 
$T_{f(p)} \Agot\subset T_{f(p)}\C^n\cong \C^n$ 
over all points $p\in M$ equals $\c^n$.
\end{lemma}

%
%   ADMISSIBLE SETS AND GCMIs
%

We now introduce sets in Riemann surfaces that are used in Mergelyan approximation theorems for
conformal minimal immersions, and the notion of a generalized conformal minimal
immersion on them. Such sets appear naturally in the constructions of conformal minimal immersions. 

\begin{definition}\label{def:admissible}
Let $M$ be an open Riemann surface. A compact set $S\subset M$ is {\em admissible} if it is Runge in $M$ 
and of the form $S=K\cup \Gamma$, where $K$ is a finite union of pairwise disjoint smoothly bounded compact domains 
in $M$ and $\Gamma= S \setminus\mathring  K$ is a finite union of pairwise disjoint smooth Jordan arcs and closed 
Jordan curves meeting $K$ only in their endpoints (or not at all) and such that their intersections with the boundary 
$bK$ of $K$ are transverse.
\end{definition} 

\begin{definition}\label{def:GCMI}  
Let $S=K\cup\Gamma$ be an admissible set in an open Riemann surface $M$ and let $\theta$ be a 
nowhere vanishing holomorphic $1$-form on $M$. A {\em generalized conformal minimal immersion}  
$S\to\r^n$ is a pair $(X,f\theta)$, where $X\colon S\to \r^n$ is a smooth map that is a conformal 
minimal immersion on an open neighborhood of $K$ and $f\colon S\to\Agot_*$ is a smooth map that is 
holomorphic on a neighborhood of $K$, such that
\begin{itemize}
\item $f\theta =2\di X$ holds on an open neighborhood of $K$, and
\vspace{1mm}
\item for any smooth path $\alpha$ in $M$ parametrizing a connected component of $\Gamma$ we have 
$\Re(\alpha^*(f\theta))=\alpha^*(dX)=d(X\circ \alpha)$.
\end{itemize}
\end{definition}

We denote the space of all generalized conformal minimal immersions $S\to\r^n$ by 
\[
	\GCMI(S,\r^n).
\] 
 
%%%%%%%%%%
%%%%%%%%%%
%%%%%%%%%% SECTION: OKA THEORY, PERIOD DOMINATING SPRAYS,...
%%%%%%%%%%
%%%%%%%%%% 
%%%%%%%%%%

\section{\sc Oka theory, period dominating sprays, and loops with given periods in the null quadric} \label{sec:Oka}

Oka theory concerns the existence, approximation, and extension theorems 
for holomorphic maps $f\colon S\to O$ from Stein manifolds $S$ to Oka manifolds $O$. 
In this section we recall the main results of Oka theory which 
are used in the study of minimal surfaces. For Stein manifolds, see any of the monographs 
\cite{GrauertRemmert1979,GunningRossi2009,Hormander1990book}
or \cite[Chap.\ 2]{Forstneric2017E}. For Oka theory, see \cite[Chaps.\ 5--7]{Forstneric2017E}
and the surveys \cite{Forstneric2013KAWA,ForstnericLarusson2011}.
A recent survey of holomorphic approximation theory is available in \cite{FornaessForstnericWold2018}.

%
%   OKA MANIFOLDS
%
\subsection{Stein manifolds}\label{ss:Stein}
A complex manifold $S$ is said to be a {\em Stein manifold}
(named after Karl Stein who introduced this important class of complex manifolds in 1951)
% under the name  {\em holomorphically complete manifolds}) 
if it satisfies the following two conditions:
\begin{enumerate}
\item holomorphic functions on $S$ separate any pair of distinct points, and
\vspace{1mm}
\item if $K$ is a compact subset of $S$, then so is its $\Ocal(S)$-convex hull
\[
	\wh K =\bigl\{x \in S : |f(x)| \le \sup_K |f|\ \ \forall f\in \Ocal(S)\bigr\}.
\]
\end{enumerate}
A compact set $K\subset S$ is called {\em $\Ocal(S)$-convex} if $K=\wh K$. 
If $S=\C^n$ then $\wh K$ is the {\em polynomial hull} of $K$. 
Clearly, no manifold containing a compact complex submanifold of positive dimension is Stein.

The main example for the purposes of this paper is when $\dim S=1$, i.e., $S$ is a Riemann surface.
Every open Riemann surface is a Stein manifold according to Behnke and Stein (1949), and 
in this case the hull $\wh K$ of any compact set $K\subset S$ is the union
of $K$ with all relatively compact connected components of its complement $S\setminus K$.
Furthermore, the Cartesian product $S_1\times S_2$ of a pair of Stein manifolds is Stein,
and the total space $E$ of any holomorphic vector bundle $E\to S$ over a Stein base $S$ is Stein.
A domain $\Omega\subset \C^n$ is Stein if and only if it is a domain of holomorphy
(which holds if and only if it is pseudoconvex). In particular, every domain in $\C$ is Stein,
and every convex domain in $\C^n$ for any $n\ge 1$ is Stein.

There are several other characterizations of the class of Stein manifolds. One is that a 
Stein $n$-manifold $S$ embeds properly holomorphically into the Euclidean space $\C^{2n+1}$
(Remmert 1956, Bishop 1960, Narasimhan 1961; see \cite[Theorem 2.4.1]{Forstneric2017E}); 
the converse is easily seen by restricting holomorphic polynomials to the embedded submanifold. 
Another characterization of Stein manifolds is by the existence of strongly plurisubharmonic exhaustion
functions (see \cite[Sect.\ 2.5]{Forstneric2017E}). 

The axioms (1) and (2) say that a Stein manifold admits many holomorphic functions. 
More explicit manifestations of this phenomenon are the {\em Oka-Weil approximation 
theorem} and the {\em Oka-Cartan extension theorem}.  The first one says
that, given an $\Ocal(S)$-convex compact set $K$ in a Stein manifold $S$ and
a holomorphic function $f$ on a neighborhood of $K$, we can 
approximate $f$ as closely as desired uniformly on $K$ by global holomorphic functions on $S$.
This generalizes the classical Runge theorem for functions on $\C$. 
(See the survey \cite{FornaessForstnericWold2018} for more information.)
The second one says that for any closed complex subvariety $S'$ of
a Stein manifold $S$ and holomorphic function $f\colon S'\to \C$ there exists
a holomorphic function $F\colon S\to \C$ extending $f$, i.e., $F|_{S'}=f$.
If $S$ is an open Riemann surface and $S'$ is a discrete subset of $S$, this  
is the classical Weierstrass interpolation theorem \cite{Weierstrass1886} (see also \cite{Florack1948SMIUM}).
One may combine the approximation and the interpolation statement, including  also 
jet interpolation on a subvariety and continuous dependence on parameters; 
see \cite[Theorem 2.8.4]{Forstneric2017E}. The same results hold for sections
of any holomorphic vector bundle over a Stein manifold. These classical results, 
along with Cartan's Theorems A and B  (see \cite[Sect.\ 2.6]{Forstneric2017E}), 
form the basis for analysis on Stein manifolds.

%
%    OKA THEORY
%
\subsection{Oka theory}\label{ss:Oka}
We may consider holomorphic functions on Stein manifolds as holomorphic maps $S\to \C$. Applying the above 
mentioned approximation and interpolation results componentwise, we can extend them to
maps $S\to\C^N$ for any $N\in\N$. A completely different picture emerges
for maps $S\to X$ to more general complex manifolds.
For example, {\em Picard's theorem} says that there are no nonconstant
holomorphic maps $\C\to \C\setminus\{0,1\}$. On the other hand, Grauert proved in 1957--58
\cite{Grauert1957MA,Grauert1958MA} that the approximation and interpolation 
results still hold in the absence of topological obstructions for maps to 
complex homogeneous manifolds; the case when $X=\C_*$ is Oka's theorem from 1939.

%
%    THE OKA-GRAUERT THEOREM
%
\begin{theorem}[The Oka-Grauert theorem]
\label{th:OP}
Assume that $S$ is a Stein manifold, $K$ is an $\Ocal(S)$-convex compact subset of $S$,
$S'$ is a closed complex subvariety of $S$, $X$ is a complex homogeneous manifold,
and $f\colon S\to X$ is a continuous map that is holomorphic on an open neighborhood 
of $K$ and whose restriction $f|_{S'}\colon S'\to X$ is holomorphic. Then, $f$ can
be approximated uniformly on $K$ by holomorphic maps $F\colon S\to X$ 
satisfying $F|_{S'}=f$. If in addition $f$ is holomorphic on a neighborhood of $S'$,
then $F$ can be chosen to agree with $f$ to any given finite order along $S'$.
The analogous result holds for sections $S\to E$ of any principal fibre bundle $\pi:E\to S$
over a Stein manifold $S$.
\end{theorem}

\begin{comment}
In fact, the initial map $f$ in Theorem \ref{th:OP} can be deformed
to a holomorphic map $F$ through a homotopy of continuous maps $f_t\colon S\to X$
$(t\in [0,1])$ which is fixed on $S'$ (to any finite order if the initial map
$f=f_0$ is holomorhic on a neighborhood of $S'$)  and such that for every $t$,
$f_t$ is holomorphic near $K$ and uniformly close to $f$ on $K$. 

The analogous result holds for families of maps depending on a parameter in a compact Hausdorff space.
\end{comment}

In the theory of minimal surfaces, Theorem  \ref{th:OP} is mainly used with 
$X$ either the punctured null quadric $\Agot_*\subset \C^n$, 
the intersection of $\Agot_*$ with an affine complex hyperplane in $\C^n$,
the punctured Euclidean space $\C^n_*$, or a projective space $\CP^n$. 
All these manifolds are complex homogeneous.

A complex manifold $X$ satisfying the conclusion of Theorem \ref{th:OP} is called an {\em Oka manifold}. 
The class of Oka manifolds also contains many nonhomogeneous manifolds; 
see \cite[Sect.\ 5.6 and Chap.\ 7]{Forstneric2017E}. 
The most general Oka principle for maps from Stein manifolds to Oka manifolds
is given by \cite[Theorem 5.4.4]{Forstneric2017E} which also includes the parametric case,
i.e., families of maps depending continuously on a parameter in a compact Hausdorff space.
It follows in particular that for every Stein manifold $S$ and Oka manifold $O$
the natural inclusion $\Ocal(S,O)\hra \Cscr(S,O)$ is a weak homotopy equivalence
(see \cite[Corollary 5.5.6]{Forstneric2017E}), and is the inclusion of a strong deformation retract
(hence a homotopy equivalence) if $S$ is of finite analytic type in the sense that 
it admits a strongly plurisubharmonic exhaustion function
with only finitely many critical points (see \cite[Theorem 5.5.9]{Forstneric2017E} due to
L\'arusson). Note that an open Riemann surface $S$ is of finite analytic type if
and only if it is of finite topological type, i.e., the homology group $H_1(S;\Z)$ is finitely generated.

Theorem \ref{th:OP} and its extension to Oka manifolds also hold with Mergelyan type approximation; 
see \cite[Corollaries 5.4.6 and 5.4.7]{Forstneric2017E} and \cite{FornaessForstnericWold2018}.

A useful sufficient condition for a manifold $X$ to be Oka is the existence of  finitely many $\C$-complete 
holomorphic vector fields $V_1,\ldots, V_N$ on $X$ which span the tangent space of $X$ at any point. 
(If $X=G/H$ is a homogeneous manifold of a complex Lie group $G$,  
this holds for $G$-invariant holomorphic vector fields on $X$ which are always complete.)  
The composition of their flows $\phi^j_t$ for complex values of $t$ 
gives the map $\sigma \colon X\times \C^N\to X$, defined by
\begin{equation}\label{eq:flowspray}
	\sigma(x,t_1,\ldots,t_N) = \phi^1_{t_1}\circ \cdots \circ \phi^N_{t_N}(x)\in X
\end{equation}
for $x\in X$ and $t=(t_1,\ldots,t_N)\in\C^N$, satisfying the domination condition
\begin{equation}\label{eq:domination}
	\frac{\di \sigma(x,t)}{\di t}\bigg|_{t=0} \colon \C^N \to T_x X\quad \text{is surjective for every}\ x\in X. 
\end{equation}
A holomorphic map $\sigma\colon X\times \C^N\to X$ satisfying 
$\sigma(x,0)=x$ for all $x\in X$ and the domination condition \eqref{eq:domination} is called a 
{\em dominating spray} on $X$. More generally, we may take as the domain of the spray the total space 
$E$ of any holomorphic vector bundle $\pi\colon E\to X$. Gromov proved in \cite{Gromov1989JAMS} 
that every complex manifold admitting a dominating spray is an Oka manifold. 
For more on this subject see \cite[Chap.\ 6]{Forstneric2017E}.

A {\em (holomorphic) dominating spray of maps} $S\to X$ is a holomorphic map 
$F\colon S\times V\to X$, where $V\subset \C^N$ is an open
neighborhood of the origin in a complex Euclidean space, such that
\begin{equation}\label{eq:domination-maps}
	\frac{\di F(s,t)}{\di t}\bigg|_{t=0} : \C^N \longrightarrow T_{F(s,0)} X\quad \text{is surjective for every}\ s\in S.
\end{equation}
The map $F_0=F(\cdotp,0): S\to X$ is called the {\em core map}, or simply the {\em core}, 
of $F$. If $X$ admits a dominating spray $\sigma\colon X\times\C^N\to X$, then for any
holomorphic map $f\colon S\to X$, the map $F\colon S\times \C^N\to X$ given by
\[
	F(s,t) = \sigma(f(s),t) \in X,\qquad s\in S,\ t\in \C^N
\]
is a dominating spray of maps with the core $F_0=f$. For instance, if $\sigma$ 
is of the type \eqref{eq:flowspray} defined by flows $\phi^j_t$ of complete holomorphic vector fields, then
\begin{equation}\label{eq:flowspray-maps}
	F(s,t_1,\ldots,t_N) = \phi^1_{t_1}\circ \cdots \circ \phi^N_{t_N}(f(s))\in X,
	\qquad s\in S,\ t\in\C^N.
\end{equation}
In general, globally defined dominating sprays with a given core do not exist
unless $S$ is an Oka manifold. However,  for every holomorphic  
map $f\colon S\to X$ from a Stein manifold $S$ to an arbitrary complex manifold $X$
and for any compact subset $K\subset S$ there exist a Stein neighborhood $U\Subset S$ of $K$ and a 
dominating spray $F\colon U\times V\to X$ of the form \eqref{eq:flowspray-maps} with $F(\cdotp, 0)=f|_U$,
where $V$ is a neighborhood of the origin in some $\C^N$. Such $F$ is obtained by composing flows of 
(not necessarily complete) holomorphic vector fields on $X$ defined on a 
neighborhood  $\Omega \subset S\times X$ of the graph 
$	
	G_f(U) = \{(s,f(s)): s\in U\} \subset S\times X
$
of $f|_U$. Note that $G_f(U)$ admits an open Stein neighborhood in $S\times X$ 
by Siu's theorem (see \cite[Theorem 3.1.1]{Forstneric2017E}), and the rest
follows from Cartan's Theorem A on Stein manifolds.

%
%   PERIOD DOMINATING SPRAYS AND MAPS INTO THE NULL QUADRIC
%
\subsection{Period dominating sprays of maps into the null quadric}\label{ss:period-dominating}

Let $M$ be a compact connected bordered Riemann surface with boundary $bM$. 
Denote by $g\ge 0$ the genus of $M$ and by $m\ge 1$ the number of its boundary components;
hence $H_1(M;\Z)\cong \Z^l$ with $l=2g+m-1$.
We may assume that $M$ is a smoothly bounded domain in an open Riemann surface $R$.
For a fixed  choice of a nowhere vanishing holomorphic $1$-form $\theta$ on $R$ and of a
basis $\{C_j\}_{j=1}^l$ of $H_1(M;\Z)$ we let 
\begin{equation}\label{eq:P}
	\Pcal=(\Pcal_1,\ldots, \Pcal_l) : \Ascr(M,\C^n)\to (\C^n)^l = \C^{ln}
\end{equation}
be the {\em period map} whose $j$-th component equals 
\begin{equation}\label{eq:Pj}
	\Pcal_j(f) = \oint_{C_j} f\theta \in \C^n,\qquad  f\in \Ascr(M,\C^n). 
\end{equation}
Note that the holomorphic 1-form $f\theta$ on $M$ is exact if and only if $\Pcal(f) =0$;
this condition is clearly independent of the choice of a period basis.

Recall that $\Agot_*$ denotes the punctured null quadric \eqref{eq:null*}. The following lemma 
(see \cite[Lemma 5.1]{AlarconForstneric2014IM} and \cite[Lemma 3.2]{AlarconForstnericLopez2016MZ}) 
provides one of our main technical tools. 

%
%   EXISTENCE OF PERIOD DOMINATING SPRAYS
%
\begin{lemma}
\label{lem:deformation} 
Given a nonflat map $f\in \Ascr(M,\Agot_*)$ (see Definition \ref{def:nondegenerate}),  
there exist an open neighborhood $V$ of the origin in $\C^{ln}$ and a 
map $\Phi_f\colon M\times V \to \Agot_*$ of class $\Acal(M\times V,\Agot_*)$ such that $\Phi_f(\cdotp,0)=f$ and 
\begin{equation}\label{eq:period-domination}
	\frac{\di}{\di t}\bigg|_{t=0} \Pcal(\Phi_f(\cdotp,t)) : (\C^n)^l \to (\C^n)^l 
	\quad \text{is an isomorphism}.
\end{equation}
Furthermore, given a finite set $P\subset M$ and $r\in \N$, we may choose $\Phi_f$ such that 
\begin{equation}\label{eq:agree}
\text{$\Phi_f(\cdotp,t)\colon M\to\Agot_*$ agrees with $f$ to order $r$ at each 
$p\in P$ for all $t\in V$.}
\end{equation}
There is a neighborhood $\Omega_f$ of $f$ in $\Ascr(M,\Agot_*)$ such that the map $\Omega_f \ni g \mapsto \Phi_{g}$ 
depends holomorphically on $g$.
\end{lemma}

A map $\Phi_f$ satisfying Lemma \ref{lem:deformation}  is called a 
{\em period dominating spray} of maps $M\to \Agot_*$ with the core $\Phi_f(\cdotp,0)=f$.

\begin{proof}
We first consider the case without paying attention to \eqref{eq:agree}; the 
modification to ensure this matching condition will be explained at the end.

Let $C_1,\ldots,C_l\subset \mathring M$ be smooth oriented Jordan curves 
providing a homology basis for $H_1(M;\Z)$ and
such that $C=\bigcup_{j=1}^l C_j$ is Runge in $M$ (see Sect.\ \ref{ss:Riemann}).  
We may assume that the curves $C_i$ have a single common point $p_0\in M$,
i.e., $C_i\cap C_j=\{p_0\}$ for any $i\ne j$.
Let $\Pcal=(\Pcal_1,\ldots,\Pcal_l)$ be the associated period map \eqref{eq:P}, \eqref{eq:Pj}.
Since $f$ is nonflat,  Lemma \ref{lem:nonflat} and the identity principle show that for every $j=1,\ldots, l$ 
there are points $p_{j,k}\in C_j\setminus \{p_0\}$ and holomorphic vector fields 
$V_{j,k}$  ($k=1,\ldots, n$) on $\C^n$, tangent to the null quadric $\Agot$, such that 
\begin{equation}\label{eq:span}
	\span \bigl\{V_{j,k}(x_{j,k}): k=1,\ldots,n\bigr\}=\C^n\quad \text{where}\ x_{j,k}=f(p_{j,k}).
\end{equation}
Let $\phi^{j,k}_t$ denote the local holomorphic flow of $V_{j,k}$ for a complex time variable $t$. 
Write $t=(t_1,\ldots,t_l)\in (\C^{n})^l$ where $t_j=(t_{j,1},\ldots,t_{j,n})\in\C^n$. 
For every $j=1,\ldots,l$ and $k=1,\ldots,n$ we pick a smooth function $h_{j,k} \colon C\to \C$, supported
on a short arc in $C_j$ around the point $p_{j,k}\in C_j$, and consider the map 
\begin{equation}\label{eq:Phi}
	\Phi(p,t)=\phi_{h_{1,1}(p)t_{1,1}}^{1,1} \circ \cdots \circ \phi_{h_{l,n}(p)t_{l,n}}^{l,n}(f(p)) \in\Agot_*,\qquad p\in C.
\end{equation}
(We take the composition of flows $\phi_{h_{j,k}(p)t_{j,k}}^{j,k}$ for all $j=1,\ldots,l$ and $k=1,\ldots,n$.)
Note that $\Phi(\cdotp,0)=f$, $\Phi$ is well defined for all $t\in \C^{ln}$ sufficiently close to the origin, and 
it has range in $\Agot_*$. Clearly we have that
\[
	\frac{\di \Phi(p,t) }{\di t_{j,k}}\bigg|_{t=0} = h_{j,k}(p) V_{j,k}(f(p)), \qquad p\in C,
\]
and hence
\[
	\frac{\di \Pcal_i(\Phi(\cdotp,t))}{\di t_{j,k}}\bigg|_{t=0}  = \oint_{C_i} h_{j,k} (V_{j,k}\circ f) \theta.
\]
A suitable choice of the functions $h_{j,k}$ ensures that the above expression is as close as desired
to $V_{j,k}(x_{j,k})$ if $i=j$, and it equals zero otherwise. 
In view of \eqref{eq:span} it follows that the differential 
$\frac{\di}{\di t}|_{t=0} \Pcal(\Phi(\cdotp,t))\colon \C^{ln}\to  \C^{ln}$
has a block structure with vanishing off-diagonal $n\times n$ blocks
and with invertible diagonal blocks; hence it is invertible. By Mergelyan's theorem 
we can approximate each function $h_{j,k}$ uniformly on $C$ by a holomorphic
function $\tilde h_{j,k}\in \Ocal(M)$. Inserting these new functions into 
the definition of $\Phi$ \eqref{eq:Phi} we obtain a spray $\Phi_f$
of maps $M\to \Agot_*$ satisfying the conclusion of the lemma. 
Indeed, \eqref{eq:period-domination} holds provided that the approximation 
of the functions $h_{j,k}$ by $\tilde h_{j,k}$ is close enough, and the other properties are obvious.

In order to ensure the condition \eqref{eq:agree}, we choose the curves $C_i$  in the homology basis
so that they do not intersect the finite set $P$. Choose a  funtion 
$g\in\Ocal(M)$ that vanishes to order $r+1$ at each of the points in $P$ and has no other zeros.
We replace each of the functions $h_{j,k}$ in the spray \eqref{eq:Phi} by the product
$g h_{j,k}$. Proceeding as before, we obtain a new spray of the same type with 
$h_{j,k}\in\Ocal (M)$. It is elementary to see that the map 
$(p,t) \to \phi^{j,k}_{g(p)h_{j,k}(p) t}(f(p))$ is tangent to $f$ to order $r$ 
at every point $p\in P$ (see \cite[Lemma 2.2]{AlarconCastro-Infantes2017}); 
hence the same holds for their composition $\Phi_f$. 
\end{proof}

\begin{remark}\label{rem:PD1}
By using additional flows in the definition of 
$\Phi_f$ \eqref{eq:Phi} we can ensure that the spray $\Phi_{f_q}$ is 
period dominating for a given continuous family $\{f_q\colon q\in Q\}$ of holomorphic maps 
$f_q\colon M\to\Agot_*$ with the parameter in a compact Hausdorff space $Q$. 
In this case, condition \eqref{eq:period-domination} is replaced by asking
that the $t$-differential of the period map  is surjective at $t=0$.
On the other hand, we are unable to find a period dominating spray
whose core is a flat map since the tangent spaces to $\Agot_*$ are constant along a complex ray of $\Agot_*$,
and hence they do not span $\C^n$.
\end{remark}

\begin{remark}\label{rem:PD2}
Proofs of Lemmas \ref{lem:deformation} and \ref{lem:existence-sprays}
extend in an obvious way to the case when $M=K\cup \Gamma$ is an admissible set in an open Riemann surface $R$; 
see Definition \ref{def:admissible}. A map $f\colon M\to\Agot_*$ of class $\Ascr(M)$ 
is said to be nonflat or full if the restriction of $f$ to $K$ and to each connected component of $\Gamma$
is nonflat or full, respectively. Such $f$ typically arises as the derivative map 
$f=2\di\Phi/\theta\colon M\to\Agot_*$ of a generalized conformal minimal immersion on an admissible set;
see Definition \ref{def:GCMI}. 
\end{remark}

The following result (see \cite[Theorem 3.1]{AlarconForstnericLopez2016MZ}) is a straightforward 
application of Lemma \ref{lem:deformation}. The notation has been established in Section \ref{ss:spaces}.

%
%
%   BANACH MANIFOLD STRUCTURE THEOREMS
%
%
\begin{theorem} \label{th:structure}
Let $M$ be a compact bordered Riemann surface with nonempty boundary $bM$,
and let $n\ge 3$ and $r\ge 1$ be integers.  Then the following hold.
\begin{itemize}
\item[\rm (a)] 
	The space $\CMI_{\rm nf}^r(M,\r^n)$ is a real analytic Banach manifold. 
\vspace{1mm}
\item[\rm (b)]
 	The space $\NC_{\rm nf}^r(M,\C^n)$ is a complex Banach manifold. 
\end{itemize}
\end{theorem}

We do not know whether the spaces $\CMI^r(M,\r^n)$ and $\NC^r(M,\C^n)$ are also 
Banach manifolds. In fact, it seems that  flat conformal minimal immersions and 
holomorphic null curves are singular points of these spaces.

\begin{proof} 
By \cite[Theorem 1.1]{Forstneric2007AJM} % (see also \cite[Theorem 8.13.1]{Forstneric2017E})
the space $\Ascr^{r-1}(M,\Agot_*)$ is a complex Banach manifold modeled on the complex Banach space 
$\Ascr^{r-1}(M,\c^{n-1})$, where $\dim \Agot_*=n-1$. 
Let $\Pcal \colon \Ascr^{r-1}(M,\c^n)\to (\c^n)^l$ denote the holomorphic period map (\ref{eq:P}). Set 
\[
	\Ascr_0^{r-1}(M,\Agot_*)=\bigl\{f\in \Ascr^{r-1}(M,\Agot_*): \Re(\Pcal(f))=0\bigr\},
\]
and let $\Ascr_{0,{\mathrm{nf}}}^{r-1}(M,\Agot_*)$ denote the open subset of $\Ascr_0^{r-1}(M,\Agot_*)$ 
consisting of all nonflat maps (see Definition \ref{def:nondegenerate}). 
Lemma \ref{lem:deformation} implies that the differential $d\Pcal_{f_0}$ at any point 
$f_0\in \Ascr_{0,{\mathrm{nf}}}^{r-1}(M,\Agot_*)$ has maximal rank equal to $ln$. 
By the implicit function theorem, $f_0$ admits an open neighborhood 
$\Omega \subset \Ascr^{r-1}(M,\Agot_*)$ such that 
$\Omega\cap \Ascr_0^{r-1}(M,\Agot_*)=\Omega \cap \Ascr_{0,{\mathrm{nf}}}^{r-1}(M,\Agot_*)$ 
is a real analytic Banach submanifold of $\Omega$ parametrized by the kernel of the 
real part $\Re(d\Pcal_{f_0})$ of the differential of $\Pcal$ at $f_0$;
this is a real codimension $ln$ subspace of the complex Banach space $\Ascr^{r-1}(M,\c^{n-1})$
(the tangent space of the complex Banach manifold $\Ascr^{r-1}(M,\Agot_*)$).
This shows that $\Ascr_{0,{\mathrm{nf}}}^{r-1}(M,\Agot^*)$ is a real analytic Banach manifold. 
The integration $p\mapsto v+\int_{p_0}^p \Re(f\theta)$ $(p\in M)$, with an arbitrary choice 
of the value $v\in \r^n$ at a base point $p_0\in M$, provides an isomorphism between the 
Banach manifold  $\Ascr_{0,{\mathrm{nf}}}^{r-1}(M,\Agot^*) \times \r^n$  and 
$\CMI^r_{\mathrm{nf}}(M,\r^n)$, so the latter is also a real analytic Banach manifold. This proves (a). 
Essentially the same argument applies in case (b).
\end{proof}

We now give another version of Lemma \ref{lem:deformation} in which a period dominating
spray is obtained by multiplying the given core map $f\colon M\to\Agot_*$ by a 
nonvanishing holomorphic function (a {\em multiplier}). 

A path $f\colon I=[0,1]\to\c^n$ is said to be {\em full} if the $\c$-linear span of its image equals $\c^n$.
Let $\Pcal\colon \Cscr(I,\c^n)\to\c^n$ denote the map
\[
        \Pcal(f)=\int_0^1 f(s)\, ds\in\c^n,\qquad f\in \Cscr(I,\c^n).
\]
The following result is \cite[Lemma 2.1]{AlarconForstnericLopez2017JGEA}.

%
%   EXISTENCE OF PERIOD DOMINATING MULTIPLIERS
%
\begin{lemma}\label{lem:spray-loops}
Assume that $I'$ is a nontrivial closed subinterval of $I=[0,1]$ and $Q$ is a compact Hausdorff space. 
Given a continuous map $f\colon  I\times Q \to\c^n$ such that $f(\cdot,q)$ is full on $I'$ for every $q\in Q$, 
there exist finitely many continuous functions $g_1,\ldots,g_N\colon I\to\c$,  
supported on $I'$, such that the function $h\colon I\times \c^N  \to\c$ given by
\begin{equation}\label{eq:h-spray}
             h(s,t)=1 + \sum_{i=1}^N t_i g_i(s),\qquad s\in I,\ t=(t_1,\ldots,t_N)\in\c^N,
\end{equation}
is a {\em period dominating multiplier of $f$}, in the sense that
\begin{equation}\label{eq:period-domination1}
      \frac{\di}{\di t} \Pcal (h(\cdot,t)f(\cdot,q)) \big|_{t=0} : \c^N\to\c^n
      \ \ \text{is surjective for every $q\in Q$.}
\end{equation}
\end{lemma}

%\begin{remark}\label{rem:stability}
%Note that \eqref{eq:period-domination} is an open condition which remains valid with the same function $h$ 
%if we replace $f$ by any $f' \in \Cscr(Q\times I, \c^n)$ sufficiently close to $f$.
%\end{remark}

\begin{proof}
Let $N\ge n$ be an integer and, for each $i\in\{1,\ldots,N\}$, let $g_i\colon I\to\c$ be a continuous function; 
both the number $N$ and the functions $g_i$ will be specified later. 
Let $h$ be defined by \eqref{eq:h-spray}. Note that
$\frac{\di h(s,t)}{\di t_i} \big|_{t=0} =  g_i(s)$ and hence
\begin{equation}\label{eq:diQ}
     \frac{\di}{\di t_i} \Pcal(h(\cdot,t)f(\cdot.q)) \bigg|_{t=0} 
     =   \int_0^1 \left. \frac{\di h(s,t)}{\di t_i}  \right|_{t=0} f(s,q)\, ds  
     =   \int_0^1  g_i(s) f(s,q)\, ds.
\end{equation}
Since $f(\cdot,q)$ is full on $I'$ for every $q\in Q$, compactness of $Q$ and continuity of $f$ ensure that 
there are points $s_1,\ldots,s_N\in \mathring I'$ for a big $N\in\N$ such that
\begin{equation}\label{eq:span-N}
       \span\bigl\{ f(s_1,q),\ldots, f(s_N,q) \bigr\}=\c^n\quad \text{for all $q\in Q$}.
\end{equation}
Pick a small $\epsilon>0$ and for every $i=1,\ldots, N$ a continuous function $g_i\colon I\to\c$ supported on 
$(s_i-\epsilon,s_i+\epsilon)\subset I$ such that 
\begin{equation}\label{eq:intg1}
  	 \int_0^1 g_i(s)\, ds=\int_{s_i-\epsilon}^{s_i+\epsilon} g_i(s)\, ds=1.
\end{equation}
For small $\epsilon>0$ we have in view of \eqref{eq:diQ} and \eqref{eq:intg1} that
\[
    \frac{\di \Pcal(h(\cdot,t)f(\cdot.q))}{\di t_i} \bigg|_{t=0} =  \int_0^1 g_i(s) f(s,q)\, ds \approx f(s_i,q)
\]
for all $q\in Q$ and $i\in\{1,\ldots,N\}$.
Assuming as we may that the approximations are close enough, it follows from \eqref{eq:span-N} that
\eqref{eq:period-domination1} holds. 
\end{proof}

By using Lemma \ref{lem:spray-loops} we easily obtain the following result
which is essentially \cite[Lemma 3.2]{AlarconForstnericLopez2017JGEA}.
In this lemma, $\Pcal\colon \Ascr(M,\C^n)\to (\C^n)^l$ 
again denotes the period map \eqref{eq:P}, \eqref{eq:Pj}.
The same result holds if $M$ is a compact admissible subset of an open Riemann surface;
see Definition \ref{def:admissible} and Remark \ref{rem:PD2}.

%
%   PERIOD DOMINATING SPRAYS OF MULTIPLIERS
%
\begin{lemma}\label{lem:existence-sprays}
Let $M$ be a compact bordered Riemann surface with $H_1(M;\z)=\Z^l$,
let $\theta$ be a nowhere vanishing holomorphic $1$-form on $M$, and let $Q$ be a compact Hausdorff space. 
Assume that $f\colon M\times Q\to \c^n$ is a continuous map such that $f(\cdot,q)\colon M\to\C^n$ 
is a full map of class $\Ascr(M)$ for every $q\in Q$. 
Then there exist finitely many holomorphic functions $g_1,\ldots,g_N\in \Ocal(M)$ 
such that the function $h\colon M\times \c^N \to\c$ given by
\[
      h(p,t)= 1 + \sum_{i=1}^N t_i g_i(p),\qquad t=(t_1,\ldots,t_N)\in\c^N,\ p\in M,
\]
is a {\em period dominating multiplier of $f$}, meaning that 
\begin{equation}\label{eq:period-domination2}
      \frac{\di}{\di t} \Pcal (h(\cdot,t)f(\cdot,q)) \big|_{t=0} : \c^N \to (\c^n)^l
      \ \ \text{is surjective for every $q\in Q$.}
\end{equation}
\end{lemma}

%
%   CURVES WITH PRESCRIBED PERIODS IN THE NULL QUADRIC
%
\subsection{Paths with prescribed periods in the null quadric}
\label{ss:loops}

In this section we present a construction of paths with prescribed integrals in 
the punctured null quadric. An elementary result concerning a single path is
\cite[Lemma 7.3]{AlarconForstneric2014IM}. The parametric version  
(see \cite[Lemma 3.1]{ForstnericLarussonCAG}) is needed in the investigation of the homotopy structure 
of the spaces $\CMI(M,\R^n)$ and $\NC(M,\C^n)$; see Sect.\ \ref{ss:rough}.
Here we present a $1$-parametric version, \cite[Lemma 2.3]{AlarconForstnericLopez2017JGEA},
which has the advantage of preserving the Gauss map, so it can
be used to construct conformal minimal immersions with prescribed
complex Gauss map (see Sect.\ \ref{ss:Gauss}).

\begin{lemma}\label{lem:periods}
Set $I=[0,1]$. Let $\alpha\colon I\to \c^n$ and $f\colon I\times I \to \c^n$ be continuous maps
such that the path $f_t:= f(\cdotp,t) \colon I\to\c^n$ is full for every $t\in I$. 
Then there exists a continuous function $h\colon I\times I\to \c_*$ such that 
$h(s,t)=1$ for $t\in I$ and $s\in  \{0,1\}$ and
\begin{equation}\label{eq:exact} 
	\int_0^1 h(s,t) f(s,t)\, ds = \alpha(t), \qquad t\in[0,1].
\end{equation}
If in addition we have that $\int_0^1 f(s,0)\, ds= \alpha(0)$,  then $h$ can be chosen such that
$h(s,0)=1$ for all $s\in [0,1]$.
\end{lemma}

\begin{remark}
If the map $f$ in Lemma \ref{lem:periods} has range in the punctured null quadric $\Agot_*$, then the same
holds for the map $hf$ for any nowhere vanishing function $h$. 
This is how the lemma is used in the present paper.
The analogous conclusion holds when $f$ has range in any conical complex subvariety of $\C^n$.
\end{remark}

\begin{proof}
We begin by explaining a reduction to the case when the exact condition \eqref{eq:exact} is replaced by 
an approximate condition
\begin{equation}\label{eq:approximate}
	\left| \int_0^1 h(s,t) f(s,t)\, ds - \alpha(t)\right| <\epsilon, \qquad t\in[0,1],
\end{equation}
where $\epsilon>0$ is any given number. Indeed, since the path $f_t$ is full for every $t\in I$, 
we can divide the $t$-interval $I$  into finitely many subintervals $I_1,\ldots, I_m$ such that for every $i=1,\ldots,m$ 
there is a closed subinterval $J_i\subset I$ such that the restricted path $f_t\colon J_i\to\C^n$ is full 
for every $t\in I_i$. Clearly it suffices to consider the problem separately on each $I_i$.
Hence, replacing $I$ by $I_i$, we may assume that there is a closed subinterval $J \subset I$
such that $f_t\colon J \to\C^n$ is full for every $t\in I$. Choose nontrivial disjoint subintervals 
$J_1,J_2\subset J$. Replacing the $s$-interval $I$ by $J_1$, it suffices  to prove that for 
any given $\epsilon>0$ there is a function $h\colon I\times I\to \c_*$ satisfying \eqref{eq:approximate}.
Choosing $\epsilon$ small enough, we can correct the small error and obtain \eqref{eq:exact} 
by applying the period dominating argument, furnished by Lemma \ref{lem:spray-loops}, 
on the subinterval $J_2$. 

It remains to explain the construction of a function $h$ satisfying \eqref{eq:approximate}. 
Since $f_t$ is full for each $t\in I$, there is a division $0=s_0<s_1<\cdots <s_N=1$ of $I$ such that 
\[
	\span\{f_t(s_1),\ldots, f_t(s_N)\}=\c^n\quad \text{for all}\ t\in I.
\]
Set 
\[
	V_j(t) = \int_{s_{j-1}}^{s_j} f_t(s)\, ds,\qquad  j=1,\ldots,N. 
\]
Note that $V_j(t)$ is close to $f_t(s_j)(s_j-s_{j-1})$ if the intervals $[s_{j-1},s_j]$ are short.
Passing to a finer division if necessary we may therefore assume that
\[
	\span\bigl\{V_1(t),\ldots, V_N(t)\bigr\} =\c^n,\qquad t\in I. 
\]
For each $t\in I$ we let $\Sigma_t\subset \c^N$ denote the affine complex hyperplane defined by 
\[
	\Sigma_t= \biggl\{(g_1,\ldots,g_N) \in \c^N : \sum_{j=1}^N g_j V_j(t) = \alpha(t) \biggr\}.
\]
Clearly, there exists a continuous map $g=(g_1,\ldots, g_N)\colon I\to \c^N$ such that 
$g(t)\in \Sigma_t$ for every $t\in I$. (We may view $g$ as a section of the affine bundle
over $I$ whose fiber over the point $t$ equals $\Sigma_t$.) This can be written as follows:
\begin{equation}\label{eq:exact2}
	\sum_{j=1}^N \int_{s_{j-1}}^{s_j} g_j(t) f_t(s)\, ds = \alpha(t),\qquad t\in I.
\end{equation}
Note that $\sum_{j=1}^N V_j(t) = \int_0^1 f_t(s)\, ds$. Hence, if $\int_0^1 f(0,s)\, ds= \alpha(0)$ 
then $g$ can be chosen such that $g(0)=(1,\ldots,1)\in \c^N$.
We assume in the sequel that this holds since the proof is even simpler otherwise.

By a small perturbation we may assume that $g_j(t)\in \c_*$ for every $t\in I$ and $j=1,\ldots, N$.
% (At this point we need that the parameter space $I$ is one-dimensional.) 
This changes the exact condition \eqref{eq:exact2} to the approximate condition 
\begin{equation}\label{eq:epsilon2}
	\left| \, \sum_{j=1}^N \int_{s_{j-1}}^{s_j} g_j(t) f_t(s)\, ds - \alpha(t)\right| < \frac{\epsilon}{2}, \qquad t\in I.
\end{equation}
For a fixed $t\in I$ we consider the vector $g(t)=(g_j(t))_j\in \C^N$ as a step function of 
$s\in I$ which equals the constant $g_j(t)$ on $s\in [s_{j-1},s_{j})$ for every $j=1,\ldots,N$. 
We now approximate this step function by a continuous function $h_t=h(\cdotp,t)\colon I\to\c_*$ which agrees 
with the step function, except near the points $s_0,s_1,\ldots,s_{N}$, ensuring also that 
$h_t(0)=h_t(1)=1$. Here are the details. Let $C>1$ be chosen such that 
\[
	\max_{(s,t)\in I\times I} |f(s,t)|\le C, \qquad \max_{t\in I,\, j=1,\ldots,N} |g_j(t)| \le C.
\]
Pick a number $\eta>0$ such that 
\begin{equation}\label{eq:eta}
	4C(C+1) N \eta < \epsilon.
\end{equation}
For each $t\in I$ and $j=1,\ldots N$ we define the function $h(\cdotp,t)\colon [s_{j-1},s_j]  \to \c_*$ by
\[
	h(s,t) = \begin{cases} 
			g_{j}((s-s_{j-1})t/\eta),  & s\in [s_{j-1},s_{j-1}+\eta]; \\
			g_j(t),                           & s\in [s_{j-1}+\eta,s_j-\eta]; \\
			g_{j}((s_j-s)t/\eta),       & s\in [s_{j}-\eta,s_j].
                    \end{cases} 
\]
Thus, $h(s,t)$ spends most of its time (for $s\in [s_{j-1}+\eta,s_j-\eta]$) at the point $g_j(t)$, 
and it travels between the point $1\in \c_*$ (where it is at the endpoints $s=s_{j-1}$ and $s=s_j$) 
and the point $g_j(t)$ along the trace of the path $\tau\mapsto g_{j}(\tau t)\in\c_*$.
This defines a continuous function $h\colon I\times I \to \c_*$ satisfying 
\begin{equation}\label{eq:hleC}
	|h(s,t)|\le C\quad \text{for all $(s,t)\in I\times I$.}
\end{equation}
It follows easily from \eqref{eq:epsilon2}, \eqref{eq:eta}, and \eqref{eq:hleC} that 
the replacement of the step function % $g(t)=(g_j(t))_j$ 
by $h(s,t)$ causes an error of size $<\epsilon/2$. This yields the estimate \eqref{eq:approximate}.
\end{proof}

%
%   TRANSVERSALITY
%
\subsection{Transversality methods for conformal minimal surfaces}
\label{ss:transversality}

In this section we indicate how the techniques of Section \ref{ss:period-dominating}, especially
Lemma \ref{lem:deformation}, can be used to prove the following general position theorem
for conformal minimal immersions of bordered Riemann surfaces. 
The original reference is \cite[Theorem 4.1]{AlarconForstnericLopez2016MZ}.

%
%
%  THEOREM OF GENERAL POSITION FOR BORDERED RIEMANN SURFACES
%
%
\begin{theorem}
\label{th:desingBRS}
Let $M$ be a compact bordered Riemann surface and $r\in \N$.
Every conformal minimal immersion $X \in  \CMI^r(M,\r^n)$ for $n\ge 5$ can be approximated
arbitrarily closely in the $\Cscr^r(M)$ norm by a conformal minimal embedding 
$\wt X \in  \CMI^r(M,\r^n)$ satisfying $\Flux_{\wt X}=\Flux_X$.
If $n=4$ then $X$ can be approximated by conformal minimal immersions
with simple (transverse) double points.
\end{theorem}

Since the set of embeddings $M\to\r^n$ is clearly open in the set of immersions of class 
$\Cscr^r(M)$ for any $r\ge 1$ and $\CMI^r(M,\r^n)$ is a closed subset of the Banach 
space $\Cscr^r(M,\r^n)$ (hence a Baire space), the following corollary is immediate.

\begin{corollary}
Let $M$ be a compact bordered Riemann surface. For every pair of integers
$n\ge 5$ and $r\ge 1$  the set of conformal minimal embeddings $M\hookrightarrow \r^n$ of class $\Cscr^r(M)$ is
residual (of the second category) in the Baire space $\CMI^r(M,\r^n)$. The same holds for
the set of conformal minimal immersions $M\to\R^4$ with simple double points. 
\end{corollary}

\begin{proof}[Sketch of proof of Theorem \ref{th:desingBRS}]
We may assume that $M$ is a smoothly bounded domain in an open Riemann surface $R$ and $X$ is a nonflat 
conformal minimal immersion in an open neighborhood of $M$ in $R$. 
We associate to $X$ the {\em difference map} $\delta X\colon M\times M\to \r^n$ defined by
\[
	\delta X(p,q)=X(q)-X(p), \qquad p,q\in M.
\]
Clearly, $X$ is injective if and only if $(\delta X)^{-1}(0)= D_M:=\{(p,p): p\in M\}$. 
Since $X$ is an immersion, it is locally injective, and hence there is an open neighborhood 
$U\subset M\times M$ of the diagonal  $D_M$ such that $\delta X$ does not assume the value 
$0\in \r^n$ on $\overline U\setminus D_M$. To prove the theorem, it suffices to find arbitrarily close 
to $X$ a conformal minimal immersion $\wt X  \colon M\to\r^n$ whose difference map 
$\delta \wt X$  is transverse to the origin $0\in \r^n$ on $M\times M\setminus U$.
Since $\dim_\r M\times M=4<n$, this will imply that $\delta\wt X$ does not assume the 
value zero on $M\times M\setminus U$, so $\wt X(p)\ne \wt X(q)$ if $(p,q)\in M\times M\setminus U$. 
If $(p,q)\in U \setminus D_M$ then $\wt X(p)\ne \wt X (q)$ provided that 
$\wt X$ is close enough to $X$, so $\wt X$ is an embedding. 
To obtain such $\wt X$, we find a neighborhood $\Omega \subset \r^N$ 
of the origin in a Euclidean space and a real analytic map $H\colon \Omega \times M \to \r^n$ 
satisfying the following conditions:
\begin{itemize}
\item[\rm (a)] $H(0,\cdotp)=X$,
\vspace{1mm}
\item[\rm (b)] $H(\xi,\cdotp)\in \CMI^r(M,\R^n)$ for every $\xi \in \Omega$, and 
\vspace{1mm}
\item[\rm (c)]  the difference map $\delta H\colon \Omega \times M\times M \to \r^n$, defined by 
\[
	\delta H(\xi,p,q) = H(\xi,q)-H(\xi,p), \qquad \xi\in \Omega, \ \ p,q\in M,
\] 
is a submersive family  on $M\times M\setminus U$, in the sense that the partial differential 
\begin{equation} \label{eq:pd}
	d_\xi|_{\xi=0} \, \delta H(\xi,p,q) \colon \r^N \to \r^n
\end{equation}
is surjective for every $(p,q)\in M\times M\setminus U$. 
\end{itemize}
For the details of the construction of $H$ see \cite[Theorem 4.1]{AlarconForstnericLopez2016MZ};
one uses Lemma \ref{lem:deformation} and the implicit function theorem.
Assume now that such $H$ exists. By compactness of $M\times M \setminus U$, 
the partial differential $d_\xi (\delta H)$  (\ref{eq:pd}) is surjective for all $\xi$ in a neighborhood 
$\Omega'\subset \Omega$ of the origin in $\r^N$. 
Hence, the map $\delta H \colon M\times M\setminus U\to\r^n$ 
is transverse to any submanifold of $\r^n$, in particular, to the origin $0\in \r^n$. 
The transversality argument due to Abraham \cite{Abraham1963TAMS}
(see also \cite[Sect.\ 8.8]{Forstneric2017E}) 
shows that for a generic choice of $\xi\in\Omega'$,  the difference map 
$\delta H(\xi,\cdotp,\cdotp)$ is transverse to $0\in\r^n$ on $M\times M\setminus U$,
and hence it omits the value $0$ by dimension reasons. 
By choosing $\xi$ sufficiently close to $0\in\r^N$ we thus obtain a 
conformal minimal embedding $\wt X=H(\xi,\cdotp)\colon M \to \r^n$ close to $u$,
thereby proving the theorem. 
% 
%
\begin{comment}
For the details of the construction of $H$, see \cite[proof of Theorem 4.1]{AlarconForstnericLopez2016MZ}.
The main point is to use Lemma \ref{lem:deformation} and the implicit function theorem;
here is the main idea. Let $\theta$ be a noweher vanishing holomorphic $1$-form on $M$.
Fix a pair $(p,q)\in M\times M\setminus U$ and an embedded arc $E\subset M$ with the endpoints $p,q$.
By using Lemma \ref{lem:deformation} we find a holomorphic spray of maps $f_t \colon M\to\Agot_*$
with the core $f_0=2\di X/\theta$, depending on $t=(t',t'')\in \C^{n'}\times \C^{n''}$,  which is period dominating with
respect to $t''$ at $t=0$ and such that that the differential of the map 
$t'\mapsto \int_E \Re(f_{(t',0'')} \theta) \in \R^n$ is nondegenerate at $t=0$. 
(If  $\Re(f_{(t',0'')} \theta)$ were exact, the integral $2\int_E \Re(f_{(t',0'')} \theta)$ would equal
the difference of the values of the corresponding coformal minimal immersion at the endpoints $p,q$ of $E$.)
The implicit function theorem provides a holomorphic function $t''=\rho(t')$ near the origin of $\C^{n'}$,
with $\rho(0')=0''$, such that the form $\Re(f_{(t',\rho(t'))} \theta)$ is exact for every $t'\in\C^{n'}$
near the origin, and the spray of conformal minimal immersions 
\[
	H_{t'} = X(p_0)+2\int_{p_0}^\cdotp \Re(f_{(t',\rho(t'))} \theta)
\]
satisfies condition \eqref{eq:pd} for the given pair $(p,q)$. The composition of finitely many such sprays
satisfies the same condition for all pairs $(p,q)\in M\times M\setminus U$.
\end{comment}
\end{proof}

%%%%%%%%%%
%%%%%%%%%%
%%%%%%%%%%
%%%%%%%%%% SECTION: CONFORMAL MINIMAL IMMERSIONS
%%%%%%%%%%
%%%%%%%%%%
%%%%%%%%%%

\section{\sc Conformal minimal immersions: approximation, interpolation, embeddings, and isotopies}
\label{sec:OkaP}
At the dawn of the 21st century, not much was known about how to deform a given minimal surface 
in $\r^n$ into another one with more desirable properties. 
At that time we only counted on a few techniques which had been created ad hoc in order 
to settle specific problems. This is for instance the case of the {\em L\'opez-Ros deformation} for minimal surfaces 
$X\colon M\to \r^3$ (see \cite{LopezRos1991JDG}) which amounts to multiplying the complex Gauss map 
$g_X$ by a nowhere vanishing holomorphic function, subject to suitable period vanishing conditions
on the Weierstrass data. Its main shortcoming is that one needs the initial conformal minimal
immersion $X$ already defined everywhere on $M$; let us point out that, at that time, few open Riemann surfaces 
were known to be the underlying complex structure of a minimal surface in $\r^3$. 

The implementation of the complex analytic tools from Sections \ref{ss:Oka}--\ref{ss:transversality},
and also those to be explained in Section \ref{sec:RH}, gave rise to the birth and development of the theories of 
approximation, interpolation, and isotopies for conformal minimal immersions $M\to\r^n$ 
from any given open Riemann surface $M$,
%  This furnished powerful and versatile tools for constructing minimal surfaces in $\R^n$ with interesting global properties, thereby 
leading to an array of new results. 
In this section we discuss both the foundations of the aforementioned theories and some of their applications.
Results depending on the Riemann-Hilbert boundary value problem (see e.g.\ Theorem \ref{th:I-complete})
are treated in the following section.

%
%    RUNGE APPROXIMATION
%
\subsection{Runge approximation with jet interpolation for conformal minimal immersions}
\label{ss:Mergelyan}

The following is one of the main new tools for the construction of minimal surfaces in $\r^n$ for any $n\ge 3$
with interesting global properties and arbitrary conformal structure. It is analogous in spirit to the 
combination of the Runge approximation theorem and the Weierstrass interpolation theorem
for holomorphic maps from open Riemann surfaces to $\C^n$.

%
%    THE FIRST MAIN THEOREM ON CMI'S
%
\begin{theorem}[Runge approximation with jet interpolation for conformal minimal surfaces]\label{th:ALL}
Let $M$ be an open Riemann surface, $\Lambda\subset M$ be a closed discrete subset, 
and $K\subset M$ be a smoothly bounded compact Runge domain.
For each $p\in\Lambda$ let $\Omega_p\subset M$ be a 
neighborhood of $p$ in $M$ such that $\Omega_p\cap \Omega_q=\varnothing$ for all $p\neq q\in \Lambda$, 
and set $\Omega:=\bigcup_{p\in\Lambda}\Omega_p$. 
Given a function $r\colon \Lambda \to \N$,  every conformal minimal immersion $X\colon K\cup\Omega\to\r^n$ $(n\ge 3)$ 
can be approximated uniformly on $K$ by conformal minimal immersions $\wt X\colon M\to\r^n$ 
having a contact of order $r(p)$ with $X$ at every point $p\in\Lambda$.
\end{theorem}

\begin{remark}
In fact, more is true: the conformal minimal immersions $\wt X:M\to \R^n$ in  Theorem \ref{th:ALL}
can be chosen complete (see Theorem \ref{th:ALL-complete}); furthermore, if the map $X:\Lambda \to\R^n$ is proper 
(this holds in particular if $\Lambda$ is finite) then $\wt X$ can also be chosen proper 
(see Theorems \ref{th:SSY} and \ref{th:ALL-proper}). 
As we shall see in the proof, one can also obtain Mergelyan approximation on admissible sets
(see Definition \ref{def:admissible}). By using  the general position argument in 
Theorem \ref{th:desingBRS}, one easily sees that the  immersions 
$\wt X\in \CMI(M,\r^n)$ can be chosen embeddings if $n\ge 5$,
immersions with simple double points if $n=4$, and to have prescribed flux 
compatible with the flux of the initial immersion $X$  for any loop in $K$.
\end{remark}

\begin{remark}
The analogous Runge approximation theorems with jet interpolation holds for holomorphic null curves, and more generally
for holomorphic immersions $M\to \C^n$ directed by any conical complex subvariety 
$A\subset \C^n$ such that $A\setminus\{0\}$ is an Oka manifold (see \cite[Theorems 7.2 and 7.7]{AlarconForstneric2014IM} 
and \cite[Theorem 1.3]{AlarconCastro-Infantes2017}).
Here we say that a holomorphic immersion $Z\colon M\to\C^n$ is {\em directed by} $A$, or an {\em $A$-immersion},
if $(dZ/\theta)(M)\subset A\setminus\{0\}$, where $\theta$ is any nowhere vanishing holomorphic $1$-form on $M$.
Thus, null holomorphic immersions correspond to the case when $A=\Agot$ \eqref{eq:null}.
\end{remark}

Theorem \ref{th:ALL} is a compilation of results from the paper \cite{AlarconLopez2012JDG} by Alarc\'on and L\'opez
where the existence and approximation was proved for conformal minimal immersions into $\R^3$,
the paper \cite{AlarconForstnericLopez2016MZ} by the authors and L\'opez where the same was done
in any dimension $n\ge 3$, and the paper \cite{AlarconCastro-Infantes2017} by Alarc\'on and Castro-Infantes 
where interpolation was added. 

\begin{proof}[Sketch of the proof of Theorem \ref{th:ALL}]
We assume that the function $r\colon \Lambda\to \N$ is constant; the general case is obtained by 
an obvious modification. 

Pick a smooth strongly subharmonic Morse exhaustion function $\rho\colon M\to\r$ and  
exhaust $M$ by an increasing sequence 
\begin{equation}\label{eq:exhaustion}
	K=M_1\Subset M_2\Subset\cdots\Subset \bigcup_{i=1}^\infty M_i=M
\end{equation}
of compact smoothly bounded domains of the form $M_i=\{p\in M\colon \rho(p)\le c_i\}$,
where $c_1<c_2<\cdots$ is an increasing sequence of regular values of $\rho$
with $\lim_{i\to\infty} c_i =+\infty$. Thus, each domain $M_i$ is a possibly disconnected compact bordered Riemann surface. 
For convenience of exposition we also assume that $\rho$ has at most one critical
point $p_i$ in each difference $\mathring M_{i+1}\setminus M_i$, and that
no point of $\Lambda$ is a critical point of $\rho$. It follows that $M_{i}$ is
Runge in $M$ for every $i\in \n$. Set $\Lambda_i=\Lambda\cap M_i$ for each $i\in\n$;
this is a finite set since $\Lambda\subset M$ is closed and discrete. 
Up to enlarging $K$ and $\Lambda$ if necessary, we may assume that $\rho$ is chosen 
such that $\Lambda\cap bM_i=\varnothing$ and $\Lambda_{i+1}\setminus\Lambda_i$ consists of a single point for all $i\in\n$.

Set $X_1=X|_{M_1}$ and assume as we may that $X_1$ is nonflat. To prove the theorem, we inductively construct a sequence of nonflat
conformal minimal immersions $\{X_i\in\CMI(M_i)\}_{i\ge 2}$ satisfying the following conditions.
\begin{itemize}
\item[\rm (a)] $X_i$ is as close to $X_{i-1}$ as desired in the $\Cscr^1(M_{i-1})$ topology  for all $i\geq 2$.
\smallskip
\item[\rm (b)] $X_i$ and $X$ have a contact of order $r$ at every point in $\Lambda_i$.
\end{itemize}
It is clear that if the approximations in {\rm (a)} are close enough then the limit 
$\wt X=\lim_{i\to\infty} X_i\colon M\to\r^n$ satisfies the conclusion of the theorem. 

The basis of the induction is given by the already fixed $X_1$.
Assume that we already have the immersion $X_i$ 
for some $i\in\n$. We consider two different cases depending on the topology of $M_{i+1}\setminus M_i$.

\noindent{\em The noncritical case}: $\rho$ has no critical value in $[c_i,c_{i+1}]$. In this case
$M_i$ is a strong deformation retract of $M_{i+1}$. 
We may assume that $M_i$ is connected; otherwise we apply the same argument in each connected 
component. Set $f_i=2 \di X_i/\theta\colon M_i\to\Agot_*$, write $\Lambda_i=\{q_1,\ldots,q_k\}$, and denote by $q_0$ 
the only point in $\Lambda_{i+1}\setminus\Lambda_i$. 
Pick a point $p_0\in \mathring M_i\setminus\Lambda$ and choose a family of smooth Jordan arcs 
$\alpha_0,\alpha_1,\ldots,\alpha_k$ in $\mathring M_{i+1}$ and smooth Jordan curves 
$\alpha_{k+1},\ldots,\alpha_{k+l}$  in $\mathring M_i$ $(l=\dim H_1(M_i;\z))$ satisfying the following conditions.
\begin{itemize}
\item $\alpha_a\cap\alpha_b=\{p_0\}$ for all $a\neq b\in \{0,\ldots,k+l\}$.
\item The endpoints of $\alpha_a$ are $p_0$ and $q_a$ for all $a\in\{0,\ldots,k\}$. 
We orient each $\alpha_a$ so that $p_0$ is its initial point and $q_a$ is its final point.
\item The curves $\alpha_{k+1},\ldots,\alpha_{k+l}$ determine a homology basis of $M_i$.
\item $\Upsilon=\bigcup_{a=0}^{k+l} \alpha_a$ is a Runge set in $M$.
\item The set $S=M_i\cup \Upsilon=M_i\cup \Gamma$, where $\Gamma=\bigcup_{a=0}^k \alpha_a$, is admissible in $M$ 
(see Definition \ref{def:admissible}).
\end{itemize}

By Lemma \ref{lem:periods}  we can extend $X_i\colon M_i\to\r^n$ to a generalized conformal minimal immersion 
$(\wt X_i,f_i\theta)\colon S\to\r^n$ (see Definition \ref{def:GCMI}) 
such that $\wt X_i=X$ on $\Lambda_{i+1}$ and on a neighborhood of $q_0$; this is possible by condition {\rm (b)} 
for the index $i$. (Here, $\theta$ is a nowhere vanishing holomorphic $1$-form on $M$.) Consider the period map
\[
	\Pcal(f)=\left(\int_{\alpha_a}f\theta\right)_{a=0}^{k+l},\qquad f\in \Acal(S,\C^n).
\]
Lemma \ref{lem:deformation} and Remark \ref{rem:PD2} furnish a period dominating spray of maps 
$f_{i;w}\colon S\to\Agot_*$ of class $\Ascr(S)$ with core $f_{i;0}=f_i$,  
depending holomorphically on a parameter $w$ in a ball $B\subset \C^N$ for some $N\in\n$,
such that $f_{i;w}$ and $f_i$ have a contact of order $r$ at every point in $\Lambda_{i+1}$. 
Since $\Agot_*$ is a complex homogeneous manifold and $S$ is Runge in $M$ and a deformation retract of 
$M_{i+1}$, we may apply Theorem \ref{th:OP} to approximate $f_{i;w}$ uniformly on $M_i$ and uniformly 
with respect to $w\in B$ (shrinking $B$ slightly if necessary) by a holomorphic spray of holomorphic maps 
$g_w\colon M_{i+1}\to\Agot_*$ having a contact of order $r$ with $f_i$ at every point in $\Lambda_{i+1}$. 
Assuming that the approximation is close enough, the period domination condition of 
$f_{i;w}$ and the implicit function theorem give a point $w_0\in B$ close to $0\in\c^N$ such that
$\Pcal(g_{w_0})=\Pcal(f_i)$. The conformal minimal immersion 
\[
	X_{i+1}(p)=X_i(p_0)+  \int_{p_0}^p \Re(g_{w_0}\theta),\qquad p\in M_{i+1},
\]
then satisfies conditions {\rm (a)} and {\rm (b)} for the index $i+1$.

\noindent{\em The critical case}: $\rho$ has a unique (Morse) critical point $p_{i+1}\in M_{i+1}\setminus M_i$. 
Since $\rho$ is strongly subharmonic, $p_{i+1}$ has Morse index either $0$ or $1$. 

If the Morse index is $0$, a new simply connected component of the sublevel set $\{\rho\leq c\}$ 
appears at $p_{i+1}$ when $c$ passes the value $\rho(p_{i+1})$. We define $X_{i+1}$ on this new component 
as any conformal minimal immersion, thereby reducing the proof to the noncritical case.

If the Morse index of $p_{i+1}$ is $1$, the change of topology at $p_{i+1}$ is described by attaching to $M_i$ a smooth
arc $E\subset \mathring M_{i+1}\setminus (M_i\cup\Lambda)$ such that $M_i\cup E$ is a compact admissible Runge set 
(see Definition \ref{def:admissible}) which is a strong deformation retract of $M_{i+1}$. 
Let $\theta$ be a nowhere vanishing 
holomorphic $1$-form on $M$. Consider the smooth map $f_i=\di X_i/\theta:M_i\to \Agot_*$ which is 
holomorpic in $\mathring M_i$. We can extend $f_i$ to a smooth map $\tilde f_i \colon M_i\cup E\to\Agot_*$.
We orient $E$ and let $p,q\in bM_i$ denote the beginning and the endpoint of $E$, respectively. 
Lemma \ref{lem:periods} applied on $E$ 
furnishes a smooth function $h\colon E\to \C_*$ which equals $1$ near both endpoints such that
\[
	\int_E  h\tilde f_i\theta = X_i(q)-X_i(p).
\]
We extend $h$ to a smooth function on $M_i\cup E$ by setting $h|_M=1$. Let $\hat f_i=h\tilde f_i$.
By integrating $\hat f_i\theta$ from any initial point $p_0\in M_i$ we obtain a generalized conformal minimal immersion 
$(\widehat X_i,\hat f_i\theta)\in\GCMI(M_i\cup E,\r^n)$ (see Definition \ref{def:GCMI}) 
such that $\widehat X_i=X_i$ on $M_i$. We finish as in the noncritical case considered above,
applying the method of period dominating sprays on the admissible set $M_i\cup E$.
\end{proof}

The L\'opez-Ros deformation  \cite{LopezRos1991JDG}  for minimal surfaces in $\r^3$ enables one to perturb a given 
conformal minimal immersion by preserving one of its component functions; this is crucial in all applications of 
this technique  in the literature.  Theorem \ref{th:ALL} also admits a version in which 
all but two  components of the initial immersion are preserved. The next theorem is a compilation of results from 
\cite{AlarconCastro-Infantes2017,AlarconFernandezLopez2013CVPDE,AlarconForstnericLopez2016MZ,AlarconLopez2012JDG}. 

%
% MERGELYAN'S THEOREM WITH A FIXED COMPONENT
%
\begin{theorem}\label{th:ALL-2}
(Assumptions as in Theorem \ref{th:ALL}.) 
Assume in addition that $X=(X_1,\ldots,X_n)$ is nonflat and that the functions $X_3,\ldots,X_n$ extends harmonically to $M$.
Then the approximating conformal minimal immersions $\wt X=(\wt X_1,\ldots,\wt X_n)\in\CMI(M,\r^n)$ in Theorem \ref{th:ALL} can 
be found with $\wt X_k=X_k$ for $k=3,\ldots,n$.
\end{theorem}

%
%   COMPLETENESS
%
The following extension of Theorem \ref{th:ALL} requires some additional work.

\begin{theorem}\label{th:ALL-complete}
The conformal minimal immersions $\wt X\colon M\to\r^n$ in Theorem \ref{th:ALL} can be chosen complete.
\end{theorem}

If one ignores the interpolation, then Theorem \ref{th:ALL-complete} follows from Theorem \ref{th:SSY} to the effect 
that a conformal minimal
immersion $X\colon K\to\R^n$ for $n\ge 3$ from a Runge set $K$ in an open Riemann surface $M$ can be approximated 
by {\em proper} (hence complete) conformal minimal immersions $\wt X\colon M\to\R^n$.
Assuming in addition that $X: \Lambda \to\R^n$ is a proper map, it is also possible to match the interpolation
condition in Theorem \ref{th:ALL} by a proper conformal minimal immersions $\wt X\colon M\to\R^n$ 
(see Theorem \ref{th:ALL-proper}).

\begin{proof}[Sketch of proof]
Following the noncritical case in the proof of Theorem \ref{th:ALL}, we assume without loss of generality that $M_i$ is connected and, 
for simplicity of exposition, that $bM_i$ is connected as well; hence $A=M_{i+1}\setminus \mathring M_i$ is a smoothly 
bounded compact annulus with $bA=bM_{i+1}\cup bM_i$.
Write $X_i=(X_{i;1},\ldots,X_{i;n})$.  We split $A$ into two annuli $A_0$ and $A_1$  such that $A_0\cap A_1$ is a boundary 
component of both $A_0$ and $A_1$, $bM_i\subset bA_0$, $bM_{i+1}\subset bA_1$, and the only point $q_0$ in 
$\Lambda_{i+1}\setminus \Lambda_i$ lies in $\mathring A_0$. By the proof of Theorem \ref{th:ALL} we may assume that 
$X_i$ extends to $M_{i+1}$ having a contact of order $k$ with $X$ at every point in $\Lambda_{i+1}$, and that $\di X_{i;1}$ 
vanishes nowhere on $A_1$. We then consider a labyrinth of compact sets $\Upsilon$ in $\mathring A_1$ as in 
Jorge-Xavier \cite{JorgeXavier1980AM}, i.e., $\Upsilon$ is a finite union of pairwise disjoint compact sets in $\mathring A_1$ 
such that if $\gamma\colon[0,1]\to A_1\setminus \Upsilon$ is a path connecting the two boundary components of $A_1$ then
\begin{equation}\label{eq:diX1}
	\int_\gamma |\di X_{i;1}|>2\tau
\end{equation}
for a given number $\tau>0$. By Theorem \ref{th:ALL-2} we obtain 
$X_{i+1}=(X_{i+1;1},\ldots, X_{i+1;n})\in\CMI(M_{i+1},\r^n)$ which is close to $X_i$ in the $\Cscr^1(M_i\cup A_0)$ norm, 
has a contact of order $k$ with $X$ everywhere on $\Lambda_{i+1}\subset M_i\cup A_0$, $X_{i+1;1}=X_{i;1}$ everywhere 
on $M_{i+1}$, and $|X_{i+1;2}(p)-X_{i;2}(q)|>\tau$ for all points $p\in \Upsilon$ and $q\in A_0$. Together with \eqref{eq:metric}, 
this and \eqref{eq:diX1} guarantee that, if the approximation of $X_i$ by $X_{i+1}$ is close enough, the intrinsic distance between 
the boundaries of $A_1$ with respect to the metric induced on $M_{i+1}$ by the Euclidean metric in $\r^n$ via $X_{i+1}$ is 
greater than $\tau$. Since $\tau>0$ is arbitrary, this shows that we may arbitrarily enlarge the intrinsic diameter of the 
surface in every step of the inductive construction in the proof of Theorem \ref{th:ALL}, thereby ensuring  
completeness of the limit map. 
\end{proof}

%
% Subsection: ON SULLIVAN'S & SCHOEN-YAU'S CONJECTURES AND THE EMBEDDING PROBLEM
%
\subsection{On Sullivan's and Schoen-Yau's conjectures and the embedding problem}
\label{ss:SullivanSchoenYau}
As we have mentioned in the introduction, as late as in the 1990s hyperbolic Riemann surfaces were thought 
to play only a marginal role in the global theory of minimal surfaces as seen from the following well known conjectures.

\begin{conjecture}[Sullivan]\label{co:Sullivan}
Every properly immersed minimal surface in $\r^3$ with finite topology is parabolic.
\end{conjecture}

\begin{conjecture}[Schoen-Yau \text{\cite[p.\ 18]{SchoenYau1997CIP}}]
No hyperbolic open Riemann surface $M$ carries proper harmonic maps $M\to\r^2$. In particular,
every minimal surface in $\r^3$ with proper projection to $\R^2$ is parabolic.
\end{conjecture}

The first and more ambitious part of Schoen-Yau's conjecture was refuted in 1999 by Bo{\v z}in \cite{Bozin1999IMRN}
who constructed in a very explicit way a proper harmonic map $\d\to\r^2$. 
Another counterexample was given in 2001 by Forstneri{\v c} and Globevnik \cite{ForstnericGlobevnik2001MRL} who constructed 
a proper holomorphic map $f=(f_1,f_2)\colon\d\to\C^2$ with $f(\d)\subset (\C_*)^2$; hence, 
$(\log|f_1|,\log|f_2|)\colon \d\to\R^2$ is a proper harmonic map.  
However, the second part of the conjecture concerning minimal surfaces
remained open at that time. Sullivan's conjecture was refuted in 2003 by Morales \cite{Morales2003GAFA} 
who constructed a proper conformal minimal immersion $\d\to\r^3$ by using the L\'opez-Ros deformation 
and the Runge theorem in a highly intricate way. Morales' result was later extended to the existence of proper hyperbolic 
minimal surfaces in $\r^3$ with arbitrary topology; see Ferrer, Mart\'in and Meeks \cite{FerrerMartinMeeks2012AM}.

Finally, Alarc\'on and L\'opez \cite{AlarconLopez2012JDG} proved in 2012 that every open Riemann surface admits a 
conformal minimal immersion into $\r^3$ properly projecting to a plane; this gave a counterexample to the second part of 
Schoen-Yau's conjecture and provided an optimal solution to the two problems. The following more precise result 
in this direction is due to the authors and L\'opez (see \cite[Theorem 7.1]{AlarconForstnericLopez2016MZ}).

%
% THEOREM:  CMI's WITH PROPER PROJECTIONS TO R^2  
%
\begin{theorem}
[Conformal minimal immersions with proper projections to $\R^2$] \label{th:SSY}
Let $M$ be an open Riemann surface and $K\subset M$ be a Runge compact set. Every conformal minimal immersion 
$U\to\r^n$  $(n\ge 3)$ from an open neighborhood $U\subset M$ of $K$ 
can be approximated uniformly on $K$ by proper conformal minimal immersions 
$M\to\r^n=\r^2\times\r^{n-2}$ properly projecting into $\r^2\times\{0\}^{n-2}\cong\r^2$. The approximating immersions 
can be chosen with prescribed flux compatible with the flux of the initial immersion, with simple double points if 
$n=4$, and embeddings if $n\ge 5$.
\end{theorem}

Concerning the analogue of the Schoen-Yau conjecture in higher dimension, it was recently shown by 
Forstneri{\v c} \cite[Corollary 3.5]{ForstnericJAM} that every Stein manifold $X$ of complex dimension $n\ge 1$ admits 
a proper pluriharmonic map into $\R^{2n}$. 

Although Theorem \ref{th:SSY} contributes to the aforementioned conjectures, its main relevance concerns the 
problem of determining the minimal dimension $d$ for which every open Riemann surface properly embeds into $\r^d$ 
as a conformal minimal surface; compare with Theorem \ref{th:I-properRn} and Conjecture \ref{conj:AF} in the Introduction.

\begin{proof}[Sketch of proof of Theorem \ref{th:SSY}]
We may assume that $K$ is a smoothly bounded compact  Runge domain in $M$. Let $X_0\in\CMI(K,\r^n)$. 
Choose an exhaustion $K=M_0\Subset M_1\Subset\cdots$ of $M$ as in \eqref{eq:exhaustion} and inductively construct 
a sequence $X_i=(X_{i,1},X_{i,2},\ldots,X_{i,n})\in\CMI(M_i,\r^n)$ $(i\in\N)$ satisfying the following conditions.
\begin{itemize}
\item[\rm (a)] $\max\{X_{i,1},X_{i,2}\}> i$ everywhere on $bM_i$.
\item[\rm (b)] $\max\{X_{i,1},X_{i,2}\}> i-1$ everywhere on $M_i\setminus\mathring M_{i-1}$.
\item[\rm (c)] $X_i$ is as close to $X_{i-1}$ as desired in the $\Cscr^1(M_{i-1})$ norm.
\item[\rm (d)] $X_i$ only has simple double points if $n=4$ and is an embedding if $n\ge 5$.
\end{itemize}
(The way to prescribe the flux map is the standard one; we shall omit it.)
Clearly, if the approximation in {\rm (c)} is close enough then the limit conformal minimal immersion 
$\lim_{i\to\infty} X_i\colon M\to\r^n$ satisfies the conclusion of the theorem. 

We begin the induction with $X_0\in \CMI(K,\r^n)$ which, up to composing with a translation and 
using Theorem \ref{th:desingBRS}, satisfies (a) and (d), while (b), (c) are vacuous.

We now explain the {\em noncritical case} in the inductive step. Assume that $M_{i-1}$ is a strong deformation retract of $M_i$ 
for some $i\ge 1$ and that we already have $X_{i-1}\in\CMI(M_{i-1},\r^n)$ with the desired properties.
Note that $M_i\setminus \mathring M_{i-1}$ is union of finitely many pairwise disjoint compact annuli.
For simplicity of exposition we assume that there is only one annulus, so $A=M_i\setminus \mathring M_{i-1}$, 
since the same argument can be  applied separately to each one of them. Note that $bA=bM_i\cup bM_{i-1}$.
By Theorem \ref{th:desingBRS} it suffices to find $X_i\in\CMI(M_i,\r^n)$ satisfying {\rm (a)}, {\rm (b)}, 
and {\rm (c)}. In view of condition {\rm (a)} for the index $i-1$, $bM_{i-1}$ splits into $l\ge 3$
%
% AA: We choose $l\ge 3$ just to make that no different arcs shared their two endpoints. Perhaps $l\ge 2$ is enough, but I see it clearer this way.
%
compact subarcs $\alpha_k$, $k\in \z_l=\z/l\z$, lying end to end, for which there are complementary subsets 
$I_1$ and $I_2=\z_l\setminus I_1$ of $\z_l$ 
satisfying that $X_{i-1,\sigma}>i-1$ everywhere on $\alpha_k$ for all $k\in I_\sigma$, $\sigma=1,2$. Denote by $p_k\in bM_{i-1}$ 
the only point in $\alpha_k\cap\alpha_{k+1}$, $k\in\z_l$, and choose a family $\gamma_k$ $(k\in\z_l)$ of pairwise disjoint 
smooth Jordan arcs in $A$ such that $\gamma_k$ connects $p_k$ with a point $q_k\in bM_i$ and is otherwise disjoint 
from $bA$. We choose these arcs such that the set $S=M_{i-1}\cup \bigcup_{k\in\z_l} \gamma_k\subset M$
is admissible (see Definition \ref{def:admissible}). Denote by $\beta_k$ the Jordan arc in $bM_i$ connecting $q_{k-1}$ and 
$q_k$, and by $\Omega_k\subset A$ the closed disc bounded by $\gamma_{k-1}\cup\alpha_k\cup\gamma_k\cup\beta_k$
for $k\in\z_l$. Theorem \ref{th:ALL}, applied to a suitable generalized conformal minimal immersion on $S$ 
extending $X_{i-1}$, furnishes $Y=(Y_1,Y_2,\ldots,Y_n)\in \CMI(M_i,\r^n)$ as close as desired to $X_{i-1}$ in the $\Cscr^1(M_{i-1})$ norm 
and smoothly bounded compact discs $D_k\subset \Omega_k\setminus(\gamma_{k-1}\cup\alpha_k\cup\gamma_k)$, $k\in\z_l$,   
such that $D_k\cap\beta_k\neq\varnothing$ is a Jordan arc in $\beta_k\setminus\{q_{k-1},q_k\}$,
\begin{itemize}
\item[\rm (P1)] $Y_\sigma>i$ everywhere on $\overline{\beta_k\setminus D_k}$ for all $k\in I_\sigma$, $\sigma=1,2$, and
\item[\rm (P2)] $Y_\sigma>i-1$ everywhere on $\overline{\Omega_k\setminus D_k}$ for all $k\in I_\sigma$, $\sigma=1,2$.
\end{itemize} 
(See Figure \ref{fig:proper}.)
\vspace{-1mm}
\begin{figure}[ht]
    \begin{center}
    \scalebox{0.13}{\includegraphics{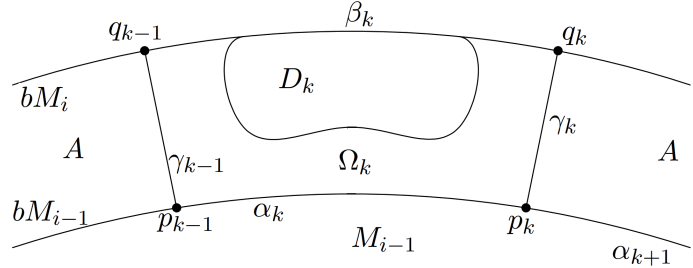}}
        \end{center}
\caption{Sets in the proof of Theorem \ref{th:SSY}}\label{fig:proper}
\end{figure}
Note that $Y$ already satisfies conditions {\rm (a)} and {\rm (b)} but only on $\bigcup_{k\in\z_l} \overline{\Omega_k\setminus D_k}$.
Now, since $Y$ is defined everywhere on $M_i$ and it may be assumed nonflat in view of Theorem \ref{th:ALL}, 
by Theorem \ref{th:ALL-2} (the Mergelyan approximation with fixed components) we may 
approximate $Y$ on $M_{i-1}\cup \bigcup_{k\in I_2}\Omega_k$ by a conformal minimal immersion 
$Y'=(Y_1',Y_2',\ldots,Y_n')\in\CMI(M_i,\r^n)$ such that
\begin{itemize}
\item[\rm (P3)] $Y_1'=Y_1$ everywhere on $M_i$, and 
\item[\rm (P4)] $Y_2'>i$ everywhere on $\bigcup_{k\in I_1} D_k$.
\end{itemize}
%
% AA: I added the following explanation
%
Indeed, it suffices to apply Theorem \ref{th:ALL-2}, keeping the first component fixed, 
with a conformal minimal immersion $\wt Y\in \CMI(M_{i-1}\cup \bigcup_{k\in I_2}\Omega_k\cup \bigcup_{k\in I_1} D_k)$ 
of the form
\[
	\wt Y=\left\{\begin{array}{ll}
	Y & \text{on $M_{i-1}\cup \bigcup_{k\in I_2}\Omega_k$,}
	\\
	(0,C,0,\ldots,0)+Y & \text{on $\bigcup_{k\in I_1} D_k$},
	\end{array}\right.
\]
where $C>0$ is a large enough constant.

Note that, by {\rm (P3)}, if the approximation of $Y$ by $Y'$ is close enough then {\rm (P1)} and {\rm (P2)} hold with $Y'$ 
in place of $Y$. Observe that $Y'$ already satisfies conditions {\rm (a)} and {\rm (b)} but only on 
$\bigcup_{k\in\z_l} \overline{\Omega_k\setminus D_k}\cup \bigcup_{k\in I_1} D_k$; we will now deform it to meet these 
requirements also on $\bigcup_{k\in I_2} D_k$ and this will finish the proof. Indeed, proceeding in a symmetric way we may 
approximate $Y'$ on $M_{i-1}\cup \bigcup_{k\in I_1}\Omega_k$ by a conformal minimal 
$Y''=(Y_1'',Y_2'',\ldots,Y_n'')\in\CMI(M_i,\r^n)$ such that
\begin{itemize}
\item[\rm (P5)] $Y_2''=Y_1'$ everywhere on $M_i$, and 
\item[\rm (P6)] $Y_2''>i$ everywhere on $\bigcup_{k\in I_2} D_k$.
\end{itemize}
As above, by {\rm (P5)}, if the approximation of $Y'$ by $Y''$ is close enough then $Y''$ formally satisfies 
{\rm (P1)}, {\rm (P2)}, and {\rm (P4)}. This and {\rm (P6)} shows that $X_i:=Y''$ meets conditions 
{\rm (a)} and {\rm (b)}. Finally, {\rm (c)} also holds 
provided that the approximations of $X_{i-1}$ by $Y$, of $Y$ by $Y'$, and of $Y'$ by $Y''$ are sufficiently close. 
This concludes the proof.
\end{proof}

By joining the ideas in the above proof with those in Theorem \ref{th:ALL} we obtain the following extension 
of Theorem \ref{th:ALL}, due to Alarc\'on and Castro-Infantes \cite{AlarconCastro-Infantes2017}.

\begin{theorem}[Theorem 1.2 in \cite{AlarconCastro-Infantes2017}] \label{th:ALL-proper}
In the assumptions of Theorem \ref{th:ALL}, if in addition $X|_\Lambda\colon \Lambda\to\r^n$ is a proper map, then 
the conformal minimal immersions $\wt X\colon M\to\r^n$ in Theorem \ref{th:ALL} can be chosen proper.
\end{theorem}

%
% Subsection: ON THE GAUSS MAP
%

\subsection{On the Gauss map}
\label{ss:OnGauss}

The Gauss map of a minimal surface in $\r^3$ and, more generally, the generalized Gauss map $G_X\colon M\to\cp^{n-1}$ 
(see \eqref{eq:GX}) of a conformal minimal immersion $X\colon M\to\r^n$ $(n\ge 3)$, is a fundamental object in the theory. 
It is classical that $G_X$ is a holomorphic map assuming values in the null quadric (see Sec.\ \ref{ss:Gauss}).
Somewhat surprisingly, the following converse was proved only very recently
by the authors and L\'opez (see \cite[Theorem 1.1 and Corollary 1.2]{AlarconForstnericLopez2017JGEA}).

\begin{theorem}\label{th:Gauss}
Let $M$ be an open Riemann surface.
For any holomorphic map $\Gscr\colon M\to Q_{n-2}\subset \cp^{n-1}$ $(n\ge 3)$ into the quadric 
\eqref{eq:nullquadric-projected} there is a conformal minimal immersion $X\colon M\to\r^n$ with
the generalized Gauss map $G_X=\Gscr$ and with vanishing flux. If in addition the map $\Gscr$ is full 
(i.e., its image is not contained in any proper projective subspace), 
then $X$ can be chosen to have arbitrary flux and 
to be an immersion with simple double points if $n=4$ and an embedding if $n\ge 5$. 

In particular, every holomorphic map $g\colon M\to\CP^1$ is the complex
Gauss map \eqref{eq:C-Gauss} of a conformal minimal immersion $X\colon M\to \R^3$ with vanishing flux.
If $g$ is nonconstant, then we can find $X$ with arbitrary given flux.
\end{theorem}

\begin{proof}[Sketch of proof]
We first apply the Oka-Grauert principle (see Theorem \ref{th:OP}) to 
lift the map $\Gscr\colon M\to \cp^{n-1}$ to a holomorphic map $G\colon M\to\c_*^n$ such that $\pi\circ G=\Gscr$, 
where $\pi\colon\c_*^n\to\cp^{n-1}$ is the canonical projection. Obviously, $G$ assumes values in the punctured null quadric 
$\Agot_*\subset\c^n$ \eqref{eq:null}, \eqref{eq:null*}. To complete the proof of the first part of the theorem, it then suffices to 
fix a nowhere vanishing holomorphic $1$-form $\theta$ on $M$ and find a holomorphic function $\varphi\colon M\to\c_*$ 
such that the real part of the $1$-form $\Phi=\varphi G\theta$ is exact on $M$. If such $\varphi$ exists then given 
$p_0\in M$ the Enneper-Weierstrass formula (Theorem \ref{th:EW}) shows that the map $X\colon M\to\r^n$ given by 
$X(p)=\int_{p_0}^p \Re(\Phi)$, $p\in M$, is a conformal minimal immersion with the generalized Gauss map 
$G_X=[\di X]=[\Phi]=\pi\circ (\varphi G)=\pi\circ G=\Gscr$.

The construction of the function $\varphi$ follows the scheme of proof of Theorem \ref{th:ALL} but using 
Lemma \ref{lem:existence-sprays} instead of Lemma \ref{lem:deformation}. Let us focus on the case of vanishing flux,
i.e., we look for $\varphi$ such that $\varphi G\theta$ is exact. Choose an exhaustion 
$M_1\Subset M_2\Subset\cdots$ of $M$ as in \eqref{eq:exhaustion} such that $M_1$ is simply connected. 
We inductively construct a sequence of holomorphic functions $\varphi_i\colon M_i\to\c_*$ $(i\in\n)$ such that
\begin{itemize}
\item[\rm (a)] $\varphi_i G\theta$ is exact on $M_i$, and 
\item[\rm (b)] $\varphi_i$ is as close to $\varphi_{i-1}$ as desired in the $\Cscr(M_{i-1})$ norm for all $i\ge 2$.
\end{itemize}
The limit function $\varphi=\lim_{i\to\infty}\varphi_i\colon M\to\c_*$ clearly meets the requirements if all approximations 
in {\rm (b)} are close enough. 
Since $M_1$ is simply connected, the basis of the induction is given by any holomorphic function $\varphi_1\colon M_1\to\c_*$. 
For the inductive step we assume that we already have $\varphi_{i-1}$ for some $i\ge 2$. For simplicity of exposition we 
assume that $M_{i-1}$ is connected and a strong deformation retract of $M_i$ (i.e., we only discuss the noncritical case). 
Lemma \ref{lem:existence-sprays} provides a period dominating multiplier $h\colon M_{i-1}\times\c^N\to\c$ of $\varphi_{i-1}$. 
Next, we approximate $\varphi_{i-1}$ and $h$ on $M_{i-1}$ by a holomorphic function $f\colon M\to\c_*$ and a spray 
of holomorphic functions $h'\colon  M_i\times\c^N\to\c$, respectively, such that $h'$ is a period dominating multiplier of $f$. If 
the approximations of $\varphi_{i-1}$ by $f$ and of $h$ by $h'$ are close enough, then there is a point 
$\zeta_0\in\c^N$ close to the origin such that the function $\varphi_i(p)=f(p) h(p,\zeta_0)$, $p\in M_i$, 
does not vanish anywhere and meets conditions {\rm (a)} and {\rm (b)}.

For the second assertion concerning the cases $n=4$ and $n\ge 5$, we adapt the transversality method described in 
Section \ref{ss:transversality} to the current framework. By using period dominating multipliers given by 
Lemma \ref{lem:existence-sprays}, we may improve Theorem \ref{th:desingBRS} by ensuring that the approximating immersion 
$\wt X$ has the same generalized Gauss map as $X$. (See \cite[Proof of Theorem 1.1]{AlarconForstnericLopez2017JGEA} 
for the details.) This enables us to find the function $\varphi_i$ in the inductive construction such that the conformal 
minimal immersion $X_i\colon M_i\to\r^n$ given by $X_i(p)=\int_{p_0}^p \Re(\varphi_i G\theta)$, $p\in M_i$, 
has simple double points if $n=4$, and is an embedding if $n\ge 5$. The same holds for $X=\lim_{i\to\infty}X_i\colon M\to\r^n$ 
provided the approximations in {\rm (b)} are sufficiently close. 
\end{proof}

The size of the spherical image of the Gauss map of a minimal surface in $\r^3$ has important implications. For instance, 
Barbosa and do Carmo \cite{BarbosaDoCarmo1976AJM} proved that if the area of the spherical image is smaller than 
$2\pi$ (half of the area of the sphere) then the surface is stable. Thus, Theorem \ref{th:Gauss} leads to the following corollary.

%
% COROLLARY ON STABILITY
%
\begin{corollary}\label{co:stable}
If $M$ is an open Riemann surface and $g\colon M\to\cp^1$ is a holomorphic map whose image $g(M)$ has 
spherical area less than $2\pi$, 
then there is a stable conformal minimal immersion $M\to\r^3$ with the complex Gauss map $g$.
\end{corollary}

Another important direction in the study of the Gauss map of conformal minimal surfaces in $\R^3$ is to 
understand how many points it can omit. A seminal result of Fujimoto \cite{Fujimoto1988JMSJ} says that the Gauss map 
of a complete nonflat minimal surface in $\r^3$ can omit at most four points of $\CP^1$;
there 
are examples with four omitted points, for instance, the classical Sherk's doubly periodic surface. 
In higher dimensions, Ru \cite{Ru1991JDG} 
proved that if $X\colon M\to\r^n$ is a complete nonflat conformal minimal immersion then its generalized Gauss map $G_X$ 
can omit at most $\frac12 n(n+1)$ hyperplanes in general position in $\cp^{n-1}$. 
(As pointed out in \cite[p.\ 280]{Fujimoto1983JMSJ}, this is equivalent to Fujimoto's theorem for $n=3$.) 
However, the number of exceptional hyperplanes depends on the complex structure of the surface.
Ahlfors  \cite{Ahlfors1941ASSF} proved that any holomorphic map $\c\to\cp^{n-1}$ avoiding $n+1$ hyperplanes 
of $\cp^{n-1}$ in general position is degenerate (for $n=2$ this is Picard theorem). 
This shows that the following result of Alarc\'on, Fern\'andez, and L\'opez \cite{AlarconFernandezLopez2012CMH,AlarconFernandezLopez2013CVPDE}
is the best possible for all minimal surfaces with nondegenerate Gauss map.

%
%   CMIS WITH GAUSS MAP OMITTING HYPERPLANES
%
\begin{theorem}\label{th:Isa}
Let $M$ be an open Riemann surface. For any group homomorphism $\pgot\colon H_1(M;\z)\to\r^n$ $(n\ge 3)$ 
there is a complete conformal minimal immersion $M\to\r^n$ with the flux map $\pgot$ whose generalized Gauss map is 
nondegenerate and omits $n$ hyperplanes of $\cp^{n-1}$ in general position.
In particular, every open Riemann surface admits a complete nonflat conformal minimal immersion 
into $\r^3$ whose complex Gauss map  omits two points of $\CP^1$.  
\end{theorem}

On the other hand, Osserman proved in 1964 \cite{Osserman1964AM} that the Gauss map of a complete nonflat minimal surface with 
finite total curvature in $\r^3$ can omit at most three points of $\CP^1$ (see also \cite[p.\ 89]{Osserman1986book}).
His question, whether there is an example of this kind whose Gauss map omits three points, is still open.

\begin{proof}[Sketch of proof]
We explain the case $n=3$ without taking care of the flux. It suffices to find a complete nonflat 
conformal minimal immersion $X=(X_1,X_2,X_3)\colon M\to\r^3$ such that $\di X_3$ does not vanish anywhere on $M$. 
Indeed,  by \eqref{eq:C-Gauss}, this implies that the complex Gauss map $g_X\colon M\to\c$ of $X$ is holomorphic and 
nowhere vanishing, and hence the Gauss map of $X$ assumes neither the north nor the south poles of $\s^2$. 
Note that Theorem \ref{th:ALL-2} (the Runge theorem with fixed components for minimal surfaces) ensures that every 
nonconstant harmonic function $X_3\colon M\to\r$ is a component function of a nonflat conformal minimal immersion 
$X=(X_1,X_2,X_3)\colon M\to\r^3$; choosing $X_3$ with no critical points we have that $\di X_3$ vanishes nowhere on $M$. 
To complete the proof, it remains to show that such an immersion $X$ may be chosen to be complete. 
As in the proof of Theorem \ref{th:ALL}, the map $(X_1,X_2)\colon M\to\r^2$ is constructed inductively: 
$(X_1,X_2)=\lim_{i\to\infty} (\wt X_{i,1},\wt X_{i,2})$ for suitable harmonic maps 
$(\wt X_{i,1},\wt X_{i,2})\colon M_i\to\r^2$. (Here $M_1\Subset M_2\Subset\cdots$ is an exhaustion of $M$ as in 
\eqref{eq:exhaustion}.) To ensure completeness of $X=\lim_{i\to\infty} \wt X_i=(\wt X_{i,1},\wt X_{i,2},X_3)\colon M\to\r^3$ 
we suitably enlarge the intrinsic diameter of each immersion $\wt X_i\colon M_i\to\r^3$ by using a Jorge-Xavier type
labyrinth  of compact sets in $\mathring M_i\setminus M_{i-1}$ as in the proof of Theorem \ref{th:ALL-complete}.
\end{proof}

%
% Subsection: ROUGH SHAPE
%

\subsection{Rough shape of the space of conformal minimal immersions}
\label{ss:rough}

Several of the results already stated in the paper may be extended to continuous families of conformal
minimal immersions by exploiting the parametric Oka property of the punctured null quadric $\Agot_*\subset\c^n$. 
For instance, for an open Riemann surface $M$, 
{\em every conformal minimal immersion $X_0\colon M\to\r^n$ is isotopic through conformal minimal immersions 
$X_t\colon M\to\r^n$ $(t\in[0,1])$ to 
\begin{enumerate}
\item[\rm (a)] a complete conformal minimal immersion \cite{AlarconForstneric2017CRELLE},
\item[\rm (b)] a complete conformal minimal immersion with arbitrary flux if the generalized Gauss map of $X_0$ is 
nondegenerate \cite{AlarconForstneric2017CRELLE}, and
\item[\rm (c)] a complete conformal minimal immersion with vanishing flux such that all maps $X_t$ 
have the same generalized Gauss map $M\to\CP^{n-1}$ 
{\em (see \cite[Corollary 1.4]{AlarconForstnericLopez2017JGEA})}.
\end{enumerate}}

Fix a nowhere vanishing holomorphic 1-form $\theta$ on $M$ and consider the following commuting diagram
of spaces and maps:
\[ 
\xymatrix{
\NC_{\mathrm{nf}}(M,\C^n)  \ar^\phi[r] \ar_{\Re}[d]  &   
\Ocal_{\mathrm{nf}}(M,\Agot_*)  \ar@{^{(}->}^i[r] &  \Ocal(M,\Agot_*) \ar@{^{(}->}^j[r]  &   \mathscr C (M,\Agot_*)  \\  
\operatorname{\Re}\NC_{\mathrm{nf}}(M,\C^n)   \ar@{^{(}->}^\alpha[r]     
&  \CMI_{\mathrm{nf}}(M,\R^n)  \ar_{\psi}[u]  \ar@{^{(}->}^\beta[r] & \CMI(M,\R^n) 
} 
\]
The left hand side map $\Re$ (the real part projection) is a homotopy equivalence 
by continuity of the conjugate map transform. Note that $\operatorname{\Re}\NC_{\mathrm{nf}}(M,\C^n)$
is the space of conformal minimal immersions $M\to\R^n$ with zero flux.
The maps $\phi$ and $\psi$ are defined by $\phi(Z)=dZ/\theta$ and  $\psi(X)=2\di X/\theta$, respectively. 
The space $\Ocal_{\mathrm{nf}}(M,\Agot_*)$ consists of all nonflat holomorphic maps $M\to\Agot_*$ 
(see Def.\ \ref{def:nondegenerate}).

Forstneri\v c and L\'arusson proved in \cite{ForstnericLarussonCAG}  that all maps in the above diagram,
with the only possible exception of the inclusion $\CMI_{\mathrm{nf}}(M,\R^n) \hra  \CMI(M,\R^n)$, are
weak homotopy equivalences, and are homotopy equivalences if $M$ has finite topological type.
(See \cite[Theorem 1.1]{ForstnericLarussonCAG} for the inclusion 
$\alpha\colon \Re \NC_{\mathrm{nf}}(M,\c^n)\hookrightarrow \CMI_{\mathrm{nf}}(M,\r^n)$, 
\cite[Theorem 1.2]{ForstnericLarussonCAG} for maps $\phi$ and $\psi$, and
\cite[Theorem 5.4]{ForstnericLarussonCAG} for the inclusion 
$i\colon \Ocal_{\mathrm{nf}}(M,\Agot_*)  \hra \Ocal(M,\Agot_*)$. 
The inclusion $\Ocal(M,\Agot_*)  \hra  \mathscr C(M,\Agot_*)$ is a weak homotopy equivalence 
by the Oka-Grauert Theorem \ref{th:OP}. For the proof of strong homotopy equivalences,  
see \cite[Sect.\ 6]{ForstnericLarussonCAG}.)  Subsequently, Alarc\'on and L\'arusson  \cite{AlarconLarusson2017IJM}  
used the methods from \cite{ForstnericLarussonCAG} to show that 
the map $\pi\colon \Ocal_{\mathrm{nf}}(M,\Agot_*)\to H^1(M;\c^n)$ 
sending a nonflat holomorphic map $g\colon M\to\Agot_*$ to the cohomology class of $g\theta$
is a Serre fibration; this also implies the aforementioned results from \cite{ForstnericLarussonCAG}.  

The only map in the above diagram which is not completely understood is the inclusion 
$\CMI_{\mathrm{nf}}(M,\R^n) \hra  \CMI(M,\R^n)$ of the space of nonflat conformal minimal immersions
into the space of all conformal minimal immersions. The authors and L\'opez 
showed in \cite[Theorem 7.1]{AlarconForstnericLopez2017JGEA}  that this inclusion 
induces a bijection of path components of the two spaces. In particular,  we have the following result.

\begin{corollary}
For any open Riemann surface $M$ the space $\CMI(M,\R^n)$ is path connected if $n>3$, 
whereas the set of path components of $\CMI(M,\R^3)$
is in bijective correspondence with the elements of the abelian group $(\Z_2)^l$ where $H_1(M;\Z)=\Z^l$. 
\end{corollary}

%%%%%%%%%%
%%%%%%%%%%
%%%%%%%%%%
%%%%%%%%%% Section: THE RIEMANN-HILBERT METHOD
%%%%%%%%%%
%%%%%%%%%%
%%%%%%%%%%

\section{\sc The Riemann-Hilbert method for minimal surfaces} 
\label{sec:RH}

The Riemann-Hilbert problem is a classical boundary value problem for holomorphic functions and maps.
The basic form of the problem was mentioned by Riemann in his dissertation in 1851. 
A brief history, references and a list of applications can be found in \cite[Sect.\ 3]{AlarconForstneric2015Abel}. 
In Sect.\ \ref{ss:RHcomplex} we describe the original complex analytic setting. 
In Sect.\ \ref{ss:RH} we state without proof a version of the Riemann-Hilbert problem for conformal minimal
immersions from bordered Riemann surfaces to $\R^n$. 
This is the basis for the construction of complete conformal  minimal surfaces in $\R^n$ bounded by 
Jordan curves and normalized by any given bordered Riemann surface 
(see Theorems \ref{th:I-complete} and \ref{th:Jordan}), the construction of proper complete 
conformal minimal immersions of such surfaces to minimally convex domains in $\R^n$ (see Sect.\ \ref{ss:proper}),
and the description of the minimal hull of a compact set in $\R^n$ by sequences of minimal discs
(see \cite{DrinovecForstneric2016TAMS} for $n=3$ and \cite{AlarconDrinovecForstnericLopez2018TAMS}
for $n>3$). Due to space limitations we shall not discuss minimal hulls in this survey.

%
%   SUBSECTION: THE RH PROBLEM IN COMPLEX ANALYSIS
%
\subsection{The Riemann-Hilbert problem in complex analysis}\label{ss:RHcomplex}

Let $X$ be a complex manifold. We are given a holomorphic map $f\colon \cd\to X$
(an {\em analytic disc} in $X$) and for each point $z\in \t=b\D$ a holomorphic map $g_z\colon \overline\d\to X$
with $g_z(0)=f(z)$ and depending continuously on $z\in \t$. Set $T_z=g_z(\t) \subset X$ and 
$S_z=g_z(\overline \d)\subset X$ for $z\in\t$. Fix a distance function $\dist$ on $X$. 
Given numbers $0<r<1$ and $\epsilon>0$, the {\em approximate Riemann-Hilbert problem} 
asks for a holomorphic map $F\colon \overline\d\to X$ satisfying the following conditions for some $r'\in [r,1)$:
\begin{itemize}
\item[\rm (a)] $\dist(F(z),T_z)<\epsilon$ for $z\in\t$,
\vspace{1mm}
\item[\rm (b)] $\dist(F(\rho z),S_z)<\epsilon$ for $z\in \t$ and $r'\le \rho\le 1$, and
\vspace{1mm}
\item[\rm (c)] $\dist(F(z),f(z))<\epsilon$ for $|z|\le r'$.
\end{itemize}
(The domain of $f$ may in fact be any bordered Riemann surface $M$,
but the domain of the maps $g_z$ is always the closed disc $\overline\d$.) This implies that 
\begin{itemize}
\item $F(\cd)$ lies in the $\epsilon$-neighborhood of the set
$\Sigma= f(\cd)\cup \bigcup_{z\in \t} g_z(\cd)$, and 
\vspace{1mm}
\item its boundary $F(\t)$ lies in the $\epsilon$-neighborhood of the torus 
$T=\bigcup_{z\in \t} g_z(\t)$.
\end{itemize}
This shows that the placement of the image curve $F(\cd)$ in $X$ is well controlled, a
very important point in most applications. (The exact problem, asking for $F$ satisfying 
$F(z)\in T_z$ for every $z\in b\D$, is only rarely solvable.)

For $X=\c^n$ the problem is solved as follows 
(see \cite{DrinovecForstneric2012IUMJ} for the details). The map 
\[
	\t\times \overline\d \ni (z,w) \mapsto g_z(w) - f(z) \in\c^n
\]
is continuous in $z$, holomorphic in $w$, and vanishes at $w=0$ for any $z\in \t$ since $g_z(0)=f(z)$. 
We can approximate it arbitrarily closely by a rational map 
\[
	G(z,w)= z^{-m} \sum_{j=1}^N A_j(z) w^j \in\c^n, 
\]
where the $A_j$'s are $\c^n$-valued holomorphic polynomials. Pick $k\in \n$
and set 
\begin{equation}\label{eq:F}
	F(z)= f(z)+ G(z,z^k)= f(z) + z^{k-m} \sum_{j=1}^N A_j(z) z^{k(j-1)}, \quad z\in \overline \d. 
\end{equation}
The pole at $z=0$ cancels if $k>m$, and one easily verifies that $F$ satisfies conditions (a)--(c) if the integer 
$k$ is chosen big enough. 

Consider now  the case when the domain of $f$ is a bordered Riemann surface $M$
and the target manifold $X$ is arbitrary. In most applications it suffices to solve the following restricted problem. 
Pick a pair of arcs $I_0,I_1\subset bM$ with $I_0\subset \mathring I_1$ and a smooth function $\chi\colon bM\to [0,1]$ 
such that $\chi=1$ on $I_0$ and $\chi=0$ on $bM\setminus I_1$. Set $\tilde g_z(w)=g_z(\chi(z)w)$ for $z\in bM$ and 
$w\in \overline\d$. Note that $\tilde g_z$ agrees with $g_z$ for $z\in I_0$ and is the constant disc $w\to f(z)$ for any point 
$z\in bM\setminus I_1$. Let $D\subset M$ be a smoothly bounded simply connected domain (a disc)
such that $I_1$ is a relatively open subset of $bD\cap bM$.
We define $\tilde g_z$ as the constant dics $f(z)$ for points $z\in bD \setminus I_1$ (this is consistent
with the previous choices). 
Let $\wt F:\overline D\to X$ be an approximate solution of the Riemann-Hilbert problem with the data $f|_{\overline D}$ 
and $\tilde g_z$, $z\in bD$. (Such $\wt F$ is found by reducing this local problem to the Euclidean case
via suitable Stein neighborhoods of the graphs of our maps.)
By choosing the integer $k$ in (\ref{eq:F}) big enough, $\wt F$ satisfies condition (a) for $z\in I_0$, 
it satisfies (b) for $z\in bD$, and is uniformly close to $f$ on $\overline D \setminus U$ where $U\subset \overline M$ 
is any given neighborhood of the arc $I_1$. Write $\overline M=A\cup B$ where $A,B\subset \overline M$ 
are smoothly bounded compact domains such that $A$ is the complement of a small neighborhood of the arc $I_1$, 
$B\subset \overline D$ contains a small neighborhood of $I_1$, we have that 
$\overline{A\setminus B}\cap \overline{B\setminus A}=\varnothing$, and $\wt F$ is uniformly close to $f$ on $C=A\cap B$. 
Next, we glue $f$ and $\wt F$ into a solution $F\colon M\to X$ by the method of {\em gluing holomorphic sprays}. 
An outline can be found in \cite[Sect.\ 3]{AlarconForstneric2015Abel}, and the method is fully 
explained in \cite[Chapter 5]{Forstneric2017E}. This method lies at the heart of proof of 
the Oka principle for maps from Stein manifolds to Oka manifolds (see \cite[Theorem 5.4.4]{Forstneric2017E}).

%
%   THE RH PROBLEM FOR NULL CURVES AND MINIMAL SURFACES
%
\subsection{The Riemann-Hilbert method for null curves and minimal surfaces}
\label{ss:RH}
The Riemann-Hilbert boundary value problem has been adapted to null holomorphic curves in $\C^n$ and 
conformal minimal surfaces in $\R^n$ for any $n\ge 3$ in the papers 
\cite{AlarconDrinovecForstnericLopez2015PLMS,AlarconForstneric2015MA}.
A special case for null curves in $\C^3$ was first obtained by the authors in \cite{AlarconForstneric2015MA} by
using the double sheeted spinor parametrization $\pi\colon\C^2_*\to\Agot^2_*$ of the null quadric, 
lifting to derivative maps from $\Agot_*^2$ to $\C^2_*$, applying the Riemann-Hilbert
method in $\C^2_*$ and then pushing the resulting maps down to $\Agot^2_*$. 
When replacing the disc by a bordered Riemann surface $M$, one must glue local solutions on small discs 
abutting $bM$ by using the method of gluing sprays as described in the previous section. 
Since the Riemann-Hilbert problem is not used directly for null curves but for their derivatives 
with values in the null quadric $\Agot_*$, one must also pay attention 
to the period vanishing conditions to ensure that the approximating maps 
integrate to null curves. The results from \cite{AlarconForstneric2015MA} were extended 
to any dimension $n\ge 3$ and were adapted to conformal minimal immersions by the authors  
with Drinovec Drnov\v sek and L\'opez \cite{AlarconDrinovecForstnericLopez2015PLMS}.
% (see \cite[Lemmas 3.1, 3.3 and Theorem 3.5]{AlarconDrinovecForstnericLopez2015PLMS}).
The following result is \cite[Theorem 3.6]{AlarconDrinovecForstnericLopez2015PLMS}.

%
%
%  RH PROBLEM FOR CONFORMAL MINIMAL IMMERSION IN C^n
%
%
\begin{theorem}[Riemann-Hilbert problem for conformal minimal surfaces in $\r^n$] 
\label{th:RHCMI}
Let $M$ be a compact bordered Riemann surface with boundary $bM\ne\varnothing$, and let $I_1,\ldots,I_k$ be  
pairwise disjoint compact arcs in $bM$ which are not connected components of $bM$. 
Let $r \colon bM \to \r_+$ be a continuous  nonnegative function supported on  $I:=\bigcup_{i=1}^k I_i$.
Also, let $\sigma \colon I \times\overline{\d}\to\c$ be a function of class $\Cscr^1$ such that
for every $\zeta\in I$ the function $\cd\ni \xi \mapsto \sigma(\zeta,\xi)$ is holomorphic on $\d$,
$\sigma(\zeta,0)=0$, and the partial derivative $\di \sigma/\di \xi$ is nowhere vanishing on $I \times \cd$.
Choose a thin annular neighborhood  $A\subset M$ of $bM$ and a smooth retraction $\rho\colon A\to bM$. 
For each $i=1,\ldots, k$ let $\bu_i,\bv_i\in \r^n$ be a pair of orthogonal vectors 
satisfying $|\bu_i|=|\bv_i|>0$. Given $X\in\CMI^1(M,\r^n)$ $(n\ge 3)$, consider the continuous map  
$\varkappa\colon bM \times\overline{\d}\to\r^n$ given by 
\begin{equation}\label{eq:varkappa}
	\varkappa(\zeta ,\xi)=\left\{\begin{array}{ll}
	X(\zeta ),  & \zeta \in bM\setminus I; \\ 
	X(\zeta ) + r(\zeta )\bigl( \Re \sigma(\zeta,\xi) \bu_i+ \Im \sigma(\zeta,\xi)  \bv_i \bigr), & 
	\zeta \in I_i,\; i\in\{1,\ldots,k\}.
	\end{array}\right.
\end{equation}
Given $\epsilon>0$ there exist an arbitrarily small open neighborhood $\Omega\subset M$ of   
$I$ and a  conformal minimal immersion $Y \in\CMI^1(M,\r^n)$ satisfying  the following conditions.
\begin{enumerate}[\it i)]
\item $\dist(Y(\zeta ),\varkappa(\zeta ,\t))<\epsilon$ for all $\zeta \in bM$.
\item $\dist(Y(\zeta ),\varkappa(\rho(\zeta ),\overline{\d}))<\epsilon$ for all $\zeta \in \Omega$.
\item $Y$ is $\epsilon$-close to $X$ in the $\Cscr^1$ norm on $M\setminus \Omega$.
\item $\Flux (Y)=\Flux (X)$.
\end{enumerate}
\end{theorem}

The proof will not be reproduced here due to its complexity and space limitations. 
Note that the boundary discs $\varkappa(\zeta ,\cdotp)$ $(\zeta\in bM)$ lie in affine $2$-planes.
A more precise result is available in dimension $n=3$; see 
\cite[Theorem 3.2]{AlarconDrinovecForstnericLopez2018TAMS}. In that case the map
$\varkappa$ \eqref{eq:varkappa} may be chosen of the form
\begin{equation}\label{eq:varkappa2}
	\varkappa(\zeta,\xi)=X(\zeta) + \alpha\bigl(\zeta,r(\zeta) \,\xi\bigr),
\end{equation}
where $\alpha \colon I \times\overline{\d}\to\r^3$ is a map of class $\Cscr^1$ such that
for every $\zeta\in I$ the map $\cd \ni \xi \mapsto \alpha(\zeta,\xi)\in\r^3$ is a conformal minimal 
immersion with $\alpha(\zeta,0)=0$ and we take $\alpha\bigl(\zeta,r(\zeta) \,\xi\bigr)=0$ for $\zeta\in bM\setminus I$.
The advantage is that the conformal minimal discs $\cd\ni \xi\mapsto \alpha\bigl(\zeta,r(\zeta) \,\xi\bigr)\in\r^3$
are arbitrary and not necessarily flat as before.

The Riemann-Hilbert method has also been adapted by the authors and L\'opez \cite{AlarconForstnericLopez2017CM} 
to holomorphic Legendrian curves in the standard complex contact structure
on Euclidean spaces $\C^{2n+1}$.

%
% Subsection: ON THE CALABI-YAU PROBLEM
%

\subsection{On the Calabi-Yau problem}
\label{ss:CY}

The Calabi-Yau problem for minimal surfaces originated in the following conjecture of Calabi from 1965.

\begin{conjecture}[\text{Calabi \cite[p.\ 170]{Calabi1965Conjecture}}]\label{conj:Calabi}
A complete minimal hypersurface in $\r^n$ for $n\ge 3$ is unbounded. Even more, its projection
to every $(n-2)$-dimensional affine subspace of $\r^n$ is unbounded.
\end{conjecture}

Nothing seems  known about this problem for $n\ge 4$. 
For $n=3$, the latter assertion in Calabi's conjecture was refuted by Jorge and Xavier \cite{JorgeXavier1980AM} in 1980, 
and the former by Nadirashvili \cite{Nadirashvili1996IM} in 1996. In both cases the counterexample is normalized by 
the disc,   and the proof combines the L\'opez-Ros deformation for minimal surfaces \cite{LopezRos1991JDG} 
with an inventive use of the Runge approximation theorem for holomorphic functions. 
In his  {\em 2000 Millennium Lecture} \cite{Yau2000}, 
Yau revisited Calabi's conjectures and proposed several questions concerning the topology, complex structure, 
and asymptotic behavior of complete bounded minimal surfaces in $\r^3$. Ferrer, Mart\'in, and Meeks 
\cite{FerrerMartinMeeks2012AM} proved in 2012 that there is no restriction on their topological type; 
controlling the complex structure is a much more challenging task.
The second part of Conjecture \ref{conj:Calabi} was settled in 2012 by
Alarc{\'o}n, Fern{\'a}ndez, and L{\'o}pez \cite{AlarconFernandezLopez2012CMH} who proved the following result. 
(A special case was obtained beforehand in \cite{AlarconFernandez2011DGA}.)

\begin{theorem}\label{th:AFeL}
Given an open Riemann surface $M$ and a nonconstant  harmonic function $h\colon M\to\R$, 
there is a complete conformal minimal immersion $X\colon M\to\r^3$ whose third coordinate function equals $h$.
In particular, $M$ admits a complete nonflat conformal minimal immersion $M\to\r^3$ with a bounded component
function if and only if there exists a bounded nonconstant harmonic function $M\to\r$.
\end{theorem}

In the subsequent paper \cite{AlarconFernandezLopez2013CVPDE} of the same authors this result was extended to 
conformal minimal surfaces in $\R^n$ for $n>3$,  where now $n-2$ of the coordinate functions can be prescribed.

On the other hand, open Riemann surfaces normalizing complete bounded minimal surfaces in $\r^3$ are far from classified. 
By introducing the Riemann-Hilbert method into the picture, the authors with Drinovec Drnov\v sek and L\'opez 
\cite{AlarconDrinovecForstnericLopez2015PLMS} proved the following result to the effect 
that every bordered Riemann surface normalizes a complete bounded minimal surface with Jordan boundary. 

%
%   COMPLETE CMIS WITH JORDAN BOUNDARIES
%
\begin{theorem}\label{th:Jordan}
Let $M$ be a compact bordered Riemann surface. Every conformal minimal immersion $X\colon M\to\r^n$ $(n\ge 3)$ of class 
$\Cscr^1(M)$ may be approximated uniformly on $M$ by continuous maps $\wt X\colon M\to\r^n$ such that 
$\wt X|_{bM}\colon bM\hra\r^n$ is a topological embedding and $\wt X|_{\mathring M}\colon \mathring M\to\r^n$ is a 
complete conformal minimal immersion. If $n\ge 5$  there are embeddings $\wt X\colon M\hra\r^n$ with these properties.
The flux map of $\wt X$ can also be prescribed.
\end{theorem}

Theorem \ref{th:Jordan} follows by an obvious inductive application of the following approximation result 
(see \cite[Lemma 4.1]{AlarconDrinovecForstnericLopez2015PLMS}) together with Theorem \ref{th:desingBRS} 
and a transversality argument which deals with the injectivity on the boundary.

%
%   MAIN LEMMA: 
%
\begin{lemma}\label{lem:Jordan}
In the assumptions of Theorem \ref{th:Jordan}, given $p_0\in \mathring M$ and $\lambda>0$, $X$ may be approximated 
arbitrarily closely in the $\Cscr^0(M)$ topology by a conformal minimal immersion $Y\colon M\to\r^n$ of class $\Cscr^1(M)$ 
such that  $\dist_Y(p_0,bM)>\lambda$.
\end{lemma}

In turn, Lemma \ref{lem:Jordan} follows from the maximum principle, the divergence of the sequence $d_j=d_{j-1}+\frac1{j}$ 
$(d_0>0)$, the convergence of the sequence $\delta_j=\sqrt{\delta_{j-1}^2+\frac1{j}^2}$ $(\delta_0>0)$, and an inductive 
application of the following result (see \cite[Lemma 4.2]{AlarconDrinovecForstnericLopez2015PLMS}). 
% We refer to \cite[proof of Lemma 4.1]{AlarconDrinovecForstnericLopez2015PLMS} for complete details.

%
%   THE MAIN LEMMA 2
%
\begin{lemma}\label{lem:Jordan2}
In the assumptions of Theorem \ref{th:Jordan}, let $\Ygot\colon bM\to\r^n$ be  a smooth map and choose $\delta>0$ 
such that $|X(p)-\Ygot(p)|<\delta$ for all $p\in bM$. Also let $p_0\in \mathring M$ and choose $d>0$ such that $0<d<\dist_X(p_0,bM)$.
For any $\eta>0$ the map $X$ may be approximated uniformly on compacts in $\mathring M$ by conformal minimal immersions 
$Y\colon M\to\r^n$ of class $\Cscr^1(M)$ satisfying $\dist_Y(p_0,bM)> d+\eta$ and $|Y(p)-\Ygot(p)|<\sqrt{\delta^2+\eta^2}$ for all $p\in bM$.
\end{lemma}
\begin{proof}[Sketch of proof of Lemma \ref{lem:Jordan2}]
We assume that $M$ is a smoothly bounded compact domain in an open Riemann surface $\wt M$ and, for simplicity of exposition, 
that $bM$ is connected. Choose a smoothly bounded compact domain $K\subset\mathring M$ which is 
a strong deformation retract of $M$ and such that $\dist_X(p_0,bK)>d$.
We assume without loss of generality that  $X-\Ygot\neq 0$ on $bM$. 
Given  $\epsilon>0$ we look for a conformal minimal immersion 
$Y\colon M\to\r^n$ with $|Y-X|<\epsilon$ on $K$ and satisfying the lemma. 

Fix a number $\epsilon_0>0$ which will be specified later. By continuity of $X$ and $\Ygot$, $bM$ splits into $l\geq 3$ compact arcs 
$\alpha_k$, $k\in\z_l$, lying end to end and such that for all pairs of points $p,q$ in $\alpha_k$ we have
$|\Ygot(p)-\Ygot(q)|<\epsilon_0$, $|X(p)-\Ygot(q)|<\delta$, and $|X(p)-X(q)|<\epsilon_0$. 
Denote by $p_k$ the only point in $\alpha_k\cap\alpha_{k+1}$ 
and by $\pi_k\colon\r^n\to {\rm span}\bigl\{X(p_k)-\Ygot(p_k)\bigl\} \, \subset\r^n$
the orthogonal projection onto the affine real line ${\rm span}\{X(p_k)-\Ygot(p_k)\}$.
The first step consists of perturbing $X$ near the points $\{p_k\colon k\in \z_l\}$ in order to find a
conformal minimal immersion $X_0\colon M\to\r^n$ of class $\Cscr^1(M)$  which is close to $X$ in the $\Cscr^1(K)$ norm 
and such that the distance between $p_0$ and $\{p_k\colon k\in \z_l\}$ in the induced metric $X_0^*(ds^2)$
is large in a suitable way. To be precise, we ask $X_0$ to keep satisfying 
\begin{itemize}
\item[\rm (i)] $|X_0(p)-\Ygot(q)|<\delta$ and $|X_0(p)-X(q)|<\epsilon_0$ for all $\{p,q\}\in \alpha_k$, $k\in \z_l$,
\end{itemize}
and to meet also the following condition:
\begin{itemize}
\item[\rm (ii)] For each $k\in \z_l$ there is a small open neighborhood $U_k$ of $p_k$ in $M$, with $\overline U_k\cap K=\varnothing$, 
enjoying the following condition: if $\gamma\subset M$ is an arc with initial point in $K$ and final point in $\overline U_k$, 
and if $\{J_a\}_{a\in \z_l}$ is a partition of $\gamma$ by Borel measurable subsets, then
$\sum_{a\in \z_l} \length \, \pi_a(X_0(J_a)) >\eta$.
\end{itemize}
To find such $X_0$, we take a family of pairwise disjoint Jordan arcs $\{\gamma_k\subset \wt M\colon k\in \z_l\}$ such that 
each $\gamma_k$ contains $p_k$ as an endpoint, is attached transversely to $M$ at $p_k$ 
and is otherwise disjoint from $M$, and the set $S:=M\cup\bigcup_{k\in \z_l} \gamma_k \subset\wt M$
is admissible (see Definition \ref{def:admissible}). We then extend $X$ to a generalized conformal minimal immersion 
$(X,f\theta)\in\GCMI(S,\r^n)$ so that the following analogues to {\rm (i)} and {\rm (ii)} hold.
\begin{itemize}
\item $|X(x)-\Ygot(q)|<\delta$  and $|X(x)-X(q)|<\epsilon_0$ for all 
$(x,q)\in (\gamma_{k-1}\cup \alpha_k\cup \gamma_k)\times \alpha_k$.

\item If $\{J_a\}_{a\in \z_l}$ is a partition of $\gamma_k$ by Borel measurable subsets, then 
$\sum_{a\in \z_l} \length\, \pi_a(X(J_a)) > 2\eta.$
\end{itemize}
This means that $X$ is chosen on each arc $\gamma_k$ to be highly oscillating in the direction 
of $F(p_a)-\Ygot(p_a)$ for all $a\in\z_l$, but with very small extrinsic diameter. 
Now, apply Theorem \ref{th:ALL} to approximate $(X,f\theta)$ uniformly on $S$ by a 
conformal minimal immersion $X\colon \wt M\to\r^n$ of class $\Cscr^1(M)$; let us keep denoting it by $X$. 
Let $q_k$ denote the endpoint of $\gamma_k$ different from $p_k$.  If the approximation is close enough then 
\cite[Theorem 2.3]{ForstnericWold2009JMPA} provides a smooth diffeomorphism $\phi\colon M\to\phi (M)$ satisfying:
\begin{itemize}
\item $\phi \colon  \mathring M \to  \phi(\mathring M)$ is a biholomorphism,
\item $\phi$ is as close as desired to the identity in the $\Cscr^1$ norm on the complement in $M$ 
of a small neighborhood of $\{p_k\colon k\in\z_l\}$, and 
\item $\phi(p_k) = q_k\in b\,\phi(M)$ and $\phi$ maps a suitably chosen neighborhood of $\{p_k\colon k\in\z_l\}$ in $M$ 
to a small neighborhood of $\bigcup_{k\in\z_l}\gamma_k$ in $\wt M$.
\end{itemize}
Thus, doing things in the right way, when composing $\phi$ with $X$ we are merging the arcs $X(\gamma_k)$ into $X(M)$ 
without modifying $M$ itself. It follows that the $\Cscr^1(M)$ conformal minimal immersion $X_0=X\circ\phi\colon M\to\r^n$ 
satisfies conditions {\rm (i)} and {\rm (ii)}.

We may assume that the sets $\overline U_k$, $k\in \z_l$, are simply connected, smoothly bounded, and pairwise disjoint.
Roughly speaking, $X_0$ meets the requirements in the lemma, except that $\dist_{X_0}(p_0,p)>d+\delta$ only holds 
for the points $p$ in $bM$ which lie in a $\overline U_k$. To conclude the proof we perturb $X_0$ outside $\bigcup_{k\in \z_l} U_k$, 
preserving what has already been achieved. At this point  the Riemann-Hilbert method is invoked.
Fix $\epsilon_1>0$ to be specified later, choose an annular neighborhood $A\subset M\setminus K$ of $bM$ and a smooth 
retraction $\rho\colon A\to bM$. By {\rm (i)} there is a family of pairwise disjoint, smoothly bounded closed discs
$\overline D_k$ in $M\setminus K$, $k\in\z_l$, satisfying $\bigcup_{k\in \z_l} \overline D_k\subset A$, $\overline D_k\cap bM$ is a 
compact connected Jordan arc in $\alpha_k\setminus\{p_{k-1},p_k\}$ with an endpoint  in $U_{k-1}$ and the other endpoint in $U_k$, 
and the following conditions: 
\begin{itemize}
\item[\rm (iii)] $|X_0(p)-\Ygot(q)|<\delta$\ for all $(p,q)\in \overline D_k\times\alpha_k$, and 
\item[\rm (iv)] $\rho(\overline D_k)\subset \alpha_k\setminus\{p_{k-1},p_k\}$ and $|X_0(\rho(x))-X_0(x)|<\epsilon_1$ for all  
$x\in\overline D_k$, $k\in \z_l$.
\end{itemize}
For each $k\in \z_l$ we choose a pair of compact Jordan arcs 
$\beta_k\Subset I_k\Subset \overline D_k\cap\alpha_k$ with an endpoint in $U_{k-1}$ and the other endpoint in 
$U_k$, and a pair of vectors $\bu_k$, $\bv_k\in \r^n$, such that 
$|\bu_k|=1 = |\bv_k|$ and $\bu_k$, $\bv_k$, and $X(p_k)-\Ygot(p_k)$ are pairwise orthogonal.
We then choose a continuous function $\mu\colon bM\to\r_+$  such that
\[
	0\leq\mu\leq\eta,\qquad \text{$\mu=\eta$\ \ on  $\bigcup_{k\in \z_l} \beta_k$,\qquad 
	\text{$\mu=0$\; on $bM\setminus \bigcup_{k\in \z_l} I_k$}}.
\]
Consider the continuous map  $\varkappa\colon bM \times\overline{\d}\to\r^n$ given by 
\[
	\varkappa(x,\xi)=\left\{\begin{array}{ll}
	X_0(x), & x\in bM\setminus \bigcup_{k\in \z_l} I_k; \\
	X_0(x) + \mu(x)(\Re \xi \bu_k + \Im \xi  \bv_k),  & x\in I_k,\; k\in\z_l.
	\end{array}\right.
\]
These are the boundary data of a Riemann-Hilbert problem on $X_0$ for which the boundary disc 
$\varkappa(x,\cdotp)$ at each point $x\in \beta_k$ is a round planar disc of radius $\eta$ that is orthogonal to $X(p_k)-\Ygot(p_k)$. 
Theorem  \ref{th:RHCMI} provides for every $k\in \z_l$ an arbitrarily small open neighborhood $\Omega_k\subset \overline D_k$ of   
$I_k$ in $M$ and a conformal minimal immersion $Y\colon M\to\r^n$ of class $\Cscr^1(M)$ satisfying the following conditions.
\begin{itemize}
\item[\rm (v)] $\dist(Y(x),\varkappa(x,\t))<\epsilon_1$ for all $x\in bM$.
\item[\rm (vi)] $\dist(Y(x),\varkappa(\rho(x),\overline{\d}))<\epsilon_1$ for all 
$x\in \Omega:=\bigcup_{k\in\z_l}\Omega_k$.
\item[\rm (vii)] $Y$ is $\epsilon_1$-close to $X_0$ in the $\Cscr^1$ norm on $M\setminus \Omega$.
\end{itemize}

Using conditions {\rm (i)}--{\rm (vii)} and Pythagoras' theorem, it is not hard to see that $Y$ satisfies the conclusion 
of the lemma provided that $\epsilon_0>0$ and $\epsilon_1>0$ are chosen sufficiently small. 
Very briefly, by {\rm (vi)}, {\rm (iv)}, and the definition of $\varkappa$ we have that $\pi_k\circ Y$ is $2\epsilon_1$-close 
to $\pi_k\circ X_0$ in the $\Cscr^0(\Omega_k)$ topology for all $k\in\z_l$, and so a weaker version of condition {\rm (ii)} is 
preserved by the second deformation procedure. This enables to ensure that $\dist_Y(p_0,p)>d+\eta$ for all 
$p\in bM\cap\bigcup_{k\in\z_l} \overline U_k$. Taking into account {\rm (i)}, {\rm (v)}, and that $\mu=\eta$ on  
$\bigcup_{k\in \z_l} \beta_k$, we infer the same inequality for all points $p\in  bM\setminus \bigcup_{k\in\z_l} \overline U_k$. 
On the other hand, {\rm (i)}, {\rm (v)}, {\rm (vii)}, the facts that  $\mu=\eta$ on  $\bigcup_{k\in \z_l} \beta_k$ and that $\bu_k$, $\bv_k$, 
and $X(p_k)-\Ygot(p_k)$ are pairwise orthogonal, and Pythagoras' theorem guarantee that $|Y(p)-\Ygot(p)|<\sqrt{\delta^2+\eta^2}$ 
for all $p\in bM$.
\end{proof}

Another recent application of the Riemann-Hilbert method is the construction of complete minimal surfaces densely lying 
in arbitrary domains of $\r^n$. The following result is due to Alarc\'on and Castro-Infantes \cite{AlarconCastro-Infantes2018GT}.

\begin{theorem}\label{th:dense}
Let $D\subset\r^n$ $(n\ge 3)$ be an open connected set. Every bordered Riemann surface $M$ admits a complete conformal 
minimal immersion $X\colon M\to \r^n$ such that $X(M)$ is a dense subset of $D$. If $n\ge 5$ then $X$ may be chosen injective.
\end{theorem}

%%%%%%%%%%
%%%%%%%%%%
%%%%%%%%%% Subsection: PROPER CMIS IN DOMAINS IN RN
%%%%%%%%%%
%%%%%%%%%%
%%%%%%%%%%

\subsection{Proper conformal minimal surfaces in minimally  convex domains} % in $\R^n$}
\label{ss:proper}

A major problem in minimal surface theory is to understand
which domains in $\r^n$ admit complete properly immersed minimal surfaces, 
and how the geometry of the domain influences the conformal properties of such surfaces
(see \cite[Section 3]{MeeksPerez2004SDG} for a background  on this topic). 
In dimension $n=3$ this subject is connected with the Calabi-Yau problem. 
In view of Nadirashvili's example \cite{Nadirashvili1996IM}, Yau \cite{Yau2000} asked whether there 
exist complete minimal surfaces properly immersed in the ball of $\r^3$. An affirmative answer for any either convex or  
bounded and smoothly bounded domain in $\r^3$ was given by Mart\'in and Morales in 
\cite{MartinMorales2004TAMS,MartinMorales2005DMJ,MartinMorales2006CMH}. In the opposite direction, 
Mart{\'i}n, Meeks, and Nadirashvili \cite{MartinMeeksNadirashvili2007AJM} gave examples of bounded domains in $\r^3$ 
which do not admit any complete proper minimal surfaces of finite topology.

If $M$ is a bordered Riemann surface for which one is able to construct in a standard 
inductive way a proper conformal minimal immersion  into a given domain $D\subset\r^n$, then one can also 
construct a {\em complete} proper one by the procedure in Lemma \ref{lem:Jordan} which enlarges the intrinsic boundary 
distance within the surface as much as desired by an arbitrarily small displacement of the surface in $D$. 
Hence it suffices to focus on the existence of proper conformal minimal immersions.
Recent examples by Alarc\'on et al.\ \cite{AlarconDrinovecForstnericLopez2018TAMS} 
show that some geometric assumptions on the domain 
are necessary to obtain positive results. Indeed, there is a bounded simply connected domain $D\subset\r^3$ 
carrying no proper conformal minimal disc $\d\to D$ passing through a certain point in $D$, and 
a bounded domain $D\subset\r^3$ admitting no proper minimal surfaces with finite topology and a single end 
(see \cite[Examples 1.13 and 1.14]{AlarconDrinovecForstnericLopez2018TAMS}). 
Much earlier, Dor \cite{Dor1996MZ} found a bounded domain $\Omega\subset\c^m$ for any $m\ge 2$ which 
does not admit any proper holomorphic discs $\d\to\Omega$. 
It remains an open problem whether there is a domain $D\subset \R^n$ for $n>3$ without any proper minimal discs. 

A $\Cscr^2$ function $\rho\colon D\to \r$ on a domain $D \subset\r^n$ is said to be {\em strongly minimal pluri\-subharmonic} 
if $\tr_L \Hess_\rho(x)>0$ for every affine $2$-dimensional linear subspace $L$ and every point $x\in D\cap L$;
equivalently, if $\lambda_1(x) + \lambda_2(x) > 0$ for all $x\in D$ where $\lambda_1(x)\le \lambda_2(x)$ 
are the two smallest eigenvalues of the Hessian $\Hess_\rho(x )$. The domain $D$ is said to be {\em minimally convex} 
if it admits a strongly minimal plurisubharmonic exhaustion function
(see \cite[\textsection 2]{AlarconDrinovecForstnericLopez2018TAMS}). 
The following result was obtained as an application of Theorem \ref{th:RHCMI} (the Riemann-Hilbert method) 
with functions $\varkappa$ of the form \eqref{eq:varkappa2}; see
\cite[Theorem 1.1 and Remark 3.8]{AlarconDrinovecForstnericLopez2018TAMS}.

%
%   MINIMAL SURFACES IN MINIMALLY CONVEX DOMAINS
%
\begin{theorem}\label{th:ADFL}
Let $D\subset\r^n$ for $n\ge 3$ be a minimally convex domain and $M$ be a compact 
bordered Riemann surface. Every conformal minimal immersion $M\to D$ may be approximated uniformly on compacts in 
$\mathring M=M\setminus bM$ by proper (and complete if so desired) conformal minimal immersions 
(embeddings if $n\ge 5$) $\mathring M\to D$.
\end{theorem}

In particular, every {\em mean-convex} domain in $\r^3$ admits complete proper minimal surfaces normalized by any bordered 
Riemann surface. We refer to \cite{AlarconDrinovecForstnericLopez2015PLMS,AlarconDrinovecForstnericLopez2018TAMS} 
for more precise results including infinite topologies, control of the flux, and continuous extendibility up to the boundary.

Since the complement of an embedded minimal surface in $\R^3$ is a minimally convex domain
\cite[Corollary 1.3]{AlarconDrinovecForstnericLopez2018TAMS}, we have the following corollary
to Theorem \ref{th:ADFL}.

\begin{corollary}
Let $S\subset \R^3$ be a properly embedded minimal surface and let $D$ be a connected component
of $\R^3\setminus S$. Every bordered Riemann surface admits a proper conformal minimal 
immersion into $D$.
\end{corollary}

\begin{proof}[Sketch of proof of Theorem \ref{th:ADFL}]
We discuss the case $n=3$. Let $\rho\colon D\to\R_+$ be a smooth strongly minimally plurisubharmonic
Morse exhaustion function with the (discrete) critical locus $P$. 
We extend $\rho$ to a function on the tube $D\times\imath \R^3\subset\C^3$
which is independent of the imaginary coordinates. A simple calculation shows that $\rho$ is strongly minimally plurisubharmonic
if and only if the Levi form of the extended function at any point of $D\times\imath \R^3$ is positive on every null vector $w\in \Agot_*$ 
(see \cite[Lemma 4.3]{DrinovecForstneric2016TAMS}; such functions are called {\em strongly null plurisubharmonic}).
By using this fact, it is not hard to see (cf.\ \cite[Lemma 3.1]{AlarconDrinovecForstnericLopez2018TAMS})
that  for any compact set $L \subset D\setminus P$ there is a constant $c=c_L>0$ and families of embedded 
null holomorphic discs $\sigma_x^j = \alpha_x^j + \imath  \beta_x^j \colon \cd\to \c^3$  $(x\in L,\ j=1,2)$,
depending locally $\Cscr^1$ smoothly on the point $x\in L$ and satisfying the following conditions:
\begin{itemize}
\item[\rm (a)]   $\sigma_x^j(0)=0$, 
\item[\rm (b)]   $\{x + \alpha_x^j(\xi) : \xi\in \cd\} \subset D$, and 
\item[\rm (c)]   the function $\cd\ni \xi\mapsto \rho\bigl(x+ \alpha_x^j(\xi)\bigr)$ is strongly convex and satisfies
$\rho\bigl(x + \alpha_x^j(\xi)\bigr) \ge \rho(x)+c |\xi|^2$ for $\xi\in \cd$.
\end{itemize}
Indeed, it suffices to choose the null discs $\sigma_x^j$ in the quadratic complex hypersurface
\[
	\Sigma_{x}= \biggl\{w=(w_1,w_2,w_3)\in\C^3 : \sum_{j=1}^3 \frac{\di \rho}{\di x_j}(x)  w_j  
	+   \sum_{j,k=1}^3  \frac{\partial^2 \rho}{\partial z_j \partial\bar z_k}(x) w_j w_k=0 \biggr\}.
\]
The tangent plane $T_0\Sigma_x\subset\C^3$ contains precisely two null directions which leads to two families
of null discs as above. The restriction of $\rho$ to the affine hypersurface $x+\Sigma_{x}$ has 
the Taylor expansion 
\[
	\rho(x +w)=\rho(x)+\Lcal_\rho(x;w)+o(|w|^2), 
\]
where $\Lcal_\rho$ denotes the Levi form of $\rho$. Since the Levi form $\Lcal_\rho(x;w)$ 
is positive on null vectors $w\in \Agot_*$, we get the estimates in condition (c).

Using the conformal minimal discs $\alpha_x^j$  as above and a conformal minimal 
immersion $X\colon M\to D$, we consider the 
Riemann-Hilbert problem in Theorem \ref{th:RHCMI}, but with the function $\varkappa$ 
of the form \eqref{eq:varkappa2}. This shows that for any compact set $L\subset D\setminus P$
there are constants $\epsilon_0>0$ and $C_0>0$ such that the following holds.
Let $M$ be a compact bordered Riemann surface and $X\colon M\to D$ be a conformal minimal immersion 
of class $\Cscr^1(M)$ with $X(bM)\subset L$. 
Given a continuous function $\epsilon \colon bM\to [0,\epsilon_0]$ supported on 
the set $J=\{\zeta\in bM: X(\zeta) \in L\}$, an open set $U\subset M$ containing $\supp(\epsilon)$ in 
its relative interior, and a constant $\delta>0$, there exists a conformal minimal immersion 
$Y\colon M\to D$ satisfying the following conditions.
\begin{itemize}
\item[\rm (i)] $|\rho(Y(\zeta)) - \rho(X(\zeta)) -\epsilon(\zeta)| <\delta$ for every $\zeta\in bM$.          % lifting
\item[\rm (ii)] $\rho(Y(\zeta))\ge \rho(X(\zeta)) -\delta$ for every $\zeta\in M$.                                      % not dropping much
\item[\rm (iii)] $Y$ is $\delta$-close to $X$ in the $\Cscr^1$ norm in $M\setminus U$.      		      % approximation
\item[\rm (iv)] $Y$ is $C_0 \sqrt{\epsilon_0}$-close to $X$ in the $\Cscr^0$ norm in $M$.
\end{itemize} 
The theorem is proved by a inductive application of this procedure, together with a well known method of avoiding critical points of $\rho$.
By using also Lemma  \ref{lem:Jordan} we can obtain complete proper conformal minimal immersions $\mathring M\to D$.
\end{proof}

There is a variety of results and open questions in the literature as to which domains in $\R^3$ may or may not
contain minimal surfaces (possibly with additional properties) 
which are proper in $\R^3$. One of the main examples is the theorem  
of Hoffman and Meeks from 1990 \cite{HoffmanMeeks1990IM} to the effect that a properly immersed minimal surface 
$M\subset \R^3$ is never contained in a half space unless $M$ is a plane. Since minimally convex domains
are not necessarily convex and they may be quite big, it is a natural question whether they may contain
nonflat minimal surfaces which are proper in $\R^3$. Although we do not know a definitive answer to this question, 
we have the following rigidity result \cite[Theorem 1.16]{AlarconDrinovecForstnericLopez2018TAMS}
for properly immersed minimal surfaces of finite total curvature in $\R^3$ lying in minimally convex domains.

%
%   Maximal minimally convex domains
%
\begin{theorem} \label{th:FTC}
Let $S$ be a complete connected properly immersed minimal surface with finite total curvature
in $\r^3$, possibly with (compact) boundary, and let $D\subset\r^3$ be a connected 
minimally convex domain containing $S$. If $S$ is not a plane then $D=\r^3$. If $S$ is a plane
then $D$ is a slab (a domain bounded by two parallel planes), a halfspace, or $\r^3$.
\end{theorem}

%%%%%%%%%%
%%%%%%%%%%
%%%%%%%%%%
%%%%%%%%%%   THANKS
%%%%%%%%%%
%%%%%%%%%%

\subsection*{Acknowledgements}
A.\ Alarc\'on is partially supported by the MINECO/FEDER grants no.\ MTM2014-52368-P and MTM2017-89677-P, Spain.
F.\ Forstneri\v c is partially  supported  by the research program P1-0291 and the grant J1-7256 from ARRS, Republic of Slovenia.

We wish to thank Barbara Drinovec Drnov\v sek and Francisco J.\ L\'opez for the 
collaboration on the subject matter of this survey, and Finnur L\'arusson (an Associate Editor) and George
Willis (Editor-in-Chief) for their kind invitation to write this survey for the Journal of the Australian Mathematical Society.

Finally, we sincerely thank an anonymous referee for pointing out several misprints and for 
the remarks which helped us to improve the presentation.

%%%%%%%%%%
%%%%%%%%%%
%%%%%%%%%%
%%%%%%%%%%   THE BIBLIOGRAPHY
%%%%%%%%%%
%%%%%%%%%%

%\renewcommand{\refname}{\sc References}
%{\bibliographystyle{abbrv} \bibliography{bibAF}}

\def\cprime{$'$}

%%%%%%%%%%
%%%%%%%%%%
%%%%%%%%%%
%%%%%%%%%%   AFFILIATIONS
%%%%%%%%%%
%%%%%%%%%%

\vspace{3mm}

\noindent Antonio Alarc\'{o}n \\
\noindent Departamento de Geometr\'{\i}a y Topolog\'{\i}a e Instituto de Matem\'aticas (IEMath-GR), Universidad de Granada, Campus de Fuentenueva s/n, E--18071 Granada, Spain. \\
\noindent  e-mail: {\tt alarcon@ugr.es}

\vspace{3mm}

\noindent Franc Forstneri\v c \\
\noindent Faculty of Mathematics and Physics, University of Ljubljana, and Institute
of Mathematics, Physics and Mechanics, Jadranska 19, SI--1000 Ljubljana, Slovenia. \\
\noindent e-mail: {\tt franc.forstneric@fmf.uni-lj.si}

\end{document}